\documentclass[10pt, reqno]{amsart}
\usepackage{amsmath, amssymb, latexsym, enumerate,  graphicx, MnSymbol}
\usepackage{amscd,mathrsfs,epic,empheq,float}
\usepackage[all]{xy}
 \usepackage{pdfsync}
 \usepackage{todonotes}
\usepackage{tikz}
\usetikzlibrary{arrows}
\usepackage{graphicx}
\usepackage[vcentermath]{youngtab}
\usetikzlibrary{decorations.markings}
\usetikzlibrary{fadings}
\usepackage{array}
\usepackage{multirow}
\usepackage{stmaryrd}
\usepackage[latin1]{inputenc}

\theoremstyle{definition}
\newtheorem{thm}{Theorem}[section]
\newtheorem{prop}[thm]{Proposition}
\newtheorem{cor}[thm]{Corollary}
\newtheorem{lem}{Lemma}
\newtheorem{defn}[thm]{Definition}

\newtheorem{rem}{Remark}
\newtheorem{ex}[thm]{Example}

\newtheorem*{first-step}{First Step}
\newtheorem*{third-step}{Third Step}
\newtheorem*{second-step}{Second Step}

\newtheorem*{question}{Question}

\newtheorem*{main idea}{Main idea}

\newtheorem{cl}{Claim}

\newcommand{\op}{\operatorname}

\newcommand{\tightoverset}[2]{%
  \mathop{#2}\limits^{\vbox to -.5ex{\kern-0.75ex\hbox{$#1$}\vss}}}
  
\newcommand{\tightunderset}[2]{%
  \mathop{#2}\limits_{\vbox to -.5ex{\kern-0.75ex\hbox{$#1$}\vss}}}

\newcommand{\bs}[1]{\boldsymbol{#1}}
\newcommand\smvee{\raise0.9ex\hbox{$\scriptscriptstyle\vee$}}

\newcommand{\J}{{\rm J}}
\newcommand{\Y}{{\rm Y}}
\newcommand{\C}{{\rm C}}
\newcommand{\T}{{\rm T}}
\newcommand{\E}{{\rm E}}
\newcommand{\R}{{\rm R}}
\newcommand{\U}{{\rm U}}
\newcommand{\mm}{{\rm H}}
\newcommand{\B}{{\rm B}}
\newcommand{\s}{{\rm S}}
\newcommand{\I}{{\rm I}}

\newcommand{\N}{{\rm N}}

\newcommand{\e}{\varepsilon}
\newcommand{\p}{{\rm P}}
\newcommand{\A}{{\rm A}}
\newcommand{\F}{{\rm F}}
\newcommand{\V}{{\rm V}}
\newcommand{\li}{{\rm L}}
\newcommand{\X}{{\rm X}}

\newcommand{\K}{{\rm K}}
\newcommand{\D}{{\rm D}}
\newcommand{\Z}{{\rm Z}}
\newcommand{\W}{{\rm W}}
\newcommand{\G}{{\rm G}}
\newcommand{\<}{\langle}
\newcommand{\rr}{\rangle}

\newcommand{\nn}{\nonumber}

\newcommand{\gr}[1]{\overline{#1}}

\newcommand{\ov}[1]{\overline{#1}}

\date{}                                           
\newcount\cols
{\catcode`,=\active\catcode`|=\active
 \gdef\Young#1{\hbox{$\vcenter
 {\mathcode`,="8000\mathcode`|="8000
  \def,{\global\advance\cols by 1 &}%
  \def|{\cr
        \multispan{\the\cols}\hrulefill\cr
        &\global\cols=2 }%
  \offinterlineskip\everycr{}\tabskip=0pt
  \dimen0=\ht\strutbox \advance\dimen0 by \dp\strutbox
  \halign
   {\vrule height \ht\strutbox depth \dp\strutbox##
    &&\hbox to \dimen0{\hss$##$\hss}\vrule\cr
    \noalign{\hrule}&\global\cols=2 #1\crcr
    \multispan{\the\cols}\hrulefill\cr%
   }
 }$}}
\gdef\Skew(#1:#2){\hbox{$\vcenter
{\mathcode`,="8000\mathcode`|="8000
  \dimen0=\ht\strutbox \advance\dimen0 by \dp\strutbox
  \def\boxbeg{\vbox
    \bgroup\hrule\kern-0.4pt\hbox to\dimen0\bgroup\strut\vrule\hss$}%
  \def\boxend{$\hss\egroup\hrule\egroup}%
  \def,{\boxend\boxbeg}%
  \def|##1:{\boxend\vrule\egroup\nointerlineskip\kern-0.4pt
    \moveright##1\dimen0\hbox\bgroup\boxbeg}%
  \def\\##1\\##2:{\boxend\vrule\egroup\nointerlineskip\kern-0.4pt
    \kern ##1\dimen0\moveright##2\dimen0\hbox\bgroup\boxbeg}%
  \moveright#1\dimen0\hbox\bgroup\boxbeg#2\boxend\vrule\egroup
 }$}}
}
\begin{document}

\title{The Symplectic Plactic Monoid, Crystals, and MV Cycles
}

\author{Jacinta Torres
}

\thanks{The results in this paper are part of the author's PhD thesis. The author has been supported by the Graduate School 1269: Global structures in geometry and analysis, financed by the Deutsche Forschungsgemeinschaft. Part of this work was completed during the author's visit to St. Etienne, where she was partially supported by the SPP1388.
}

\maketitle

\begin{abstract}
We study cells in generalised Bott-Samelson varieties for type $\C_{n}$. These cells are parametrised by certain galleries in the affine building. We define a set of \textsl{readable galleries} - we show that the closure in the affine Grassmannian associated to a gallery in this set is an MV cycle. This then defines a map from the set of readable galeries to the set of MV cycles, which we show to be a morphism of crystals. We further compute the fibres of this map in terms of the Littelmann path model. 
\end{abstract}

\section{Introduction}
\label{intro}
This paper is part of a project which was started in \cite{ls} by Gaussent and Littelmann, the aim of which is to establish an explicit relationship between the path model and the set of MV cycles used by Mirkovi\'c and Vilonen for the Geometric Satake equivalence proven in \cite{mirkovicvilonen}.

\subsection{}
 We consider a complex connected reductive algebraic group $\G$ and its affine Grassmannian $\mathcal{G} = \G(\mathbb{C}(\!(t)\!))/\G(\mathbb{C}[[t]])$. We fix a maximal torus $\T \subset \G$. The coweight lattice $\X^{\vee} = \op{Hom}(\mathbb{C}^{\times}, \T)$ can be seen as a subset of $\mathcal{G}$. For a coweight $\lambda$, which we may assume dominant with respect to some choice of Borel subgroup containing $\T$, the closure $\X_{\lambda}$ of the $\G(\mathbb{C}[[t]])$-orbit of $\lambda$ in $\mathcal{G}$ is an algebraic variety which is usually singular. The Geometric Satake equivalence identifies the complex irreducible highest weight module $\li(\lambda)$ for the Langlands dual group $\G^{\vee}$ with the intersection cohomology of $\X_{\lambda}$, a basis of which is given by the classes of certain subvarieties of $\X_{\lambda}$ called MV cycles. 
The set of these subvarieties is denoted by $\mathcal{Z}(\lambda)$. The Geometric Satake equivalence implies that the elements of $\mathcal{Z}(\lambda)$ are in one to one correspondence with the vertices of the crystal $\B(\lambda)$. In \cite{bravermangaitsgory}, Braverman and Gaitsgory endow the set $\mathcal{Z}(\lambda)$ with a crystal structure and show the existence of a crystal isomorphism 
$\varphi: \B(\lambda)\tightoverset{\sim}{\longrightarrow} \mathcal{Z}(\lambda)$. 
\subsection{}
In \cite{ls}, Gaussent and Littelmann define a set $\Gamma(\gamma_{\lambda})^{\op{LS}}$ of LS galleries, which are galleries in the affine building $\mathcal{J}^{aff}$ associated to $\G$, and they endow this set with a crystal structure and an isomorphism of crystals 
$\B(\lambda)\tightoverset{\sim}{\longrightarrow}\Gamma(\gamma_{\lambda})^{\op{LS}}$.  They view the latter as a subset of the $\T$-fixed points in a desingularization $\Sigma_{\gamma_{\lambda}}\overset{\pi}{\longrightarrow}\X_{\lambda}$. To each of these particular fixed points $\delta \in \Gamma(\gamma_{\lambda})^{\op{LS}}$ corresponds a Bia\l{}ynicki-Birula cell $\C_{\delta} \subset \Sigma_{\gamma_{\lambda}}$. Gaussent and Littelmann show in \cite{ls} that the the closure $\overline{\pi(\C_{\delta})}$ is an MV cycle, and Baumann and Gaussent show in \cite{baumanngaussent} that the map

\begin{align*}
\Gamma(\gamma_{\lambda})^{\op{LS}} &\longrightarrow \mathcal{Z}(\lambda)\\
\delta &\mapsto \overline{\pi(\C_{\delta})}
\end{align*}
\noindent
is a crystal isomorphism with respect to the crystal structure on $\mathcal{Z}(\lambda)$ described by Braverman and Gaitsgory in \cite{bravermangaitsgory}. It is natural to ask whether the closures $\overline{\pi(\C_{\delta})}$ are still MV cycles for a more general choice of fixed point $\delta$. 

\subsection{}
In \cite{onesk} they consider \textit{one skeleton} galleries, which are piecewise linear paths in $\X^{\vee}\otimes_{\mathbb{Z}}\mathbb{R}$. Such galleries can be interpreted in terms of Young tableaux for types A, B and C. For $\G^{\vee} = \op{SL}(n,\mathbb{C})$, Gaussent, Littelmann and Nguyen show in \cite{knuth} that for any fixed point $\delta \in \Sigma_{\gamma_{\lambda}}^{\T}$, the closure  $\overline{\pi(\C_{\delta})}$ is in fact an MV cycle. They achieve this using combinatorics of Young tableaux such as word reading and the well known Knuth relations, and by relating them to the Chevalley relations for root subgroups which hold in the affine Grassmannian $\mathcal{G}$. In \cite{jt} it is observed that word reading is a crystal morphism, and this allows one to prove that in this case, the map from all galleries to MV cycles is in fact a morphism of crystals. \\

  It was conjectured in \cite{knuth} that  generalizations of their results hold for arbitrary complex semi-simple algebraic groups, in terms of the plactic algebra defined by Littelmann in \cite{placticalgebra}. It is with this in mind that we formulate and state our results. 

\subsection{Results}
	We work with $\G^{\vee}= \op{SP}(2n, \mathbb{C})$. We define a set $\Gamma(\gamma_{\lambda})^{\op{R}}\supset \Gamma(\gamma_{\lambda})^{\op{LS}}$ of \textit{readable} galleries, which have an explicit formulation in terms of Young tableaux. It is worth mentioning that these galleries correspond to all galleries in type A, also called keys in \cite{knuth}. Type C combinatorics related to LS galleries has been developed by De Concini \cite{deconcini}, Proctor \cite{proctor}, King \cite{king}, Kashiwara-Nakashima \cite{kashiwaranakashima}, Sheats \cite{sheats}, Lakshmibai \cite{Lak1} in the context of standard monomial theory, and Lecouvey \cite{lecouvey}, among others. We use the description of LS galleries of fundamental type given by Lakshmibai in \cite{Lak1}, \cite{Lak2}. We use the formulation given by Lecouvey in \cite{lecouvey}. There is a certain word reading described in \cite{lecouvey} which we show to be a crystal morphism when restricted to readable galleries.\\
	
We obtain results similar to those obtained in \cite{knuth} concerning the defining relations of the \textit{symplectic plactic monoid}, described explicitly by Lecouvey in \cite{lecouvey}, as well as words of readable galleries. These results together with the work of Gaussent-Littelmann \cite{ls}, \cite{onesk}, and Baumann-Gaussent \cite{baumanngaussent} allow us to show in Theorem \ref{main} that given a readable gallery $\delta \in \Gamma(\gamma_{\lambda})^{\op{R}}$ there is an associated dominant coweight $\nu_{\delta} \leq \lambda$ such that:

\begin{enumerate}
\item The closure $\overline{\pi(\C_{\delta})}$ is an MV cycle in $\X_{\nu_{\delta}}$.
\item The map

\begin{align*}
\Gamma(\gamma_{\lambda})^{\op{R}} &\overset{\varphi_{\gamma_{\lambda}}}{\longrightarrow} \underset{\delta \in \Gamma(\gamma_{\lambda})^{\op{R}}}{\bigoplus}\mathcal{Z}(\mu_{\delta^{+}})\\
\delta &\mapsto \overline{\pi(\C_{\delta})}
\end{align*}
\end{enumerate}
is a morphism of crystals. We compute the fibers of this map in terms of the Littelmann path model. Moreover, this map induces an isomorphism when restricted to each connected component. We then provide some examples of galleries $\delta \in \Sigma_{\gamma_{\lambda}}^{\T}- \Gamma(\gamma_{\lambda})^{\op{R}}$ for which $\overline{\pi(\C_{\delta})}$ is not an MV cycle in $\mathcal{Z}(\nu_{\delta})$.

\subsection{}
This paper is organized as follows. In Section 2 we introduce our notation and recall several general facts about affine Grassmannians, MV cycles, galleries in the affine building, generalised Bott-Samelson varieties, and concrete descriptions of the cells $\C_{\delta}$ in them. In Section 3 we introduce crystal structure on combinatorial galleries, motivate our results in terms of the Littemann path model, and define readable galleries as concatenations of LS galleries of fundamental type and `zero lumps.' From Section 4 on we work with $\G^{\vee} = \op{SP}(2n,\mathbb{C})$, where we recall some type C combinatorics and build up to our main result, which we state and prove in Section 6. However, the main ingredients of the proof, stated in Section 5, are proven in Section 7. In section 8 we exhibit some examples in special cases where the image of a certain cell cannot be an MV cycle. In the appendix we show a small technical result that we need.


\section*{Acknowledgements}
The author would like to thank Peter Littelmann for his encouragement and his advice, and Stephane Gaussent for many discussions, specially during the author's visits to Saint \'Etienne, as well as proof reading. Special thanks go to Michael Ehrig for his comments, advice, questions, answers, time, patience and proof reading as well as for many enjoyable discussions.

\section{Preliminaries}
\label{preliminaries}
\subsection{Notation}
Throughout this section, we consider $\G$ to be a complex connected reductive algebraic group associated to a root datum
$(\X, \X^{\vee}, \Phi, \Phi^{\vee})$, and we denote its Langlands dual by $\G^{\vee}$. Let $\T \subset \G$ be a maximal torus of $\G$ with character group $\X= \op{Hom}(\T, \mathbb{C}^{\times})$ and cocharacter group $\X^{\vee}= \op{Hom}(\mathbb{C}^{\times}, \T)$. We identify the Weyl group $\W$ with the quotient $\N_{\G}(\T)/\T$, and will make abuse of notation by denoting a representative in $\N_{\G}(\T)$ of an element $w \in \W$ in the Weyl group by the same symbol, ``$w$''  that we use to denote the element itself. We fix a choice of positive roots $\Phi^{+}$ (this determines a set $\Phi^{\vee,+}$ of positive coroots), and denote the dominance order on $\X$ and $\X^{\vee}$ determined by this choice by `$\leq$'. Let $\Delta \subset \Phi^{+}$ be the basis or set of simple roots of $\Phi$ that is determined by $\Phi^{+}$. Then the set $\Delta^{\vee}$ of all coroots of elements of $\Delta$ forms a basis of the root system $\Phi^{\vee}$. Let $\<-,-\rr$ be the non-degenerate pairing between $\X$ and $\X^{\vee}$, and denote the half sum of positive roots (respectively coroots) by $\rho$ (respectively $\rho^{\vee}$). Note that if $\lambda = {\sum}_{\alpha \in \Delta} n_{\alpha}\alpha$ is a sum of positive roots (respectively $\lambda = {\sum}_{\alpha^{\vee} \in \Delta^{\vee}} n_{\alpha}\alpha^{\vee}$) then $\<\lambda,\rho^{\vee}\rr = {\sum}_{\alpha \in \Delta} n_{\alpha}$ (respectively $\<\rho, \lambda\rr = {\sum}_{\alpha^{\vee} \in \Delta^{\vee}} n_{\alpha}$).\\

Let $\B \subset \G$ be the Borel subgroup of $\G$ containing $\T$ that is determined by the choice of positive roots $\Phi^{+}$, and let $\U \subset \B$ be its unipotent radical. The group $\U$ is generated by the elements $\U_{\alpha}(b)$ for $b \in \mathbb{C}$, $\alpha \in \Phi^{+}\}$, and where for each root $\alpha$, $\U_{\alpha}$ is the one-parameter group it determines. For each cocharacter $\lambda \in \X^{\vee}$ and each non-zero complex number $a \in \mathbb{C}^{\times}$, denote by $a^{\lambda}$ its image $\lambda(a) \in \T$.\\

 The following identities hold in $\G$ (See \cite{steinberg}, \S 6):

\begin{itemize}
\item[] For any $\lambda \in \X^{\vee}, a\in \mathbb{C}^{\times}, b\in \mathbb{C},$ and $\alpha \in \Phi$,
\begin{align}
\label{id1}
a^{\lambda}\U_{\alpha}(b) = \U_{\alpha}(a^{\<\alpha,\lambda\rr}b)a^{\lambda}
\end{align}
\item[](Chevalley's commutator formula) Given linearly independent roots $\alpha, \beta \in \Phi$, there exist numbers $c^{i,j}_{\alpha, \beta} \in \{\pm 1, \pm 2, \pm 3\}$ such that, for all $a, b \in \mathbb{C}$: 
\begin{align}
\label{chevalley}
\U_{\alpha}(a)^{-1}\U_{\beta}(b)^{-1}\U_{\alpha}(a)\U_{\beta}(a) = \underset{i, j \in \mathbb{N}^{>0}}{\prod}\U_{i\alpha +j\beta}(c^{i,j}_{\alpha,\beta}(-a)^{i}b^{j})
\end{align}
where the product is taken in some fixed order. The ${c_{\alpha, \beta}^{ij}}'s$ are integers depending on $\alpha, \beta$, and on the chosen order in the product. 
\end{itemize}
\subsection{Affine Grassmannians}
Let $\mathcal{O} = \mathbb{C}[[t]]$ denote the ring of complex formal power series and  let $\mathcal{K} = \mathbb{C}(\!(t)\!)$ denote its field of fractions; it is the field of complex Laurent power series. For any $\mathbb{C}$-algebra $\mathcal{R}$, denote by $\G(\mathcal{R})$ the set of $\mathcal{R}$-valued points. The set

\begin{align*}
\mathcal{G} = \G(\mathcal{K})/\G(\mathcal{O})
\end{align*}

\noindent
is called the \textbf{affine Grassmannian} associated to $\G$. We will denote the class in $\mathcal{G}$ of an element $g\in \G(\mathcal{K})$ by $[g]$. A coweight $\lambda : \mathbb{C}^{\times}\rightarrow \T \subset \G$ determines a point  $t^{\lambda} \in \G(\mathcal{K})$ and hence a class $[t^{\lambda}] \in \mathcal{G}$. This map is injective, and we may therefore consider $\X^{\vee}$ as a subset of $\mathcal{G}$.\\

$\G(\mathcal{O})$-orbits in $\mathcal{G}$ are determined by the Cartan decomposition:

\begin{align*}
\mathcal{G} = \underset{\lambda \in \X^{\vee,+}}{\bigsqcup}\G(\mathcal{O})[t^{\lambda}].
\end{align*}
\noindent
Each $\G(\mathcal{O})$-orbit has the structure of an algebraic variety induced from the pro-group structure of $\G(\mathcal{O})$ and it is known that for a coweight $\lambda \in \X^{\vee, +}$: 
\begin{align*}
\overline{\G(\mathcal{O})[t^{\lambda}]} = \underset{\mu \in \X^{\vee, +}, \mu \leq \lambda}{\bigsqcup} \G(\mathcal{O})[t^{\mu}].
\end{align*}
We call the closure $\overline{\G(\mathcal{O})[t^{\lambda}]}$ a \textbf{generalised Schubert variety} and we denote it by $\X_{\lambda}$. This variety is usually singular. In $\ref{bottsamelson}$, we will review certain resolutions of singularities of it.  \\

The $\U(\mathcal{K})$-orbits in $\mathcal{G}$ are given by the Iwasawa decomposition: 

\begin{align*}
\mathcal{G} = \underset{\lambda \in \X^{\vee}}{\bigsqcup}\U(\mathcal{K})[t^{\lambda}].
\end{align*}
\noindent
These orbits are ind-varieties, and their closures can be described as follows (see \cite{mirkovicvilonen}, Proposition 3.1 a.): 

\begin{align*}
\overline{\U(\mathcal{K})[t^{\lambda}]} = \underset{\mu \leq \lambda}{\bigcup} \U(\mathcal{K})[t^{\mu}]
\end{align*}

\noindent
for any $\lambda \in \X$.
\subsection{MV Cycles and Crystals}
Let $\lambda \in \X^{\vee,+}$ and $\mu \in \X^{\vee}$ be a dominant integral coweight and any coweight, respectively. Then by Theorem 3.2 \textsl{a} in \cite{mirkovicvilonen}, the intersection $\U(\mathcal{K})[t^{\mu}]\cap \G(\mathcal{O})[t^{\lambda}]$ is non-empty if and only if $\mu \leq \lambda$, and in that case its closure is pure dimensional of dimension $\<\rho, \lambda +\mu\rr$ and has the same number of irreducible components as the dimension of the irreducible representation $\li(\lambda)$ of $\G^{\vee}$ of highest weight $\lambda$ (Corollary 7.4 in \cite{mirkovicvilonen}). Note that this makes sense because $\X^{\vee}$ may be identified with the character group of a maximal torus of $\G^{\vee}$. Explicitly, $\X^{\vee} \cong \op{Hom}(\T^{\vee}, \mathbb{C}^{\times})$, where $\T^{\vee}$ is the Langlands dual of $\T$, which is a maximal torus of $\G^{\vee}$ (see \cite{mirkovicvilonen}, Section 7). \\

We denote the set of all irreducible components of a given topological space $\Y$ by $\op{Irr}(\Y)$. Consider the sets
\begin{align*}
\mathcal{Z}(\lambda)_{\mu} &= \op{Irr}(\overline{\U(\mathcal{K})[t^{\mu}]\cap \G(\mathcal{O})[t^{\lambda}]}) \hbox{ and }\\
\mathcal{Z}(\lambda) &= \underset{\mu \in \X^{\vee}}{\bigsqcup}\mathcal{Z}(\lambda)_{\mu}.
\end{align*}
The elements of these sets are called \textbf{MV cycles}. In \cite{bravermangaitsgory}, Section 3.3, Braverman and Gaitsgory have endowed the set $\mathcal{Z}(\lambda)$ with a crystal structure and have shown the existence of an isomorphism of crystals $\B(\lambda)\tightoverset{\sim}{\longrightarrow} \mathcal{Z}(\lambda)$. We do not use the definition of this crystal structure, but we denote by $\tightoverset{\sim}{f_{\alpha_{i}}}$ (respectively $\tightoverset{\sim}{e_{\alpha_{i}}}$) the corresponding root operators for $i \in \{1, \cdots, n\}$, where $n$ is the rank of the root system $\Phi$. See  \ref{crystals} below for the definition of a crystal.

\subsection{Galleries in the Affine Building}
\label{galleries}
Let $\mathcal{J}^{\op{aff}}$ be the affine building associated to $\G$ and $\mathcal{K}$. It is a union of simplicial complexes called \textit{apartments}, each of which is isomorphic to the Coxeter complex  of the same type as the extended Dynkin diagram associated to $\G$. 
The affine Grassmannian $\mathcal{G}$ can be $\G(\mathcal{K})$ equivariantly embedded into the building $\mathcal{J}^{aff}$, which also carries a 
$\G(\mathcal{K})$ action. Denote by $\Phi^{\op{aff}}$ the set of real affine roots associated to $\Phi$; we identify it with the set $\Phi \times \mathbb{Z}$.\\

Let $\mathbb{A} = \X^{\vee} \otimes_{\mathbb{Z}}\mathbb{R}$. For each $(\alpha, n) \in \Phi^{\op{aff}}$, consider the associated hyperplane

\begin{align*}
\mm_{(\alpha, n)} &=  \{x\in \mathbb{A}: \<\alpha, x\rr = n\}
\end{align*}
and the positive, respectively negative half spaces

\begin{align*}
\mm_{(\alpha, n)}^{+} &=  \{x\in \mathbb{A}: \<\alpha, x\rr \geq n\}\\
\mm_{(\alpha, n)}^{-} &=  \{x\in \mathbb{A}: \<\alpha, x\rr \leq n\}.
\end{align*}

\noindent
Denote by $\W^{\op{aff}}$ the affine Weyl group generated by all the affine reflections $s_{\alpha,n}$ with respect to the affine hyperplanes $\mm_{\alpha, n}$. We have an embedding $\W \hookrightarrow \W^{\op{aff}}$ given by $s_{\alpha} \mapsto s_{\alpha,0}$. The \textbf{dominant Weyl chamber} is the set

\begin{align*}
\C^{+} = \{x \in \mathbb{A}: \<\alpha, x\rr\ > 0\hbox{ }\forall \alpha \in \Delta\}
\end{align*}

\noindent
and the \textbf{fundamental alcove} is in turn

\begin{align*}
\Delta^{\op{f}} = \{x \in \C^{+}: \<\alpha, x\rr\ \leq 1\hbox{ }\forall \alpha \in \Phi^{+}\}.
\end{align*}

There is a unique apartment in the affine building $\mathcal{J}^{\op{aff}}$ that contains the image of the set of coweights $\X^{\vee}\subset \mathcal{G}$ under the embedding $\mathcal{G}\hookrightarrow \mathcal{J}^{\op{aff}}$. This apartment is isomorphic to the affine Coxeter complex associated to $\W^{\op{aff}}$; its faces are given by all possible intersections of the hyperplanes $\mm_{(\alpha, n)}$ and their associated (closed) positive and negative half-spaces $\mm^{\pm}_{(\alpha,n)}$. It is called the \textbf{standard apartment} in the affine building $\mathcal{J}^{\op{aff}}$. The action of $\W^{\op{aff}}$ on the affine building $\mathcal{J}^{\op{aff}}$ coincides, when restricted to the standard apartment, with the one induced by the natural action of $\W^{\op{aff}}$ on $\mathbb{A}$; the fundamental alcove is a fundamental domain for the latter. \\

To each real affine root $(\alpha,n) \in \Phi^{\op{aff}}$ is attached the one-parameter additive \textbf{root subgroup } $\U_{(\alpha,n)}$ of $\G(\mathcal{K})$ defined by $b \mapsto \U_{\alpha}(bt^{n})$ for $b\in \mathbb{C}$. Let $\lambda \in \X^{\vee}$ and $b\in \mathbb{C}$. Then identity (\ref{id1}) implies that:

\begin{align}
\label{stabiliseone}
\U_{(\alpha,n)}(b)[t^{\lambda}] = [\U_{\alpha}(bt^{n})t^{\lambda}] = [t^{\lambda}\U_{\alpha}(bt^{n-\< \alpha,\lambda\rr})],
\end{align}

\noindent
and $[t^{\lambda}\U_{\alpha}(bt^{n-\>\alpha,\lambda\rr})] = [t^{\lambda}]$ if and only if $\U_{\alpha}(bt^{n-\<\alpha,\lambda\rr }) \subset \G(\mathcal{O})$, or, equivalently, $\<\alpha,\lambda\rr \leq n$. Hence, the root subgroup $\U_{(\alpha,n)}$ stabilises the point $[t^{\lambda}] \in \mathcal{G}\hookrightarrow \mathcal{J}^{\op{aff}}$ if and only if $\lambda \in \mm^{-}_{(\alpha,n)}$. For each face $\F$ in the standard apartment, denote by $\p_{\F}, \U_{\F}$ and $\W^{\op{aff}}_{\F}$ its stabilizer in $\G(\mathcal{K}), \U(\mathcal{K})$ and $\W^{\op{aff}}$ respectively. These subgroups are generated by the torus $\T$ and the root subgroups $\U_{(\alpha,n)}$ such that $\F \subset \mm^{-}_{(\alpha,n)}$, the root subgroups $\U_{(\alpha,n)}\subset \p_{\F}$ such that $\alpha \in \Phi^{+}$, and those affine reflections $s_{(\alpha,n)} \in \W^{\op{aff}}$ such that $\F \subset \mm_{(\alpha,n)}$, respectively. See \cite{ls}, Section 3.3, Example 3, and \cite{baumanngaussent}, Proposition 5.1 (ii).  \\

\begin{ex}
\label{picture1}
Let $\G^{\vee} = \op{SP}(4, \mathbb{C})$, then $\Phi^{+} = \{\alpha_{1}, \alpha_{2}, \alpha_{1}+\alpha_{2}, \alpha_{1}+2\alpha_{2}\}$. In the picture below the shaded region is the upper halfspace $\mm_{(\alpha_{2},0)}^{+}$. Let $\F$ be the face in the standard apartment that joins the vertices $-(\alpha_{1}+\alpha_{2})$ and $-\alpha_{1}$. 
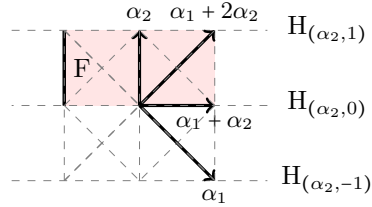
\begin{center}
\begin{tikzpicture}
\fill[red!20!, opacity = 0.5] (-1,0) rectangle (1,1);
\node[right] at (-1,0.5) {$\F$};
\draw[very thick] (-1,0) -- (-1,1);
\draw[very thick] (-1,0) -- (-1,1);
\node[below] at (1,0) {\small{\mbox{$\alpha_{1}+\alpha_{2}$}}};
\draw[->, very thick] (0,0) -- (1,0);
\node[above] at (2.5,-0.3){\hbox{$\mm_{(\alpha_{2},0)}$}};
\draw[dashed, gray, very thin] (1.7,0) --(-1.7,0);
\node[above] at (2.5,0.7){$\mm_{(\alpha_{2},1)}$};
\draw[dashed, gray, very thin] (-1.7,1) -- (1.7,1);
\node[above] at (2.5,-1.3){$\mm_{(\alpha_{2},-1)}$};
\draw[dashed, gray, very thin] (-1.7,-1) -- (1.7,-1);
\draw[->,very thick] (0,0) --(0,1);
\draw[dashed, gray, very thin] (0,1) --(0,-1);
\draw[dashed, gray, very thin] (1,1) --(1,-1);
\draw[dashed, gray, very thin] (-1,1) --(-1,-1);
\node[above] at (0,1){\small{\mbox{$\alpha_{2}$}}};
\draw[dashed, gray, very thin] (1,1) --(-1,-1);
\draw[dashed, gray, very thin] (-1,-1) --(1,1);
\node[above] at (1,1){\small{\mbox{$\alpha_{1}+2\alpha_{2}$}}};
\draw[->, very thick] (0,0) -- (1,1);
\draw[dashed, gray, very thin] (1,1) --(-1,-1);
\draw[dashed, gray, very thin] (1,0) --(0,-1);
\draw[dashed, gray, very thin] (-1,0) --(0,1);
\node[below] at (1,-1){\small{\mbox{$\alpha_{1}$}}};
\draw[->, very thick] (0,0) -- (1,-1);
\draw[dashed, gray, very thin] (1,-1) --(-1,1);
\draw[dashed, gray, very thin] (-1,1) -- (1,-1); 
\draw[dashed, gray, very thin] (1,0) -- (0,1); 
\draw[dashed, gray, very thin] (-1,1) -- (1,-1); 
\draw[dashed, gray, very thin] (-1,0) -- (0,-1); 
\end{tikzpicture}
\end{center}
The subgroup $\p_{\F}$ is generated by the root subgroups associated to the following real roots 

\begin{align*}
(\alpha_{1}, n) &\hbox{  } n \geq -1\\
(\alpha_{2}, n) & \hbox{  }n \geq 1\\
(\alpha_{1}+\alpha_{2}, n) &\hbox{  } n \geq -1\\
(\alpha_{1}+2\alpha_{2}, n) & \hbox{  }n \geq 0 \\
(-\alpha_{1}, n) &\hbox{  } n \geq 2\\
(-\alpha_{2}, n) &\hbox{  } n \geq 0 \\
(-(\alpha_{1}+\alpha_{2}), n) &\hbox{  } n \geq 1\\
(-(\alpha_{1}+2\alpha_{2}, n)) &\hbox{  } n \geq 1
\end{align*}

The stabiliser $\U_{\F}$ is generated by the root subgroups associated to those previously stated roots $(\alpha, n)$ such that $\alpha \in \Phi^{+}$ is a positive root and $\W^{\op{aff}}_{\F} = \{s_{(\alpha_{1}+\alpha_{2},-1)}, 1\}$.
\end{ex}

A \textbf{gallery} is a sequence of faces in the affine building $\mathcal{J}^{\op{aff}}$

\begin{align}
\label{gallery}
\gamma = (\V_{0} = 0, \E_{0} , \V_{1} , \cdots, \E_{k}, \V_{k+1})
\end{align}

such that:

\begin{itemize}
\item[1.] For each $i\in \{1, \cdots, k\}, \hbox{ }\V_{i} \subset\E_{i}\supset \V_{i+1}$.
\item[2.] Each face labelled $\V_{i}$ has dimension zero (is a \textbf{vertex}) and each face labelled $\E_{i}$ has dimension one (is an \textbf{edge}). In particular, each face in the sequence $\gamma$ is contained in the one-skeleton of the standard apartment. 
	\item[3.] The last vertex $\V_{k+1}$ is a \textbf{special vertex}: its stabiliser in the affine Weyl group $\W^{\op{aff}}$ is isomorphic to the finite Weyl group $\W$ associated to $\G$. 
\end{itemize}

\noindent
Denote all the set of all galleries in the affine building by $\Sigma$. If in addition each face in the sequence belongs to the standard apartment, then $\gamma$ is called a \textbf{combinatorial gallery}. We will denote the set of all combinatorial galleries in the affine building by $\Gamma$. In this case, the third condition is equivalent to requiring the last vertex $\V_{k+1}$ to be a coweight. From now on, if $\gamma$ is a combinatorial gallery we will denote the coweight corresponding to its final vertex by $\mu_{\gamma}$ in order to distinguish it from the vertex.

\begin{rem}
The galleries we defined are actually called \textit{one-skeleton} galleries in the literature. The word gallery was originally used to describe a more general  class of face sequences but since we only work with one-skeleton galleries in this paper, we leave the word `one-skeleton' out.
\end{rem}

\subsection{Bott-Samelson varieties}
\label{bottsamelson}

Let $\gamma$ be a combinatorial gallery (notation as above). The following lemma can be obtained from \cite{onesk}: Lemma 4.8 and Definition 4.6.

\begin{lem}
\label{typelemma}
For each $j \in \{1, \cdots, k\}$ there exist elements $w_{j}\in \W^{\op{aff}}_{\V_{i}}$ and a unique combinatorial gallery
\begin{align*}
\gamma^{f} = (\V_{0}^{f},\E_{0}^{f},\V_{1}^{f},\cdots,\V_{k+1}^{f})
\end{align*}
with each one of its faces is contained in the fundamental alcove such that $w_{0}\cdots w_{r} \E^{f}_{r} = \E_{r}$. 
\end{lem}


If two galleries have the same associated gallery we say that the two galleries have \textbf{the same type}. We will denote all the combinatorial galleries that have the same type as a given combinatorial gallery $\gamma$ by $\Gamma(\gamma)$. The map 

\begin{align}
\label{combg}
\W^{\op{aff}}_{\V_{0}}\times \cdots \times \W^{\op{aff}}_{\V_{k}} &\rightarrow \Gamma(\gamma) \\
(w_{0}, \cdots, w_{k}) &\mapsto (\V_{0}, w_{0}\E_{0}, w_{0}\V_{1}, w_{0}w_{1}\E_{1}, \cdots, w_{0}\cdots w_{k}\V_{k+1})
\end{align}

\noindent
induces a bijection between $\Gamma(\gamma)$ and the set $\prod_{i=0}^{r}\W^{\op{aff}}_{\V_{i}}/\W^{\op{aff}}_{\E_{i}}$; it is in particular finite. For a proof see \cite{onesk}, Lemma 4.8.  

\begin{defn}
\label{fiberedproductdef}
The \textbf{Bott-Samelson variety} of type $\gamma^{f}$ is the quotient of 
\begin{align*}
\G(\mathcal{O})\times \p_{\V_{1}^{f}} \times \cdots \times \p_{\V_{k}^{f}}
\end{align*}
by the following left action of $\p_{\E_{0}^{f}} \times \cdots \times \p_{\E_{k}^{f}}$: 
\begin{align*}
(p_{0}, p_{1}, \cdots, p_{k})\cdot (q_{0}, \cdots, q_{k}):= (q_{0}p_{0}, p_{0}^{-1}q_{1}p_{1}, \cdots, p_{k-1}^{-1}q_{k}p_{k}).
\end{align*}
We will denote it by $\Sigma_{\gamma^{f}}$.
\end{defn}
\noindent
The pro-group structure of the groups $\p_{\V_{i}^{f}}, \p_{\E_{i}^{f}}$ assures that $\Sigma_{\gamma^{f}}$ is in fact a smooth variety. To each point $(g_{0}, \cdots, g_{k})\in \G(\mathcal{O})\times \p_{\V_{1}^{f}} \times \cdots \times \p_{\V_{k}^{f}}$ one can associate a gallery $$(\V^{f}_{0}, g_{0}\E^{f}_{0}, g_{0}\V^{f}_{1}, g_{0}g_{1}\V^{f}_{2}, \cdots, g_{0}\cdots g_{k}\V^{f}_{k+1}).$$ This induces a well defined injective map 
$i:\Sigma_{\gamma^{f}} \hookrightarrow \Sigma$. With respect to this identification, $\T$ - fixed points in $\Sigma_{\gamma^{f}}$ are in natural bijection with the set $\Gamma(\gamma^{f})$ of combinatorial galleries of type $\gamma^{f}$.\\

Let $\omega \in \mathbb{A}$ be a fundamental coweight. We define a particular combinatorial gallery that starts at $0$ and ends at $\omega$. Let $\V^{\omega}_{1}, \cdots, \V^{\omega}_{k}$ be the vertices in the standard apartment that lie on the open line segment joining $0$ and $\omega$, numbered such that $\V^{\omega}_{i+1}$ lies on the open line segment joining $\V^{\omega}_{i}$ and $\omega$. Let further $\E^{\omega}_{i}$ denote the face contained in $\mathbb{A}$ that contains the vertices $\V^{\omega}_{i}$ and $\V^{\omega}_{i+1}$. The gallery

\begin{align*}
\gamma_{\omega}: = (0 =\V^{\omega}_{0}, \E^{\omega}_{0}, \V^{\omega}_{1}, \E^{\omega}_{1}, \cdots ,\E^{\omega}_{k}, \V^{\omega}_{k+1}  = \omega)
\end{align*}

\noindent
is called a \textbf{fundamental gallery}. Galleries of the same type as a fundamental gallery $\gamma_{\omega}$  will be called \textbf{galleries of fundamental type} $\mathbf{\omega}$.\\

Now let $\lambda \in \X^{\vee, +}$ be a dominant integral coweight and $\gamma_{\lambda}$ a gallery with endpoint the coweight $\lambda$ and such that it is a concatenation of fundamental galleries, where concatenation of two combinatorial galleries $\gamma_{1}*\gamma_{2}$ is defined by translating the second one to the endpoint of the first one. (Note that it follows from the definition of type that if $\gamma, \nu$ are two galleries of the same type as $\delta$, respectively $\eta$, then $\gamma * \nu$ has the same type as $\delta* \eta$. Actually, if $\gamma = \gamma_{1}*\cdots * \gamma_{r}$ then $\Gamma(\gamma) = \{\delta_{1}*\cdots* \delta_{r} : \delta_{i} \in \Gamma(\gamma_{i})\}$. ) Then the map
\begin{align}
\label{resolution}
\Sigma_{\gamma_{\lambda}^{f}}&\overset{\pi}{\longrightarrow} \X_{\lambda}\\
[g_{0}, \cdots, g_{r}]& \rightarrow g_{0}\cdots g_{r} [t^{\mu_{\gamma^{f}}}] \nn
\end{align}
 is a resolution of singularities of the generalised Schubert variety $\X_{\lambda}$.
 
 \begin{rem}
 \label{resolutionremark1}
That the above map is in fact a resolution of singularities is due to the fact that a gallery such as the one considered is minimal (see \cite{onesk}, Section 5 and Section 4.3, Proposition 3). This resembles the condition for usual Bott-Samelson varieties associated to a reduced expression: see \cite{ls}, Section 9, Proposition 7.
 \end{rem}
 
 \begin{rem}
 \label{resolutionremark2}
 The map $(\ref{resolution})$ makes sense for any combinatorial gallery $\gamma$ : in this generality one has a map $\Sigma_{\gamma^{f}} \overset{\pi}{\longrightarrow} \mathcal{G}, g_{0}, \cdots, g_{r}[t^{\mu_{\gamma}}]$.
 \end{rem}

\subsection{Cells and positive crossings}

Let $r_{\infty}: \mathcal{J}^{a} \rightarrow \mathbb{A}$ be the retraction at infinity (see \cite{ls}, Definition 8). It extends to a map $r_{\gamma^{f}}: \Sigma_{\gamma^{f}} \rightarrow \Gamma(\gamma^{f})$. The cell $\C_{\delta} = r_{\gamma^{f}}^{-1}(\delta) (\delta \in \Gamma(\gamma^{f}))$  is explicitly described in \cite{ls}, \cite{onesk} by Gaussent-Littelmann and \cite{baumanngaussent} by Baumann-Gaussent. In this subsection we recollect their results; we will need them later. These results are formulated in terms of galleries of the same type as $\gamma_{\lambda}$; we formulate them for any combinatorial gallery. The proofs remain the same, and therefore we do not provide them, but refer the reader to \cite{onesk} and \cite{ls}.  First consider the subgroup $\U(\mathcal{K})$ of $\G(\mathcal{K})$. It is generated by the elements of the root subgroups $\U_{(\alpha, n)}$ for $\alpha \in \Phi^{+}$ a positive root and $n \in \mathbb{Z}$. Let $\V \subset \E$ be a vertex and an edge (respectively) in the standard apartment, the vertex contained in the edge. Consider the subset of affine roots $\Phi^{+}_{(\V, \E)}= \{(\alpha,n)\in \Phi^{\op{aff}}: \alpha \in \Phi^{+}, \V \in \mm_{(\alpha,n)}, \E_{i} \nsubseteq \mm^{-}_{(\alpha,n)}\}$ and let $\mathbb{U}_{(\V, \E)}$ denote the subgroup of $\U(\mathcal{K})$ generated by $\U_{(\alpha,n)}$ for all $(\alpha,n) \in \Phi^{+}_{(\V, \E)}$. The following proposition will be very useful in Section \ref{countingpositivecrossings}. It is stated and proven as Proposition 5.1 (ii) in \cite{baumanngaussent}. 

\begin{prop}
\label{order}
Let $\V \subset \E$ be a vertex and an edge in the standard apartment as above. Then $\mathbb{U}_{(\V,\E)}$ is a set of representatives for the right cosets of $\U_{\E}$ in $\U_{\F}$. For any total order on the set $\Phi^{+}_{(\V, \E)}$, the map
\begin{align*}
(a_{\beta})_{\beta \in \Phi^{+}_{(\V, \E)}} \mapsto \underset{\beta \in \Phi^{+}_{(\V, \E)}}{\prod} \U_{\beta}(a_{\beta})
\end{align*}
is a bijection from $\mathbb{C}^{|\Phi^{+}_{(\V, \E)}|}$ onto $\mathbb{U}_{(\V, \E)}$. 
\end{prop}

\noindent
 Now let $\gamma$ be a combinatorial gallery with notation as in (\ref{gallery}). For each $i\in \{1, \cdots, k\}$, let $\mathbb{U}^{\gamma}_{\V_{i}}: = \mathbb{U}_{(\V_{i}, \E_{i})}$. For later use we fix the notation $\Phi^{\gamma}_{i}: = \Phi^{+}_{\V_{i}, \E_{i}}$.   

\begin{ex}
Let $\G^{\vee} = \op{SP}(4, \mathbb{C})$ as in Example \ref{picture1} and $\gamma = \gamma_{\omega_{1}}$. Then $\mathbb{U}_{0}$ is generated by the root subgroups associated to the real roots $(\alpha_{1}, 0),(\alpha_{1}+\alpha_{2}, 0), (\alpha_{1}+2\alpha_{2}, 0)$. If $\gamma = \delta$ is the gallery with one edge and endpoint $\alpha_{2}$, then $\mathbb{U}_{0}$ is generated by the groups associated to $(\alpha_{2}, 0), (\alpha_{1}+2\alpha_{2}, 0)$.

\label{picture2}
\begin{center}
\begin{tikzpicture}
\fill[red!20!, opacity = 0.5] (-1,0) rectangle (1,1);
\node[below] at (0.5,0) {\small{\mbox{$\gamma_{\omega_{1}}$}}};
\draw[-, very thick] (0,0) -- (1,0);
\node[above] at (2,-0.3){\hbox{$\mm_{(\alpha_{2},0)}$}};
\draw[dashed, gray, very thin] (1.4,0) --(-1,0);
\draw[dashed, gray, very thin] (-1,1) -- (1,1);
\draw[dashed, gray, very thin] (-1,-1) -- (1,-1);
\draw[-, very thick] (0,0) --(0,1);
\draw[dashed, gray, very thin] (0,1) --(0,-1.4);
\draw[dashed, gray, very thin] (1,1) --(1,-1);
\draw[dashed, gray, very thin] (-1,1) --(-1,-1);
\node[right] at (0,0.5){\small{\mbox{$\delta$}}};
\draw[dashed, gray, very thin] (1,1) --(-1,-1);
\draw[dashed, gray, very thin] (-1,-1) --(1,1);
\draw[dashed, gray, very thin] (1,1) --(-1.4,-1.4);
\draw[dashed, gray, very thin] (1,0) --(0,-1);
\draw[dashed, gray, very thin] (-1,0) --(0,1);
\draw[dashed, gray, very thin] (1,-1) --(-1,1);
\draw[dashed, gray, very thin] (-1,1) -- (1,-1); 
\draw[dashed, gray, very thin] (1,0) -- (0,1); 
\draw[dashed, gray, very thin] (-1,1) -- (1.4,-1.4); 
\draw[dashed, gray, very thin] (-1,0) -- (0,-1); 
\node[below] at (0,-1.1) {$\mm_{(\alpha_{1}+\alpha_{2},0)}$};
\node[below] at (1.7,-1.1) {$\mm_{(\alpha_{1}+2\alpha_{2},0)}$};
\node[below] at (-1.4,-1.1) {$\mm_{(\alpha_{1},0)}$};
\end{tikzpicture}
\end{center}
\end{ex}

\noindent
Now write $\delta = (\V_{0}, \E_{0}, \cdots, \E_{k}, \V_{k+1}) \in \Gamma(\gamma^{f})$ in terms of Definition \ref{fiberedproductdef} and Lemma \ref{typelemma} as $
\delta = [\delta_{0},\cdots, \delta_{k}].$ This means $\delta_{i} \in \W^{\op{aff}}_{\V^{f}_{i}}$ and $\delta_{0}\cdots \delta_{j} \E^{f}_{j} = \E_{j}$. A beautiful exposition of  the following description of the cell $\C_{\delta}$ can be found in \cite{onesk}, Proposition 4.19. 

\begin{thm}
\label{celldescription}
The map
\begin{align*}
\mathbb{U}:= \mathbb{U}_{\V_{0}}\times \mathbb{U}_{\V_{1}}\times \cdots \times \mathbb{U}_{\V_{k}}
&\overset{\varphi}{\longrightarrow} \Sigma_{\gamma^{\op{f}}}\\
(u_{0}, \cdots, u_{k}) &\mapsto [u_{0}\delta_{0}, \delta_{0}^{-1}u_{1}\delta_{0}\delta_{1}, \cdots , (\delta_{0}\cdots \delta_{k-1})^{-1}u_{k}\delta_{0}\cdots \delta_{k}]
\end{align*}
is injective and has image $\C_{\delta}$. 
\end{thm}

The following corollary can be found in \cite{knuth} as Corollary 3 for $\G^{\vee} = \op{SL}(n, \mathbb{C})$. Note that in particular it implies that $u\pi(\C_{\delta}) = \pi(\C_{\delta})$ for all $u \in \U_{\V_{0}}$.

\begin{cor}
\label{goodtrick}
The following equality holds.
\begin{align*}
\pi(\C_{\delta}) &= \mathbb{U}_{\V_{0}} \cdots \mathbb{U}_{\V_{k}}[t^{\mu_{\delta}}] =\U_{\V_{0}}\cdots \U_{\V_{k}}[t^{\mu_{\delta}}]
\end{align*}
\end{cor}

\begin{proof}
By Theorem \ref{celldescription} the image of the map 
\begin{align*}
\U_{\V_{0}}\times \cdots \times \U_{\V_{r}} & \rightarrow \Sigma_{\gamma^{f}}\\
(u_{0}, \cdots, u_{r}) & \mapsto [u_{0}\delta_{0}, \delta_{0}^{-1}u_{1}\delta_{0}\delta_{1}, \cdots , \delta_{0}\cdots \delta_{r-1}^{-1}u_{r}\delta_{0}\cdots \delta_{r}]
\end{align*}
 is contained in the cell $\C_{\delta}$ and is surjective. The corollary follows since \newline
 $\delta_{0}\cdots \delta_{j} \mu_{\gamma^{f}} = \mu_{\delta}$. 
\end{proof}
\noindent

\section{Crystal structure on combinatorial galleries, the Littelmann path model, and Lakshmibai Seshadri galleries}

\label{crystalsubsection}

Let $\lambda \in \X^{+, \vee}$ be a dominant integral coweight and let $\li(\lambda)$ be the corresponding simple module of $\G^{\vee}$. To $\li(\lambda)$ is associated a certain graph $\B(\lambda)$ that is its ``combinatorial model''.  It is a connected \textsl{highest weight} crystal, which means that there exists $b_{\lambda} \in \B(\lambda)$ such that $e_{\alpha_{i}}b_{\lambda} = 0$ for all $i \in \{1, \cdots, n-1\}$. The crystal $\B(\lambda)$ also has the characterising property that 
\begin{align*}
\op{dim}(\li(\lambda)_{\mu}) = \#\{b \in \B(\lambda): \op{wt}(b) = \mu\}.
\end{align*}
 See below for the definitions.  After recalling the notion of a crystal we review the crystal structure on the set of all combinatorial galleries $\Gamma$.

\subsection{Crystals}
\label{crystals}
A \textbf{crystal} is a set $\B$ together with maps 
\begin{align*}
e_{\alpha_{i}}, f_{\alpha_{i}}:& \B \rightarrow \B \cup \{0\} \mbox{(the \textbf{root operators})},\\
 \op{wt}:& \B \rightarrow \X^{\vee} 
\end{align*}
for  $i \in \{1, \cdots, n\}$ such that for every $b, b' \in \B$ and $i\in \{1, \cdots , n-1\}, b' = e_{\alpha_{i}}(b)$ if and only if $b = f_{\alpha_{i}}(b')$, and, in this case, setting $$\epsilon_{i}(b'')= \op{max}\{n: e_{\alpha_{i}}^{n}(b)\neq 0 \}$$ and $$\phi_{i}(b'')=\op{max}\{n: f_{\alpha_{i}}^{n}(b'')\neq 0\}$$ for any $b'' \in \B$, the following properties are satisfied.

\begin{enumerate}
\item $\op{wt}(b') = \op{wt}(b)+\alpha_{i}^{\vee}$
\item $\phi(b) = \epsilon_{i}(b) +\< \alpha_{i},\op{wt}(b)\rr$
\end{enumerate}

\noindent
A crystal is in particular a graph, which we may decompose into the disjoint union of its connected components. Each element $b\in \B$ lies in a unique connected component which we will denote by $\op{Conn}(b)$. A \textbf{crystal morphism} is a map $\F:\B \rightarrow \B'$ between the underlying sets of two crystals $\B$ and $\B'$ such that $\op{wt}(\F(b)) = \op{wt}(b)$ and such that it commutes with the action of the root operators. A crystal morphism is an isomorphism if it is bijective. \\

\subsection{Crystal structure on combinatorial galleries}
\begin{defn}
\label{definitionofrootoperators}
For each $i \in \{1, \cdots, n\}$ and each simple root $\alpha_{i}$, we recall the definition of the root operators $f_{\alpha_{i}}$ and $e_{\alpha_{i}}$ on the set of combinatorial galleries $\Gamma$ and endow the set of combinatorial galleries with a crystal structure. We follow Section 6 in \cite{ls} and Section 1 in \cite{bravermangaitsgory}. We refer the reader to \cite{kashiwaraoncrystalbases} for a detailed account of the theory of crystals. \\

Let $\gamma = (\V_{0}, \E_{0}, \V_{1}, \E_{1}, \cdots, \E_{k}, \V_{k+1})$ be a combinatorial gallery. Define $\op{wt}(\gamma) = \mu_{\gamma}$. Let $m_{\alpha_{i}} = m \in \mathbb{Z}$ be minimal such that $\V_{r} \in \mm_{(\alpha_{i}, m)}$ for $r \in \{1, \cdots, k+1\}$. Note that $m \leq 0$. \\

\noindent
\textbf{Definition of }$\mathbf{f_{\alpha_{i}}}$: Suppose $\<\alpha_{i},\mu_{\gamma}\rr \geq m+1$. Let $j$ be maximal such that $\V_{j} \in \mm_{(\alpha_{i},m)}$ and let $j < r \leq k+1$ be minimal such that $\V_{r} \in \mm_{(\alpha_{i}, m+1)}$. Let

\[
    \E_{i}'= 
\begin{cases}
    \E_{i}& \text{if } i < j\\
    s_{(\alpha_{i},m)}(\E_{i}) &\text{ if } j \leq i < r\\
    t_{-\alpha_{i}^{\vee}}(\E_{i}) &\text{ if } i \geq r
\end{cases}
\]
and define $$f_{\alpha_{i}}(\gamma) = (\V'_{0}, \E'_{0}, \V'_{1}, \E'_{1}, \cdots, \E'_{r}, \V'_{k+1}).$$ If $\<\alpha_{i},\mu_{\gamma}\rr < m+1$, then $f_{\alpha_{i}}(\gamma) = 0$. \\

\noindent
\textbf{Definition of }$\mathbf{e_{\alpha_{i}}}$:  Suppose that $m \leq -1$. Let $r$ be minimal such that the vertex $\V_{r} \in \mm_{(\alpha_{i}, m)}$ and let $0 \leq j < r$ maximal such that $\V_{j} \in \mm_{(\alpha_{i}, m+1)}$. Let 

\[
    \E_{i}'= 
\begin{cases}
    \E_{i}& \text{if } i < j\\
    s_{(\alpha_{i},m+1)}(\E_{i}) &\text{ if } j \leq i < r\\
    t_{\alpha_{i}^{\vee}}(\E_{i}) &\text{ if } i \geq r
\end{cases}
\]

\noindent
and define $$e_{\alpha_{i}}(\gamma) = (\V'_{0}, \E'_{0}, \V'_{1}, \E'_{1}, \cdots, \E'_{r}, \V'_{k+1}).$$

\noindent
 If $m = 0$ then $e_{\alpha_{i}}(\gamma) = 0$. \\
\end{defn}

\noindent

\begin{rem}
\label{typeremark}
It follows from the definitions that the maps $e_{\alpha_{i}}, f_{\alpha_{i}} \hbox{ and } \op{wt}$ define a crystal structure on $\Gamma$. Note as well that if $\gamma$ is a combinatorial gallery then $f_{\alpha_{i}}(\gamma)$ and $e_{\alpha_{i}}(\gamma)$ are combinatorial galleries of the same type (as long as they are not zero). We say that the root operators are type preserving. See also \cite{ls}, Lemma 6.
\end{rem}

\subsection{The Littelmann path model and Lakshmibai Seshadri galleries. Readable galleries.}
\label{pathmodelsection}

Let $\gamma$ be a combinatorial gallery that has each one of its faces contained in the fundamental chamber. We call such galleries \textbf{dominant}. By Theorem 7.1 in \cite{pathsandrootoperators} the crystal  of galleries $\p(\gamma)$ generated by $\gamma$ is isomorphic to the crystal $\B(\mu_{\gamma})$ associated to the irreducible highest weight representation $\li(\mu_{\gamma})$ of $\G^{\vee}$. In its original context \cite{pathsandrootoperators} it is known as a \textsl{Littelmann path model} for the representation  $\li(\mu_{\gamma})$.  We say that a combinatorial gallery $\gamma$ is a \textbf{Littelmann gallery} if there exist indices $i_{1}, \cdots, i_{r}$ such that $e_{\alpha_{i_{1}}}\cdots e_{\alpha_{i_{r}}}(\gamma) =: \gamma^{+}$ is a dominant gallery. If $\mu_{\gamma^{+}} = \mu_{\delta^{+}}$ and $e_{\alpha_{i_{1}}}\cdots e_{\alpha_{i_{r}}}(\gamma) = \gamma^{+}; e_{\alpha_{i_{1}}}\cdots e_{\alpha_{i_{r}}}(\delta) = \delta^{+}$ for two Littelmann galleries $\gamma$ and $\delta$ we say that they are \textbf{equivalent}.  \\

Let $\lambda \in \X^{\vee,+}$ be a dominant coweight and $\gamma_{\lambda}$ a gallery that is a concatenation of fundamental galleries and that has endopoint $\lambda$ (as above). We denote the set of \textbf{combinatorial LS  galleries} (short for Lakshmibai Seshadri galleries) of same type as $\gamma_{\lambda}$ by $\Gamma^{\op{LS}}(\gamma_{\lambda})$. Littelmann galleries generalise LS galleries enormously. In particular, all LS galleries are `Littelmann' - see \cite{pathsandrootoperators}, Section 4. Moreover, this set $\Gamma^{\op{LS}}(\gamma_{\lambda})$ is stable under the root operators and has the structure of a  crystal isomorphic to $\B(\lambda)$. It was proven by Gaussent-Littelmann in \cite{ls} that the resolution in (\ref{resolution}) induces a bijection $\Gamma^{\op{LS}}(\gamma_{\lambda}) \cong \mathcal{Z}(\lambda)$ which was shown to be a crystal isomorphism in \cite{baumanngaussent} by Baumann-Gaussent. We use this heavily in the proof of Theorem \ref{main}. See Definition 18 in \cite{ls} for a geometric definition of LS galleries, and Definition 23 in \cite{ls} for an equivalent combinatorial characterisation that for one skeleton galleries agrees with the original definition by Lakshmibai, Musili, and Seshadri (see for example \cite{LLM}) in the context of standard monomial theory. We will give a combinatorial characterisation of LS galleries of fundamental type in the case $\G^{\vee} = \op{SP}(2n, \mathbb{C})$, omitting therefore the most general definitions.\\

We finish this section with a question. Let $\gamma$ be any combinatorial gallery with each one of its edges contained in the fundamental chamber. Then the map $\Sigma_{\gamma^{f}} \rightarrow \mathcal{G}, [g_{0}, \cdots, g_{r}] \mapsto g_{0}\cdots g_{r}[t^{\mu_{\gamma}^{f}}]$ is still defined.

\begin{question}
Does this map induce a crystal isomorphism $\p(\gamma) \cong \mathcal{Z}(\mu_{\gamma})$? 
\end{question} 

\noindent
This question was answered positively in \cite{knuth} and \cite{jt} for $\G^{\vee} = \op{SL}(n, \mathbb{C})$. In the rest of this paper we do so as well for $\G^{\vee} = \op{SP}(2n,\mathbb{C})$ and $\gamma$ a $\textsl{readable}$ gallery.\\

\begin{defn}
\label{readabledefinition}
A \textbf{readable} gallery is a concatenation of its \textbf{parts}: LS galleries of fundamental type and galleries of the form $(\V_{0}, \E_{0}, \V_{1}, \E_{1}, \V_{2})$ (we call them \textbf{zero lumps}) such that both edges $\E_{0}$ and $\E_{1}$ are contained in the dominant chamber and such that  the endpoint $\V_{2} = 0$ is equal to zero. We denote the set of all readable galleries by $\Gamma^{R},$ and if a combinatorial gallery $\gamma$ is fixed, by $\Gamma(\gamma)^{\R}$ the set of all readable galleries of same type as $\gamma$. 
\end{defn}

\noindent
For $\op{G}^{\vee} = \op{SL}(n, \mathbb{C})$ all galleries are readable; this is due to the well known fact that in this case fundamental coweights are all minuscule.  In the next sections we will describe readable galleries explicitly for $\G^{\vee} = \op{SP}(2n,\mathbb{C})$ and show that they are Littelmann galleries. They are also more general than galleries of type $\gamma_{\lambda}$ for a gallery $\gamma_{\lambda}$ that is a concatenation of fundamental galleries - this means they belong to a larger class of galleries, but not that they contain $\Gamma(\gamma_{\lambda})$. 

\begin{rem}
\label{readablegalleriesarestable}
It follows from Lemma 8 in \cite{ls} that readable galleries are stable under root operators. 
\end{rem}

\section{``Type C'' combinatorics}
\label{typecombinatorics}

\subsection{Symplectic keys and words}
A \textbf{symplectic shape} 
$$\underline{d} = (d_{1}, \cdots, d_{k+1})$$ \noindent is a sequence of natural numbers $d_{i} \leq n$. An \textbf{arrangement of boxes} of shape $\underline{d}$ is an arrangement of $r$ columns of boxes such that column $s$ (read from right to left) has $d_{s}$ boxes. 

\begin{ex}
\label{ch2ex1}
\begin{align*}
\Skew(0:,,,|1:) 
\end{align*}
An arrangement of boxes of symplectic shape (1,1,2,1).
\end{ex}

Consider the ordered alphabet $$\mathcal{C}_{n} = \{1<2< \cdots< n-1< n< \gr{n} < \cdots < \gr{1} \}.$$ A \textbf{symplectic key} of shape $\underline{d}$ is a filling of an arrangement of boxes of symplectic shape $\underline{d}$  with letters of the alphabet $\mathcal{C}_{n}$ in such a way that the entries are strictly increasing along each column.

\begin{ex}
A symplectic key, for $n\geq 5$, of symplectic shape $(1,3,2,1)$.
$$
\Skew(0:\mbox{\tiny{$\overline{1}$}},\mbox{\tiny{$1$}},\mbox{\tiny{$2$}},\mbox{\tiny{$3$}}|1:\mbox{\tiny{$5$}},\mbox{\tiny{$2$}}|2:\mbox{\tiny{$\overline{2}$}})
$$
\end{ex}

\noindent We denote the word monoid on $\mathcal{C}_{n}$ by $\mathcal{W}_{\mathcal{C}_{n}}$. To a word $w = w_{1}\cdots w_{k}$ in $\mathcal{W}_{\mathcal{C}_{n}}$ we associate a symplectic key $\mathscr{K}_{w}$ that consists of only one row of length $k$, and with the boxes filled in from right to left with the letters of $w$ read in turn from left to right. For example, the word $12$ corresponds to the key $\Skew(0:\mbox{\tiny{$2$}}, \mbox{\tiny{$1$}})$.  Denote the set of all symplectic keys associated to words by $\Gamma(\op{wor})$. 

\subsection{Weights and coweights}
Consider $\mathbb{R}^{n}$ with canonical basis $\{\e_{1}, \cdots, \e_{n}\}$ and standard inner product $\<-,-\rr $ (in particular $\< \e_{i}, \e_{j}\rr = \delta_{ij}$). From now on we consider the root datum $(\X, \Phi, \X^{\vee}, \Phi^{\vee})$ that is defined by: 
 
\begin{align*}
\Phi &= \{\pm\e_{i}, \e_{i}\pm \e_{j}\}_{i,j\in \{1, \cdots, n\}}\\
\Phi^{\vee} &= \{\alpha^{\vee}:= \frac{2\alpha}{\<\alpha, \alpha \rr}\}\\
\X &= \{v \in \mathbb{R}^{n}: \<v, \alpha^{\vee}\rr \in \mathbb{\mathbb{Z}}\}\\
\X^{\vee}&=\{v\in \mathbb{R}^{n}: \<\alpha, v\rr \in \mathbb{Z}\}.
\end{align*}

\noindent
Indeed the sets $\X$ and $\X^{\vee}$  are free abelian groups which form a root datum together with the
 pairing $\<-,-\rr$ between them and the subsets $\Phi \subset \X$ and $\Phi^{\vee} \subset \X^{\vee}$. We choose a basis $\Delta \subset \Phi$ given by
$$\Delta = \{\alpha_{i} =\varepsilon_{i}-\varepsilon_{i+1}; \alpha_{n}= \varepsilon_{n}: i \in \{1, \cdots, n-1\} \},$$
hence the set
$$\Delta^{\vee} = \{\alpha^{\vee}_{i} =\varepsilon_{i}-\varepsilon_{i+1}, \alpha^{\vee}_{n}= 2\varepsilon_{n}: i \in \{1, \cdots, n-1\}\}$$
is a basis for $\Phi^{\vee}$.
Then $\X^{\vee}$ has a $\mathbb{Z}$-basis given by $\{\omega_{i}\}_{i \in \{1, \cdots, n\}}$, where 

\begin{align*}
\omega_{i} &= \e_{1}+\cdots + \e_{i}\hbox{    } 1\leq i \leq n.
\end{align*}
\noindent
Then $\G= \op{SO}(2n+1, \mathbb{C})$ and $\op{G}^{\vee} = \op{SP}(2n, \mathbb{C})$. 

\subsection{Symplectic keys associated to readable galleries}
\label{symplectickeys}

\noindent
The aim of this section is to assign a symplectic key to every readable gallery.

\subsubsection{Readable blocks}

\noindent
For a subset $\X \subseteq \mathcal{C}_{n}$, we denote the corresponding subset of barred elements by  $\overline{\X} := \{\bar{x}: x \in \X\}$, where, for $i$ unbarred, $\bar{\bar{i}} = i$. 
\begin{defn}
\label{lsblock}
Let $\mathscr{T}$ be a symplectic key. We call $\mathscr{T}$ an \textbf{LS block} if the arrangement of boxes associated to its type consists of only one box or if there exist
positive integers $k,r,s$ such that $2k +r+s \leq n$, and disjoint sets of positive integers
\begin{align*}
\A &= \{a_{i}: 1\leq i\leq r, a_{1}<\cdots <a_{r}\}\\
\B &= \{b_{i}: 1\leq i\leq s, b_{1}<\cdots <b_{s}\}\\
\Z &= \{z_{i}: 1\leq i\leq k, z_{1}<\cdots <z_{k} \}\\
\T &= \{t_{i}: 1\leq i\leq k, t_{1}<\cdots <t_{k}\}
\end{align*} 
such that $\mathscr{T}$ consists of two columns: the rightmost one (respectively the leftmost one) is the column with entries the ordered elements of the set $\{\ov{\T}, \Z, \A, \ov{\B}\}$ (respectively $\{\ov{\Z}, \T, \A, \ov{\B}\}$), and such that the elements of $\T$ are uniquely characterised by the properties
\begin{align}
\label{admissibilityconditions}
t_{k} &= \op{max}\{t \in \mathcal{C}_{n}: t<z_{k}, t\notin \Z\cup\A\cup\B\}\\
t_{j-1} &= \op{max}\{t \in \mathcal{C}_{n}: t<\op{min}(z_{j-1}, t_{j}), t\notin \Z\cup\A\cup\B\} \hbox{ for } j\leq k. 
\end{align}

We say that $\mathscr{T}$ is a \textbf{zero block} if there exists a non-zero integer $k$ such that $\mathscr{T}$ consists of two columns, both of $k$ boxes; the right-most one is filled in with the ordered letters $1 < \cdots < k$ and the left-most one, with $\bar{k} < \cdots < \bar{1}$. A symplectic key is called a \textbf{readable block} if it is either an LS block or a zero block. A \textbf{readable key} is a concatenation of readable blocks. Now assume that $\underline{d} = (d_{1}, \cdots, d_{k+1})$ is such that $d_{1} \leq \cdots \leq d_{k+1}$. A symplectic key of shape $\underline{d}$ is called an \textbf{LS symplectic key} if the entries are weakly increasing in rows and if it is a concatenation of LS blocks. We denote the set of LS symplectic keys of shape $\underline{d}$ as $\bs{\Gamma}(\underline{d})^{\op{LS}}$. 
\end{defn}

\begin{ex} \label{symplecticblockexample}The symplectic key $\Skew(0:\mbox{\tiny{$1$}},\mbox{\tiny{$2$}}|0:\mbox{\tiny{$\overline{2}$}},\mbox{\tiny{$\overline{1}$}})$ is an LS block, with $$\A = \B = \emptyset, \Z = \{2\}, \T = \{1\}.$$ The symplectic key $\Skew(0:\mbox{\tiny{$1$}},\mbox{\tiny{$\overline{2}$}}|0:\mbox{\tiny{2}},\mbox{\tiny{$\overline{1}$}})$ is not an LS block. The symplectic key $\Skew(0:\mbox{\tiny{$\overline{2}$}}, \mbox{\tiny{$1$}}|0:\mbox{\tiny{$\overline{1}$}},\mbox{\tiny{$2$}})$ is a zero block. 
\end{ex}

\begin{rem}
A pair of columns that form an LS block is sometimes called a pair of admissible columns. The original definition of admissible columns was given by DeConcini in \cite{deconcini}, using a slightly diferent convention than Kashiwara and Nakayima's (which is the one we use here). The map, given by Lecouvey, that translates the two can be found in \cite{lecouvey} at the end of Section 2.2. 
\end{rem}

To a readable block $\mathscr{T}$ we assign a gallery 
$\gamma_{\mathscr{T}}$ as follows. If $\mathscr{T}$ consists of only one box filled in with the letter $l \in \mathcal{C}_{n}$, then we define $ \V^{\mathscr{T}}_{0}, \V^{\mathscr{T}}_{1} = \e_{l}$,  $\E^{\mathscr{T}}_{0}:= \{t \V^{\mathscr{T}}_{1}, t \in [0,1]\}$, and 

$$\gamma_{\mathscr{T}}:= \{\V^{\mathscr{T}}_{0}, \E^{\mathscr{T}}_{0}, \V^{\mathscr{T}}_{1}\}.$$

\noindent
 If not, then its columns are filled with the letters $l^{1}_{1} < \cdots < l^{1}_{d}$ and $l^{2}_{1} < \cdots < l^{2}_{d}$ respectively. We then define 

\begin{align*}
\V^{\mathscr{T}}_{1} &=  \frac{1}{2}(\e_{l^{1}_{1}} + \cdots + \e_{l^{1}_{d}})\\
\V^{\mathscr{T}}_{2} &= \e_{l^{1}_{1}} + \cdots + \e_{l^{1}_{d}} +  \e_{l^{2}_{1}} + \cdots + \e_{l^{2}_{d}}\\
\E^{\mathscr{T}}_{1}&= \hbox{line segment joining } \V_{1} \hbox{ and } \V_{2}.
\end{align*}
and $$\gamma_{\mathscr{T}} = (\V^{\mathscr{T}}_{0}, \E^{\mathscr{T}}_{0}, \V^{\mathscr{T}}_{1}, \E^{\mathscr{T}}_{1}).$$

\begin{ex}
Let $n = 2$ and $\gamma = (\V_{0}, \E_{0}, \V_{1}, \E_{1},\V_{2})$ where $\V_{0} = 0, \V_{1} = \frac{1}{2}(\varepsilon_{1}+\varepsilon_{2}), \V_{2} = \e_{1}+\e_{2}$ and the edges are the line segments joining the vertices in order. Below is a picture of the associated gallery $\gamma_{\mathscr{K}}$ to the symplectic key $\mathscr{K}$.
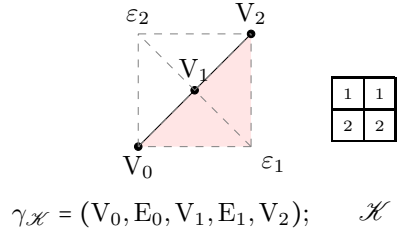
\begin{center}
\begin{tikzpicture}[scale = 1.5]
\fill[red!20!, opacity = 0.5] (0,0) -- (1,0) -- (1,1) -- cycle;
\node[draw,circle,inner sep=1pt,fill] at (0,0) {};
\node[below] at (0,0){$\V_{0}$};
\node[above] at (1,1){$\V_{2}$};
\draw[-] (0,0) -- (0.5,0.5);
\node[draw,circle,inner sep=1pt,fill] at (1,1) {};
\node[draw,circle,inner sep=1pt,fill] at (0.5,0.5) {};
\node[above] at (0.5,0.5){$\V_{1}$};
\draw[-] (0.5,0.5) -- (1,1);
\node[below] at (1.2,-0.01) {\mbox{$\varepsilon_{1}$}};
\node[above] at (0,1){$\varepsilon_{2}$};
\draw[dashed, gray, very thin] (1,0) --(0,0);
\draw[dashed, gray, very thin] (0,0) -- (1,1);
\draw[dashed, gray, very thin] (0,0) --(0,1);
\draw[dashed,gray, very thin] (1,0) --(0,1);
\draw[dashed,gray, very thin] (1,0) --(1,1);
\draw[dashed,gray, very thin] (0,1) --(1,1);
\node[below] at (2,0.7){\Skew(0:\mbox{\tiny{$1$}},\mbox{\tiny{$1$}}|0:\mbox{\tiny{$2$}},\mbox{\tiny{$2$}})};
\node[below] at (0.2,-0.4){$\gamma_{\mathscr{K}} = (\V_{0}, \E_{0}, \V_{1}, \E_{1}, \V_{2});$};
\node[below] at (2.1,-0.4){$\mathscr{K}$};
\end{tikzpicture}
\end{center}
\end{ex}

To a readable key we associate the concatenation of the galleries of each of the readable blocks that it is a concatenation of (from right to left). Given a symplectic shape $\underline{d}$, we will denote the set of all readable keys of shape $\underline{d}$ by $\bs{\Gamma}(\underline{d})^{\R}$. (This set may be empty.)  Let $\underline{d}$ be a shape such that $\bs{\Gamma}(\underline{d})^{\R} \neq 0$. Then it must have the form

\begin{align*}
\underline{d} = (\underline{d}^{l_{1}}, \cdots, \underline{d}^{l_{m}})
\end{align*}
\noindent
where $\underline{d}^{l_{i}} = l_{i}, l_{i}$ for $l_{i} \geq 2$ and $\underline{d}^{l_{i}} = 1$ if  $l_{i} = 1$. For instance, in Example \ref{symplecticblockexample}, all symplectic keys have shape $(2,2)$. To such a shape $\underline{d}$ we associate the dominant coweight 

$$\lambda_{\underline{d}} = \omega_{l_{1}} + \cdots + \omega_{l_{m}}.$$
\noindent For example, to the shape $(2,2)$ is associated the coweight $\omega_{2}$. The following proposition follows from Lemma 2 in \cite{onesk}.
\begin{prop}
\label{symplecticshapeandtype}
 The map 
\begin{align*}
\underset{\underline{d}}{\bigcup}\bs{\Gamma}^{\R}(\underline{d}) & \overset{1:1}{\longrightarrow} \Gamma^{\R}\\
\mathscr{T} &\mapsto \gamma_{\mathscr{T}}
\end{align*}
\noindent
is well defined and is a bijection. Moreover, if $d_{1} \leq \cdots \leq d_{k+1}$ then this map induces a bijection 

\begin{align*}
\bs{\Gamma}^{\op{LS}}(\underline{d}) \overset{1:1}{\longleftrightarrow} \Gamma^{\op{LS}}(\gamma_{\omega_{l_1}} * \cdots * \gamma_{\omega_{l_m}}).
\end{align*}
\end{prop}

\begin{rem}
\label{moreshapeandtype}
Zero lumps are not necessarily of fundamental type: this follows from Lemma 2 in \cite{onesk} for a zero lump with $k$ uneven in the above description. This is why readable galleries are not necessarily of the same type as  a concatenation of fundamental galleries. This also means that there can be two readable keys of the same shape but such that their associated galleries are not of the same type! For example, take $n > 3$. Then the key $\mathscr{T} = \Skew(0:\hbox{\tiny{1}},\hbox{\tiny{1}}|0:\hbox{\tiny{2}},\hbox{\tiny{2}}|0:\hbox{\tiny{3}},\hbox{\tiny{3}})$ is LS and $\gamma_{\mathscr{T}}$ is of fundamental type $\gamma_{\omega_{3}}$. The key
$\mathscr{K} = \Skew(0:\hbox{\tiny{$\bar1$}},\hbox{\tiny{1}}|0:\hbox{\tiny{$\bar 2$}},\hbox{\tiny{2}}|0:\hbox{\tiny{$\bar 3$}},\hbox{\tiny{3}})$ is a zero block. Its associated gallery, $\gamma_{\mathscr{K}}$, is not of fundamental type (see Lemma 2 in \cite{onesk}).
\end{rem}

\section{The word of a readable gallery}
The \textbf{word} of a block $\mathscr{B} = \C_{l}\C_{r}$ ($\C_{l}$ is the left column; $\C_{r}$ the right) is obtained by reading first the unbarred entries in $\C_{r}$ and then the barred entries in $\C_{l}$. We denote it by $w(\mathscr{B}) \in \mathcal{W}_{\mathcal{C}_{n}}$. For an LS block this is the word of the associated single admissible column defined by Kashiwara and Nakashima - see \cite{lecouvey}, Example 2.2.6. 

\begin{defn}
Let $\gamma_{\mathscr{K}}$ be a readable gallery associated to the key $\mathscr{K}$, which we may write as a concatenation of blocks $$\mathscr{K} = \mathscr{B}_{1} \cdots \mathscr{B}_{k}.$$ The \textbf{word} of $\gamma_{\mathscr{K}}$ (or of $\mathscr{K}$) is $w(\mathscr{B}_{k})\cdots w(\mathscr{B}_{1})$. We denote it by $w(\gamma_{\mathscr{K}})$ (or $w(\mathscr{K})$).
\end{defn}

\begin{ex}
Let $$\mathscr{B}_{1} = \Skew(0:\mbox{\tiny{$1$}},\mbox{\tiny{$2$}}|0:\mbox{\tiny{$\overline{2}$}},\mbox{\tiny{$\overline{1}$}}),\hbox{                       }  \mathscr{B}_{2} = \Skew(0:\mbox{\tiny{$1$}}),$$ and $$\mathscr{K} = \mathscr{B}_{1}\mathscr{B}_{2} = \Skew(0:\mbox{\tiny{$1$}},\mbox{\tiny{2}}, \mbox{\tiny{$1$}}|0:\mbox{\tiny{$\overline{2}$}},\mbox{\tiny{$\overline{1}$}}).$$ Then $w(\mathscr{B}_{1}) = 2 \bar{2}, w(\mathscr{B}_{2}) = 1$, and $w(\mathscr{K}) = 12 \bar{2}$.
\end{ex}

We have the following result about words of readable galleries, which we prove in Section \ref{countingpositivecrossings}. We will use it in Section \ref{main}. It is in this sense that such galleries are called \textit{readable}. 
\begin{prop}
\label{densewordreading}
Let $\gamma$ and $\nu$ be combinatorial galleries and $\mathscr{K}$ be a readable key. Then 
\begin{align*}
\overline{\pi(\C_{\gamma*\gamma_{w(\mathscr{K})}*\nu})} =\overline{\pi'(\C_{\gamma*\gamma_{\mathscr{K}}*\nu})}.
\end{align*}
\end{prop}

\subsection{Word galleries}
We associate a (readable!) gallery $\gamma_{w}$ of the same type as $\underbrace{\gamma_{\omega_{1}}*\cdot \cdot \cdot *\gamma_{\omega_{1}}}_{m \hbox{  } times }$ to a word $w \in \mathcal{W}_{\mathcal{C}_{n}}$ of length $m$ - it is the gallery $\gamma_{\mathscr{K}_{w}}$ associated to the readable key $\mathscr{K}_{w}$. We denote the set of word galleries in this case by $\Gamma_{\mathcal{W}_{\mathcal{C}_{n}}}$. Below we recall the crystal structure on the set $\mathcal{W}_{\mathcal{C}_{n}}$ as described by Kashiwara and Nakashima in \cite{kashiwaranakashima}, Proposition 2.1.1. The set of words $\mathcal{W}_{\mathcal{C}_{n}}$, just like the set $\mathcal{W}_{n}$, is in one-to-one correspondence with the set of vertices of the crystal of the representation $\bigotimes_{l \in \mathbb{Z}^{\geq 0}}\V_{n}^{\otimes l}$, where $\V_{n}$ is the natural representation $\li(\omega_{1})$ and hence inherits its crystal structure. Proposition \ref{wordcrystalispathcrystal} says that this structure is compatible with the crystal structure defined on galleries in Section \ref{crystalsubsection}.

\begin{defn}
\label{describewordcrystal}
Let $w \in \mathcal{C}_{n}$ be a word and $i \in \{1, \cdots, n\}$. To apply the root operators $e_{\alpha_{i}}$ and $f_{\alpha_{i}}$ to $w$ one first obtains a word consisting of letters in the alphabet $\{+, -, \emptyset\}$. The word will be obtained from $w$ by replacing every occurence of $i$ or $\overline{i+1}$ by $(+)$, every occurence of $i+1$ or $\overline{i}$ by $(-)$ and all other letters by $\emptyset$. This word $s(w)$ is sometimes called the \textsl{i-signature} of w. Erase all symbols $\emptyset$ and then all subwords of the form $+-$. Repeat this process until the $i$-signature $s(w)$ of $w$ has been reduced to a word of the form 

$$s(w)' = (-)^r(+)^{s}.$$

To apply $f_{\alpha_{i}}$ (respectively $e_{\alpha_{i}}$) to $w$, change the letter whose tag corresponds to the rightmost $(-)$ (respectively to the leftmost $(+)$) from $i+1$ to $i$ and from $\overline{i}$ to $\overline{i+1}$ (repectively from $i$ to $i+1$ and from $\overline{i+1}$ to $\overline{i}$).

\end{defn}

\begin{prop}
\label{wordcrystalispathcrystal}
The crystal structure on words from Definition \ref{describewordcrystal} coincides with the one induced from Definition \ref{definitionofrootoperators}. 
\end{prop}

\noindent For a proof, see Section 13 of \cite{placticalgebra}. It also follows directly from the definitions.

\subsection{Word Reading is a Crystal Morphism}
\label{wordreadingsection}
This subsection is the `symplectic' version of Proposition 2.5 in \cite{jt}.
 Since the root operators are type preserving (see \ref{definitionofrootoperators}), the set of words $\mathcal{W}_{\mathcal{C}_{n}}$ is naturally endowed with a crystal structure. The following proposition will be useful in Section \ref{main}. This result was shown for LS blocks by Kashiwara and Nakashima in \cite{kashiwaranakashima}, Proposition 4.3.2. They show that word reading induces an isomorphism of crystals from $\B(\omega_{k})$ onto the subcrystal of $\bigotimes_{l \in \mathbb{Z}_{\geq 1}}\B(\omega_{1})^{\otimes l}$ generated by the tensor product $\Skew(0:\hbox{\tiny{k}})\otimes \cdots \otimes \Skew(0:\hbox{\tiny{1}})$. We show that for readable galleries the proof is reduced to this case. 

\begin{prop}
\label{wordreading}
The map
\begin{align*}
\Gamma^{\op{R}}&\overset{w}{\longrightarrow}\Gamma_{\mathcal{W}_{\mathcal{C}_{n}}}\\
\gamma_{\mathscr{K}} &\mapsto \gamma_{w(\mathscr{K})}
\end{align*}
is a crystal morphism. 
\end{prop}

\begin{proof}
Let $\gamma$ be a readable gallery and let $$\gamma_{\mathscr{B}} = (\V^{\mathscr{B}}_{0}, \E^{\mathscr{B}}_{0}, \V^{\mathscr{B}}_{1}, \E^{\mathscr{B}}_{1}, \V^{\mathscr{B}}_{2})$$ \noindent be one of its parts, associated to the readable block $\mathscr{B}$; we write 
$$\gamma_{\mathscr{K}_{w(\mathscr{B})}} = (\V^{\mathscr{K}_{w(\mathscr{B})}}_{0},\E^{\mathscr{K}_{w(\mathscr{B})}}_{0}, \cdots, \V^{\mathscr{K}_{w(\mathscr{B})}}_{r+s}).$$ If 

$$w(\mathscr{B}) = g_{1}\cdots g_{s}\overline{h_{r}}\cdots \overline{h_{1}}$$ for $g_{i}$ and $h_{i}$ unbarred, then $\V^{\mathscr{K}_{w(\mathscr{B})}}_{j} = \sum_{i= 1}^{j}\varepsilon_{x_{i}}$, where $x_{i} = g_{i}$ for $1 \leq i \leq s$ and $x_{s+ i} = \bar{h}_{i}$ for $1 \leq i \leq r$. Let

\begin{align*}
 h(j)&= \<\alpha, \V^{\mathscr{B}}_{j}\rr \\
 h'(j)&= \<\alpha, \V^{\mathscr{K}_{w(\mathscr{B})}}_{j}\rr.
 \end{align*}

\noindent 
  Then there exist $d_{1}\leq s, s< d_{2} \leq s+r$ such that  
\[
    h'(j)= 
\begin{cases}
    h(0)& \text{ for } 0\leq j < d_{1} \\
    h(1)& \text{ for } d_{1} \leq j < d_{2}\\
    h(2)& \text{ for } d_{2} \leq j \leq r+s+1 
\end{cases}
.\]
From this we conclude that it is enough to consider readable blocks. As mentioned previously, this was shown in \cite{kashiwaranakashima} for LS blocks. Hence let $\mathscr{L}$ be a zero lump - it has word $w(\mathscr{L}) = 1\cdots k \overline{k}\cdots \overline{1}$ - and let $\alpha_{i}$ be a simple root. Then, since the galleries associated to $\mathscr{L}$ and $w(\mathscr{L})$ are both dominant, $f_{\alpha_{i}}(\mathscr{L}) = e_{\alpha_{i}}(\mathscr{L}) = f_{\alpha_{i}}(w(\mathscr{L})) = e_{\alpha_{i}}(w(\mathscr{L})) = 0$. 
\end{proof}

\begin{ex}
Let $ n = 2$ and $\mathscr{B}$ be the readable block
$
\Skew(0:\mbox{\tiny{$1$}},\mbox{\tiny{$2$}}|0:\mbox{\tiny{$\overline{2}$}}, \mbox{\tiny{$\overline{1}$}})
.$
Then 
$
w(\mathscr{B}) = 2\ov{2}
.$ To calculate $f_{\alpha_{1}}(\mathscr{B})$, note that $m_{\alpha_{1}} = -1$, j = 1, r = 2, 
hence $f_{\alpha_1}(\mathscr{B}) = \Skew(0:\mbox{\tiny{$2$}},\mbox{\tiny{$2$}}|0:\mbox{\tiny{$\overline{1}$}}, \mbox{\tiny{$\overline{1}$}})$. Similarly, $f_{\alpha_{1}}(w(\mathscr{B})) = 2\ov{1} = w(f_{1}(\mathscr{B}))$.
\end{ex}

\subsection{Readable galleries are Littelmann galleries}

We begin with a lemma. 

\begin{lem}
\label{dlem}
Let $\gamma_{\mathscr{K}}$ be a readable gallery. Then $\gamma_{\mathscr{K}}$ is dominant if and only if $\gamma_{w(\mathscr{K})}$ is dominant. 
\end{lem}
\begin{proof}
Since the entries in the columns symplectic keys are strictly increasing, it follows from the definition of word reading that if $\gamma$ is a dominant readable gallery then $w(\gamma)$ is also dominant. Now let $\gamma$ be a non-dominant readable gallery. Then there is a readable block $\mathscr{B}= \C_{l}\C_{r}$ such that $\gamma = \eta_{1} * \gamma_{\mathscr{B}} * \eta_{2}$ with $\eta_{1}$ dominant and $\eta_{1} * \gamma_{\mathscr{B}}$ not dominant. This block can't be a zero lump (they are dominant) - so it must be LS. Let $\A, \B, \Z, \T$ be the sets described in Definition \ref{lsblock}. The entries of $\C_{r}$ are the letters in $\A \cup \Z \cup \overline{\B} \cup \overline{\T}$ and the entries of $\C_{l}$ are the letters in $\A \cup \T \cup \overline{\B} \cup \overline{\Z}$. Now, $\mu_{\eta_{1} * \gamma_{\mathscr{B}}}$ may be dominant or not. If it is, then, since $\mu_{\gamma_{w(\eta_{1} * \gamma_{\mathscr{B}})}} = \mu_{\mu_{\eta_{1} * \gamma_{\mathscr{B}}}}$, the word gallery $\gamma_{w(\eta_{1} * \gamma_{\mathscr{B}})}$ is not dominant, and this implies that $\gamma_{w{\mathscr{K}}}$ is not dominant either. Now assume that 
$$\mu_{\eta_{1} * \gamma_{\mathscr{B}}} = \mu_{\eta_{1}} + \underset{a \in \A}{\sum}\e_{a} - \underset{b \in \B}{\sum}\e_{b}$$

\noindent \textit{is} dominant, but that the gallery $\eta_{1} * \gamma_{\mathscr{B}}$ is not. The last three vertices of this  gallery are 

\begin{align}
\label{firstvertex}
\V_{l-1}&= \mu_{\eta_{1}} \in \C^{+}\\
\label{secondvertex}
\V_{l}&= \mu_{\eta_{1}} + \frac{1}{2}(\underset{a \in \A}{\sum}\e_{a}+ \underset{z \in \Z}{\sum}\e_{z} - \underset{b \in \B}{\sum}\e_{b} - \underset{t \in \T}{\sum}\e_{t}) \notin  \C^{+}\\
\label{thirdvertex}
\V_{l+1}&= \mu_{\eta_{1}} + \underset{a \in \A}{\sum}\e_{a} - \underset{b \in \B}{\sum}\e_{b} \in \C^{+}.
\end{align}

\noindent Let $d_{1} < \cdots < d_{r+k}$ be the ordered elements of $\A \cup \Z$ and let $f_{1} < \cdots < f_{s+k}$ be the ordered elements of $\B \cup \Z$. We have

\begin{align*}
w(\mathscr{B}) = d_{1} \cdots d_{r+k} \bar{f}_{s+k} \cdots \bar{f}_{1}.
\end{align*}

\noindent We claim that the weight 
$$\mu_{\eta_{1}} + \overset{r+k}{\underset{i = 1}{\sum}}\e_{d_{i}} = \mu_{\eta_{1}} + \underset{a \in \A}{\sum}\e_{a} +  \underset{z \in \Z}{\sum}\e_{z},$$

\noindent which is the endpoint of $\eta * \gamma_{d_{1} \cdots d_{r+k}}$ and therefore a vertex of $\eta* \gamma_{w(\mathscr{B})}$, is not dominant. To see this, assume otherwise:

\begin{align*}
 \mu_{\eta_{1}} + \underset{a \in \A}{\sum}\e_{a} +  \underset{z \in \Z}{\sum}\e_{z} \in \C^{+}.
\end{align*}

\noindent Since the dominant Weyl chamber $\C^{+}$ is convex, this means that the line segment that joins $\mu_{\eta_{1}}$ and $ \mu_{\eta_{1}} + \underset{a \in \A}{\sum}\e_{a} +  \underset{z \in \Z}{\sum}\e_{z}$ is contained in $\C^{+}$, in particular the point

\begin{align*}
 \mu_{\eta_{1}} + \frac{1}{2}(\underset{a \in \A}{\sum}\e_{a} +  \underset{z \in \Z}{\sum}\e_{z}) \in \C^{+}
\end{align*}

\noindent belongs to the dominant Weyl chamber. The dominant Weyl chamber has, in this case, the following description in the coordinates $\e_{1}, \cdots , \e_{n}$:

\begin{align*}
\C^{+} = \{\overset{n}{\underset{i = 1}{\sum}}p_{i}\e_{i}: p_{i} \in \mathbb{R}_{\geq 0} \hbox{   }\& \hbox{   } p_{1} \geq \cdots \geq p_{n} \}.
\end{align*}

\noindent Write 

\begin{align*}
\mu_{\eta_{1}} = \overset{n}{\underset{i = 1}{\sum}}q_{i}\e_{i}
\end{align*}

\noindent We will now show that $ \mu_{\eta_{1}} + \frac{1}{2}(\underset{a \in \A}{\sum}\e_{a}+ \underset{z \in \Z}{\sum}\e_{z} - \underset{b \in \B}{\sum}\e_{b} - \underset{t \in \T}{\sum}\e_{t}) \in  \C^{+}$. This would contradict our assumption and therefore complete the proof.\\

For every $i \in \{1, \cdots, r\}$, we have $t_{i}<z_{i}< j$ for every $j \in \{1, \cdots, n\}$ such that $t_{i}< j$. Since
$\mu_{\eta_{1}} + \frac{1}{2}(\underset{a \in \A}{\sum}\e_{a} +  \underset{z \in \Z}{\sum}\e_{z}) \in \C^{+}$, we know therefore that 

\begin{align*}
q_{j} \leq q_{z_{i}} + \frac{1}{2} \leq q_{t_{i}},
\end{align*}

\noindent which implies, since $q_{t_{i}} \in \mathbb{Z}$, that

\begin{align*}
q_{j} \leq q_{z_{i}} + \frac{1}{2} \leq q_{t_{i}} - \frac{1}{2}.
\end{align*}

\noindent Now let $b \in \B$, and let $j \in \{1, \cdots, n\}$ such that $b < j$. Since (cf. (\ref{thirdvertex})),

\begin{align*}
\mu_{\eta_{1}} + \underset{a \in \A}{\sum}\e_{a} - \underset{b \in \B}{\sum}\e_{b} \in \C^{+}, 
\end{align*}

\noindent if $j \in (\Z \cup \T)^{c}$, then this implies

\begin{align*}
q_{j} \leq q_{b} - \frac{1}{2} . 
\end{align*}

\noindent If $j \in \Z \cup \T$ then, as before, by the definition of an LS block we may assume that $j = t \in \T$. But this means $q_{t} \leq q_{b}$, therefore $q_{t} - \frac{1}{2} \leq q_{b} - \frac{1}{2}$. All of these arguments imply

\begin{align*}
\mu_{\eta_{1}} + \frac{1}{2}(\underset{a \in \A}{\sum}\e_{a}+ \underset{z \in \Z}{\sum}\e_{z} - \underset{b \in \B}{\sum}\e_{b} - \underset{t \in \T}{\sum}\e_{t}) \in  \C^{+},
\end{align*}

\noindent which contradicts (\ref{secondvertex}).

\end{proof}

As in Chapter 2 we have the following lemma.

\begin{lem}
\label{symplecticd}
A readable gallery $\nu$ is dominant if and only if $e_{\alpha_{i}}(\nu) = 0$ for all $i \in \{1, \cdots, n\}$. 
\end{lem}

\begin{proof}
Notice that for a word $w \in \mathcal{W}_{\mathcal{C}_{n}}$ and $\alpha_{i}$ a simple root, $e_{\alpha_{i}}(\mathscr{K}_{w}) = 0$ means that to the right of each $i+1$ in $\mathscr{K}_{w}$ there is at least one $i$ which has not been cancelled out in the tagging and subword extraction process described in Definition \ref{describewordcrystal}. This is equivalent to the gallery $\gamma_{w}$ being dominant. Lemma \ref{dlem} and Proposition \ref{wordreading} imply the desired result.
\end{proof}

\begin{prop}
\label{rightset}
Every readable gallery is a Littelmann gallery. 
\end{prop}

\begin{proof}
Let $\V_{n}$ be the vector representation of $\op{SP}(2n, \mathbb{C})$. Then the crystal of words $\mathcal{W}_{\mathcal{C}_{n}}$ is isomorphic to the crystal associated to $\T(\V_{n}) = \bigoplus_{l\in \mathbb{Z}_{\geq 1}} \V_{n}^{\otimes l}$, see for example Section 2.1 in \cite{lecouvey}. Now let $\gamma$ be any readable gallery. Then there exist indices $i_{1}, \cdots, i_{r}$ such that $e_{\alpha_{i_{r}}}\cdots e_{\alpha_{i_{1}}}( \gamma_{w(\gamma)})$ is a highest weight vertex, hence dominant, by Lemma \ref{symplecticd}. Since word reading is a morphism of crystals by Proposition \ref{wordreading}, $\gamma_{w(e_{\alpha_{i_{r}}}\cdots e_{\alpha_{i_{1}}}(\gamma))} = e_{\alpha_{i_{r}}}\cdots e_{\alpha_{i_{1}}}( \gamma_{w(\gamma)})$. It follows from Lemma \ref{dlem} that $e_{\alpha_{i_{r}}}\cdots e_{\alpha_{i_{1}}}(\gamma)$  is dominant. 
\end{proof}

\begin{defn}
\label{symplecticrels}
The \textbf{symplectic plactic monoid} $\mathcal{P}_{\mathcal{C}_{n}}$ is the quotient of the word monoid $\mathcal{W}_{\mathcal{C}_{n}}$ by the ideal generated by the following relations
\begin{itemize}
\item[(R1)] 
\label{r1}
For $z \neq \overline{x}$:
$$
\begin{aligned}
y\hbox{ }x\hbox{ }z \equiv y\hbox{ }z\hbox{ }x &&\text{    for    }& x\leq y < z \\
x\hbox{ }z\hbox{ }y \equiv z\hbox{ }x\hbox{ }y &&\text{  for   }& x<y\leq z
\end{aligned}
$$

\item[(R2)] \label{r2}
For   $1<x\leq n$ and $x\leq y \leq \bar x$:
$$
\begin{aligned}
y\hbox{ }\gr{x-1}\hbox{ }x-1 &\equiv y\hbox{ }x\hbox{ }\bar x\\
\gr{x-1}\hbox{ }x-1\hbox{ }y& \equiv x\hbox{ }\bar x\hbox{ }y
\end{aligned}
$$

\item[(R3)] 
\label{r3}
$\hbox{  }$
$$a_{1}\cdots a_{r}\hbox{ }z\hbox{ }\hbox{ }\bar z\hbox{ }\bar b_{s} \cdots \bar b_{1} \equiv a_{1}\cdots a_{r}\hbox{ }\bar b_{s} \cdots \bar b_{1}$$
for $a_{i}, b_{i} \in \{1, \cdots, n\}, i \in \{1, \cdots, \op{max\{s,r\}}\}$, such that $a_{1}< \cdots < a_{r}, b_{1}<\cdots< b_{s}$, and such that the left hand side of the above expression is not the word of an LS block. 
\end{itemize}

\noindent If two words $w_{1}, w_{2} \in \mathcal{W}_{\mathcal{C}_{n}}$ are representatives of the same class in $\mathcal{W}_{\mathcal{C}_{n}}$ we say they are \textbf{ symplectic plactic equivalent}. 
\end{defn}

\begin{ex}
\begin{align*}
12\bar 2 \bar 1 &\sim 1\bar 1 \sim \emptyset \\
112 &\sim 121
\end{align*}
\end{ex}

\begin{rem}
Relations (R1) are the Knuth relations in type A, while relation (R3) may be understood as the general relation which specialises to $1\bar 1 \cong \emptyset$. Note that the gallery $\gamma_{w}$ associated to $w = 1 \bar{1}$ is a zero lump. This definition of the symplectic plactic monoid is the same as Definition 3.1.1 in \cite{lecouvey} except for relation (R3). The equivalence between the relation (R3) above and the one in \cite{lecouvey} is given in the Appendix. 
\end{rem}

\normalfont
The following Theorem is due to Lecouvey and it is proven in \cite{lecouvey}. 

\begin{thm}
\label{lecouvey}
Two words $w_{1}, w_{2} \in \mathcal{W}_{\mathcal{C}_{n}}$ are symplectic plactic equivalent if and only if their associated galleries $\gamma_{w_{1}}$ and $\gamma_{w_{2}}$ are equivalent. 
\end{thm}

\normalfont
\noindent Together with the results we have recollected in this section, Theorem \ref{lecouvey} implies the following proposition. 

\begin{prop}
\label{typecwordandgalleryequivalence}
Two readable galleries $\gamma$ and $\nu$ are equivalent if and only if the words $w(\gamma)$ and $w(\nu)$ are symplectic plactic equivalent.
\end{prop}

\begin{proof}
Two readable galleries $\gamma$ and $\nu$ are equivalent if and only if, by definition, there exist indices $i_{1}, \cdots, i_{r}$ such that the galleries $e_{\alpha_{i_{1}}}\cdots e_{\alpha_{i_{r}}}(\gamma)$ and $e_{\alpha_{i_{1}}}\cdots e_{\alpha_{i_{r}}}(\nu)$ are both dominant and have the same endpoint. By Lemma \ref{dlem} this is true if and only if $\gamma_{w(e_{\alpha_{i_{1}}}\cdots e_{\alpha_{i_{r}}}(\gamma))}$ and  $\gamma_{w(e_{\alpha_{i_{1}}}\cdots e_{\alpha_{i_{r}}}(\nu))}$ are also both dominant with the same endpoint. By Proposition \ref{wordreading}, we have $w(e_{\alpha_{i_{1}}}\cdots e_{\alpha_{i_{r}}}(\delta)) =e_{\alpha_{i_{1}}}\cdots e_{\alpha_{i_{r}}}(w(\gamma_{\delta}))$ for any readable gallery $\delta$. This means that the previous sequence of equivalences is also equivalent to $\gamma_{w(\gamma)} \sim \gamma_{w(\nu)}$ which by Theorem \ref{lecouvey} is equivalent to $w(\gamma) \equiv w(\nu)$.
\end{proof}

\normalfont
The following theorem is originally due to Kashiwara and Nakashima (see \cite{kashiwaranakashima}). For this particular formulation, see Proposition 3.1.2 in \cite{lecouvey}. 

\begin{thm}
\label{existenceoftheuniquesymplecticlskey}
For each word $w$ in $\mathcal{W}_{\mathcal{C}_{n}}$ there exists a unique symplectic LS key $\mathscr{T}$ such that $w \sim w(\mathscr{T})$.
\end{thm}

The following proposition will be proven in Section \ref{countingpositivecrossings}. It will play a fundamental role in the proof of Theorem \ref{main}. 
\begin{prop}
\label{denseplacticrelations}
Let $\gamma$ and $\nu$ be combinatorial galleries and let $w_{1}\in \mathcal{W}_{\mathcal{C}_{n}}$ be two plactic equivalent words. Then 
\begin{align*}
\overline{\pi(\C_{\gamma*\gamma_{w_{1}}*\nu})} = \overline{\pi'(\C_{\gamma*\gamma_{w_{2}}*\nu})}
\end{align*}
\end{prop}


\section{Readable Galleries and MV cyles}
\label{mainsection}
\normalfont
The following result (which holds in higher generality) wa shown by Gaussent-Littelmann (a. is an instance of Theorem C  in \cite{ls}) and Baumann-Gaussent (b. is an instance of Theorem 5.8 in \cite{baumanngaussent}).

\begin{thm}
\label{baumanngaussentlittelmannc}
Let $\underline{d} = (d_{1}, \cdots, d_{r})$ be a symplectic shape such that $\bs{\Gamma}(\underline{d})^{\op{LS}} \neq \emptyset$ and consider the desingularization $\pi: \Sigma_{\underline{d}} \rightarrow \X_{\lambda_{\underline{d}}}$. 
\begin{itemize}
\item[a.] If $\delta \in \bs{\Gamma}(\underline{d})^{\op{LS}}$ is a symplectic LS key, the closure $\overline{\pi(\C_{\delta})}$ is an MV cycle in $\mathcal{Z}(\lambda_{\underline{d}})$. This induces a bijection 
$\bs{\Gamma}(\underline{d})^{\op{LS}} \overset{\varphi_{\underline{d}}}{\longrightarrow} \mathcal{Z}(\lambda_{\underline{d}})$. 
\item[b.] The bijection $\varphi_{\underline{d}}$ is an isomorphism of crystals. 
\end{itemize}
\end{thm}

Given a readable gallery $\gamma$ and a dominant coweight $\lambda \in \X^{\vee, +}$, let 

\begin{align*}
n^{\lambda}_{\gamma^{f}} = \# \{\nu \in \Gamma^{\op{dom}} \cap \Gamma(\gamma^{f}): \mu_{\nu} = \lambda\},
\end{align*}

and let 

\begin{align*}
\X^{\vee,+}_{\gamma^{f}} = \{\lambda \in \X^{\vee, +}: n^{\lambda}_{\gamma^{f}} \neq 0\}.
\end{align*}

\begin{thm}
\label{main}
Let $\delta \in \Gamma(\gamma^{f})^{\op{R}}$ be a readable gallery, and $(\Sigma_{\gamma^{f}},\pi)$ the corresponding Bott-Samelson variety together with its map $\pi$ to the affine Grassmannian as in (\ref{resolutionremark2}). Let $\delta^{+}$ be the gallery that is the highest weight vertex in $\op{Conn}(\delta)$. This gallery is dominant and readable by Lemma \ref{symplecticd} and Remark \ref{readablegalleriesarestable}, respectively.  Then
\begin{enumerate}
\item[a.] $\overline{\pi(\C_{\delta})}$ is an MV cycle in $\mathcal{Z}(\mu_{\delta^{+}})_{\mu_{\delta}}$.
\item[b.] The map
\begin{align*}
\Gamma(\gamma^{f})^{\op{R}} &\overset{\varphi_{\gamma^{f}}}{\longrightarrow} \underset{\delta \in \Gamma(\gamma^{f})^{\op{R}}}{\bigoplus}\mathcal{Z}(\mu_{\delta^{+}})\\
\delta &\mapsto \overline{\pi(\C_{\delta})}
\end{align*}
\noindent is a surjective morphism of crystals. The direct sum on the right-hand side is a direct sum of abstract crystals. 
\item[c.] If $\C$ is a connected component of $\Gamma(\gamma^{f})^{\op{R}}$, then $\varphi |_{\C}$ is an isomorphism onto its image. 
\item[d.] The number of connected components $\C$ of $\Gamma^{\R}(\gamma^{f})$ such that $\varphi_{\gamma^{f}}(\C) = \mathcal{Z}(\lambda)$ is equal to $n^{\lambda}_{\gamma^{f}}$.
\item[e.] Given an MV cycle $\Z \in \mathcal{Z}(\lambda)_{\mu}$, the fibre $\varphi_{\gamma^{f}}^{-1}(\Z)$ is given by 
\begin{align*}
\varphi_{\gamma^{f}}^{-1}(\Z) = \{ \delta \in \Gamma^{\R}(\gamma^{f}): \varphi_{\gamma^{f}}(\delta) = \Z\} =  \{\delta \in \Gamma^{\R}(\gamma^{f}) : \gamma \sim \gamma^{\lambda}_{\mu, \Z}\}
\end{align*}
\noindent where $\gamma^{\lambda}_{\mu, \Z}$ is the unique LS key which exists by Theorem \ref{baumanngaussentlittelmannc}.
\end{enumerate}
\end{thm}

\begin{proof}
Let $\delta$ be a readable gallery. Then by Theorem \ref{existenceoftheuniquesymplecticlskey} there exists a (unique) LS key $\nu$ such that $\delta \sim \nu$. By Proposition \ref{typecwordandgalleryequivalence}, the words $w(\delta)$ and $w(\nu)$ are plactic equivalent. Propositions \ref{denseplacticrelations} and $\ref{densewordreading}$ together with Theorem \ref{lecouvey} then imply that 
\begin{align*}
\overline{\pi(\C_{\delta})} = \overline{\pi(\C_{\nu})},
\end{align*}
\noindent which, by Theorem \ref{baumanngaussentlittelmannc} implies that $\overline{\pi(\C_{\delta})}$ is an MV cycle in $\mathcal{Z}(\mu_{\delta^{+}})_{\mu_{\delta}}$. The map $\varphi_{\gamma^{f}}$ in b. is surjective by Theorem \ref{existenceoftheuniquesymplecticlskey} and Theorem  \ref{baumanngaussentlittelmannc} above. Now let $r$ be a root operator, and let $\tightoverset{\sim}{r}$ be the corresponding root operator that acts on the set of MV cycles. Then by Proposition \ref{densewordreading}, Proposition \ref{wordreading}, Proposition \ref{denseplacticrelations}, and Theorem \ref{baumanngaussentlittelmannc} we have: 

\begin{align*}
\overline{\pi(\C_{r(\gamma)})} = \overline{\pi(\C_{\gamma_{w(r(\gamma))}})} = \overline{\pi(\C_{\gamma_{w(r(\nu))}})} =  \overline{\pi(\C_{r(\nu)})} = \tightoverset{\sim}{r}(\overline{\pi(\C_{\nu})}) = \tightoverset{\sim}{r}(\overline{\pi(\C_{\gamma})}).
\end{align*}
 This completes the proof of b. Pat c. of Theorem \ref{main} follows immediately, since every connected component $\C$ is crystal isomorphic to the corresponding component consisting of the LS galleries equivalent to those in $\C$. Parts d. and e. follow from Theorem 7.1 in \cite{pathsandrootoperators}. This is also discussed in Section \ref{pathmodelsection}.

\end{proof}
\section{Counting Positive Crossings}
\label{countingpositivecrossings}
In  this section we provide proofs of Propositions \ref{densewordreading} and  \ref{denseplacticrelations}. We begin with analysing the \textit{tail} of a gallery in  \ref{truncatedimagesandtails}.  In \ref{examples} we calculate an example in which it can be seen how to apply it. Then in \ref{proofofdensewordreading} we prove Proposition \ref{densewordreading} and in \ref{proofofdenseplacticrelations} we prove Proposition \ref{denseplacticrelations}. We also wish to establish some notation that we will use throughout. Recall our convention $\e_{\bar{l}}: = -\e_{l}$ for $l \in \mathcal{C}_{n}$ unbarred. We will write, for $l,s,d,m \in \mathcal{C}_{n}, c^{i,j}_{ls,dm}$ for the constant $c^{i,j}_{\e_{l}+\e_{s},\e_{d}+\e_{m}}$ in Chevalley's commutator formula (\ref{chevalley}), and $c^{i,j}_{l,dm}, c^{i,j}_{ls,d}$ for $c^{i,j}_{\e_l,\e_{d}+\e_{m}}, c^{i,j}_{\e_{l}+\e_{s},\e_{d}}$ respectively. (Each time we use such notation a total order will be fixed on the set of positive roots.)
\subsection{Truncated Images and Tails}
\label{truncatedimagesandtails}

Let $\gamma$ be a combinatorial gallery with notation as in (\ref{gallery}) with endpoint the coweight $\mu_{\gamma}$ and let $r\leq k+1$ such that $\V_{r}$ is a special vertex; we denote it by $\mu_{r}\in \X^{\vee}$. By Corollary \ref{goodtrick} we know that the image 
$\pi(\C_{\gamma})$ is stable under $\U_{0}$.

\begin{prop}
\label{stable}
The \textbf{r-truncated image} of $\gamma$ 
\begin{align*}
\T_{\gamma}^{\geq r}= \mathbb{U}^{\gamma}_{r}\mathbb{U}^{\gamma}_{r+1}\cdots \mathbb{U}^{\gamma}_{k}[t^{\mu_{\gamma}}]
\end{align*}
\end{prop}

is $\U_{\mu_{r}}$-stable, i.e. for any $u \in \U_{\mu_{r}}, \hbox{  }u\T_{\gamma}^{\geq r} = \T_{\gamma}^{\geq r}$.

\begin{proof}
By (\ref{stabiliseone}), we know that $t^{\mu_{r}}\U_{0}t^{-\mu_{r}} = \U_{\mu_{r}}.$
On the other hand, we may also consider the \textbf{r-truncated gallery} $$\gamma^{\geq r} = (\V_{0}', \E_{0}', \cdots, \V_{k-r+1}'),$$ \noindent which is the combinatorial gallery obtained from the sequence $$(\V_{r}, \E_{r}, \V_{r+1}, \cdots ,\E_{k}, \V_{k+1})$$ \noindent by translating it to the origin. Since $\V_{r}$ is a special vertex, we also have $t^{\mu_{r}}\mathbb{U}^{\gamma^{\geq r}}_{i}t^{-\mu_{r}} = \mathbb{U}^{\gamma}_{i+r}$. This gallery has endpoint $\mu_{\gamma}-\mu_r$ and is in turn a $\T$-fixed point of a Bott-Samelson variety $(\Sigma, \pi')$. Let $u \in \U_{\mu_{r}}$ and $u' = t^{-\mu_{r}}ut^{\mu_{r}} \in \U_{0}$. Then:
\begin{align*}
u \T_{\gamma}^{\geq r} &= u \mathbb{U}^{\gamma}_{r}\mathbb{U}^{\gamma}_{r+1}\cdots \mathbb{U}^{\gamma}_{k}[t^{\mu_{\gamma}}]\\
&=t^{\mu_{r}}u' \mathbb{U}^{\gamma^{\geq r}}_{0}\cdots \mathbb{U}^{\gamma^{\geq r}}_{k-r}[t^{\mu_{\gamma}-\mu_{r}}]\\
\hbox{(by Corollary \ref{goodtrick})}&=t^{\mu_{r}}\mathbb{U}^{\gamma^{\geq r}}_{0}\cdots \mathbb{U}^{\gamma^{\geq r}}_{k-r}[t^{\mu_{\gamma}-\mu_{r}}]
= \T^{\geq r}_{\gamma}
\end{align*}
\end{proof}
\noindent For later use let us fix the notation $$\T_{\gamma}^{<r}= \mathbb{U}_{\V_{0}}\cdots \mathbb{U}_{\V_{r-1}};$$ one may then write $$\pi(\C_{\gamma}) = \T_{\gamma}^{<r}\T_{\gamma}^{\geq r}.$$ 

\begin{rem}
This Proposition is proven for $\op{SL}(n,\mathbb{C})$ in \cite{knuth}, Proposition 3. The proof we have provided is exactly the same, except for the restriction of only being able to truncate at special vertices. 
\end{rem}

\begin{ex}
\label{examples}Let $n = 2$. Consider the symplectic keys
\begin{align*}
\mathscr{K}_{1} &= \Skew(0:\mbox{\tiny{1}},\mbox{\tiny{1}},\mbox{\tiny{$\overline{1}$}}|0:\mbox{\tiny{2}},\mbox{\tiny{2}})\\
\mathscr{K}_{2} &= \Skew(0:\mbox{\tiny{2}},\mbox{\tiny{1}},\mbox{\tiny{2}}|1:\mbox{\tiny{$\overline{2}$}}, \mbox{\tiny{$\overline{1}$}}) 
\end{align*}

\noindent
and their words

\begin{align*}
 w(\mathscr{K}_{1})&= \bar{1}12 \\
 w(\mathscr{K}_{2})&= 2\bar{2}2.
\end{align*}

  Note that   $$\gamma_{\omega_{1}}*\gamma_{\omega_{2}} \sim \gamma_{\omega_{2}}*\gamma_{\omega_{1}}$$ \noindent since both $\gamma_{\omega_{1}}*\gamma_{\omega_{2}}$ and $\gamma_{\omega_{2}}*\gamma_{\omega_{1}}$ are contained in the fundamental chamber and have the same endpoint $\omega_{1}+\omega_{2}$; one checks that 
  
$$f_{\alpha_1}f_{\alpha_2}f_{\alpha_1}(\gamma_{\omega_{1}}*\gamma_{\omega_{2}}) = \gamma_{\mathscr{K}_{1}}$$\noindent
 and $$f_{\alpha_1}f_{\alpha_2}f_{\alpha_1}(\gamma_{\omega_{2}}*\gamma_{\omega_{1}}) = \gamma_{\mathscr{K}_{2}}.$$ \noindent Therefore 
$\gamma_{\mathscr{K}_{1}} \sim \gamma_{\mathscr{K}_{2}}$. Lemma \ref{dlem} and Proposition \ref{wordreading} then imply that $\gamma_{w(\mathscr{K}_{1})} \sim \gamma_{w(\mathscr{K}_{2})}$ (or it can also be checked directly using relation R2 in Theorem \ref{lecouvey} with $y = x = 2$).  Now consider combinatorial galleries $\gamma$ and $\nu$. The galleries $\gamma * \gamma_{\mathscr{K}_{1}} * \nu$ and 
$\gamma * \gamma_{\mathscr{K}_{2}} *\nu$  are $\T$-fixed points in the Bott-Samelson varieties $(\Sigma_{(\gamma * \gamma_{\mathscr{K}_{1}} * \nu)^f}, \pi)$ respectively $(\Sigma_{(\gamma * \gamma_{\mathscr{K}_{2}} * \nu)^f}, \pi')$. The galleries $\gamma_{w(\mathscr{K}_{1})}$ and $\gamma_{w(\mathscr{K}_{1})}$ that correspond to their words are $\T$-fixed points in 
$$(\Sigma_{(\gamma * \gamma_{\omega_{1}}*\gamma_{\omega_{1}}*\gamma_{\omega_{1}}*\nu)^{f}}, \pi'').$$  We show that 
$$\overline{\pi(\C_{\gamma *\gamma_{\mathscr{K}_{1}}*\nu})} = \overline{\pi''(\C_{\gamma_{\gamma * w(\mathscr{K}_{1})}*\nu})}  =\overline{\pi'(\C_{\gamma*\gamma_{w(\mathscr{K}_{2}})*\nu})}.$$
\noindent
 We use the same notation as in (\ref{gallery}) for $\gamma$. Since for any combinatorial gallery $\eta , (\alpha,n) \in \Phi^{\gamma*\eta}_{k+1}$ if and only if $(\alpha, n-\< \alpha, \mu_{\gamma}\rr) \in \Phi^{\gamma}_{0}$, we may assume that $\gamma = \emptyset$. Since $\gamma_{\mathscr{K}_{1}}, \gamma_{\mathscr{K}_{2}}, \gamma_{w(\mathscr{K}_{1})}$ and $\gamma_{w(\mathscr{K}_{2})}$ have the same endpoint $\varepsilon_{2}$, this also implies that   $\T^{\geq 2}_{\gamma_{\mathscr{K}_{1}}*\nu} = \T^{\geq 2}_{\gamma_{\mathscr{K}_{2}}*\nu} = \T^{\geq 3}_{\gamma_{w(\mathscr{K}_{2})}*\nu} = \T^{\geq 3}_{\gamma_{w(\mathscr{K}_{1})}*\nu}$. By Proposition \ref{order}, for $a',b',c',d' \in \mathbb{C}$\newline 
 $$\pi(\C_{\gamma_{\mathscr{K}_{1}*\nu}}) = \U_{(\varepsilon_{1}, -1)}(a')\U_{(\varepsilon_{1}+\varepsilon_{2}, -1)}(b')
\U_{(\varepsilon_{2}, 0)}(c')\U_{(\varepsilon_{1}+\varepsilon_{2}, 0)}(d')\T^{2}_{\gamma_{\mathscr{K}_{1}}*\nu}.$$ By Chevalley's commutator formula (\ref{chevalley}) and applying Proposition \ref{stable} to $\U_{(\varepsilon_{1}-\varepsilon_{2}, -1)}(e) \in \U_{\varepsilon_{2}}$,

\small{
\begin{align*}
&\pi''(\C_{\gamma_{w(\mathscr{K}_{1})*\nu}}) =\\
& \U_{(\varepsilon_{1}, -1)}(a)\U_{(\varepsilon_{1}+\varepsilon_{2}, -1)}(b)\U_{(\varepsilon_{1}-\varepsilon_{2}, -1)}(e)
\U_{(\varepsilon_{2}, 0)}(c)\U_{(\varepsilon_{1}+\varepsilon_{2}, 0)}(d)\T^{\geq 3}_{\gamma_{w(\mathscr{K}_{1})}*\nu}\\
&=\U_{(\varepsilon_{1},-1)}(a+c_{\text{\tiny{$1\overline{2},2$}}}^{\text{\tiny{$1,1$}}}(-e)c)\U_{(\varepsilon_{1}+\varepsilon_{2}, -1)}(b+c_{\text{\tiny{$1\overline{2},2$}}}^{\text{\tiny{$1,2$}}}(-e)c^{2})\U_{(\varepsilon_{2},0)}(c)\U_{(\varepsilon_{1}+\varepsilon_{2},0)}(d)\U_{(\varepsilon_{1}-\varepsilon_{2}, -1)}(e)\T^{\geq 2}_{\gamma_{\mathscr{K}_{1}}*\nu}\\
&=\U_{(\varepsilon_{1},-1)}(a+c_{\text{\tiny{$1\overline{2},2$}}}^{\text{\tiny{$1,1$}}}(-e)c)\U_{(\varepsilon_{1}+\varepsilon_{2}, -1)}(b+c_{\text{\tiny{$1\overline{2},2$}}}^{\text{\tiny{$1,2$}}}(-e)c^{2})\U_{(\varepsilon_{2},0)}(c)\U_{(\varepsilon_{1}+\varepsilon_{2},0)}(d)\T^{\geq 2}_{\gamma_{\mathscr{K}_{1}}*\nu}\\
& \subset \pi(\C_{\gamma_{\mathscr{K}_{1}}*\nu})
\end{align*}
}
\normalsize
for $a,b,c,d,e \in \mathbb{C}$. Choosing $a = a', b = b', c = c', d = d', e = 0$, we have 
$$\pi(\C_{\gamma_{\mathscr{K}_{1}}}) \subset \pi''(\C_{\gamma_{w(\mathscr{K}_{1})}}).$$ Hence in this case $\pi(\C_{\gamma_{\mathscr{K}_{1}}}) = \pi''(\C_{\gamma_{w(\mathscr{K}_{1})}})$. Similarly, for $\{a'',b'',c'',d'',e''\} \subset \mathbb{C}$,

\small{
\begin{align*}
&\pi''(\C_{\gamma_{w(\mathscr{K}_{2})}*\nu}) = \U_{(\varepsilon_{2}, 0)}(a'')\U_{(\varepsilon_{1}+\varepsilon_{2}, 0)}(b'')\U_{(\varepsilon_{1}-\varepsilon_{2}, -1)}(e'')
\U_{(\varepsilon_{2}, 0)}(c'')\U_{(\varepsilon_{1}+\varepsilon_{2}, 0)}(d'')\T^{\geq 3}_{\gamma_{w(\mathscr{K}_{2})}*\nu} =\\
&\U_{(\varepsilon_{1},-1)}(c^{\text{\tiny{$1\overline{2},2$}}}_{\text{\tiny{$1,1$}}}(-e'')c'')\U_{(\varepsilon_{1}+\varepsilon_{2}, -1)}(c^{\text{\tiny{$1\overline{2},2$}}}_{\text{\tiny{$1,2$}}}(-e''){c''}^{2})\U_{(\varepsilon_{2}, 0)}(a''+c'')\U_{(\varepsilon_{1}+\varepsilon_{2},0)}(b''+d'')\T^{\geq 3}_{\gamma_{w(\mathscr{K}_{2})}*\nu} \\
&\subset \pi(\C_{\gamma_{\mathscr{K}_{1}*\nu}}).
\end{align*}
}
\normalsize
Hence the open subset of $\pi(\C_{\gamma_{\mathscr{K}_{1}*\nu}})$ given by $a \neq 0, b\neq 0, c\neq 0, d\neq 0$ is contained in $\pi''(\C_{\gamma_{w(\mathscr{K}_{2})}*\nu})$. 
\end{ex}

\subsection{Proof of Proposition \ref{densewordreading}}

\label{proofofdensewordreading}
We want to show that if $\gamma$ and $\nu$ are combinatorial galleries and $\mathscr{K}$ is a readable block, $$\overline{\pi(\C_{\gamma*\gamma_{\mathscr{K}}*\nu})} = \overline{\pi'(\C_{\gamma*\gamma_{w(\mathscr{K})}*\nu})}.$$

\begin{proof}
We assume $\gamma = \emptyset$; we may do so by the argument given at the beginning of Example \ref{examples}. Let $\mathscr{K}$ be an LS block and let $\A = \{a_{1}, \cdots, a_{r}\}, \B = \{b_{1}, \cdots b_{s}\}, \Z = \{z_{1}, \cdots, z_{k}\}$ and $\T = \{t_{1}, \cdots, t_{k}\}$ be the subsets of $\{1, \cdots, n\}$ from Definition \ref{lsblock} that determine $\mathscr{K}$. We will use the notation $d_{1} < \cdots < d_{r+k}$ to denote the ordered elements of $\Z \cup \A$, and $f_{1}< \cdots < f_{s+k}$ the ordered elements of $\B \cup \Z$. We also write 

\begin{align*}
\gamma_{\mathscr{K}} = (\V_{0}, \E_{0}, \V_{1}, \E_{1}, \V_{2}).
\end{align*}

\noindent
The proof is divided into Lemmas \ref{firstcontention} and \ref{secondcontention} below. 

\begin{lem}
\label{firstcontention}
Let $\nu$ be a combinatorial gallery and $\mathscr{K}$ be a readable block. Then 
\begin{align*}
\overline{\pi'(\C_{\gamma_{w(\mathscr{K})}*\nu})} \subseteq \overline{\pi(\C_{\gamma_{\mathscr{K}}*\nu})}.
\end{align*}
\end{lem}

We first need the following claim.

\begin{cl}
\label{firstcontentionfirstclaim}
$$\pi'(\C_{\gamma_{w(\mathscr{K})}*\nu}) \subset \U_{0} \mathbb{P}'''_{\bar{f}_{k+s}}\cdots \mathbb{P}'''_{\bar{f}_{1}} \T^{\geq 2k +r+s}_{\gamma_{w(\mathscr{K})}* \nu}$$
\noindent where
\begin{align}
\label{dwr-1}
\mathbb{P}'''_{\bar b} &= \underset{l<b}{\underset{l \notin \Z\cup \A \cup \B \cup \T;}{\prod}}\U_{(\e_{l}-\e_{b},0)}(k_{l\bar{b}}) \underset{t \in \T^{<b}}{\prod} \U_{(\e_{t}-\e_{b}, 0)}(k_{t\bar{b}})\underset{a\in \A^{<b}}{\prod} \U_{(\e_{a}-\e_{b},1)}(k_{a\bar{b}}) \hbox{ for } b \in \B\\
\label{dwr0}
\mathbb{P}'''_{\bar z} &= \underset{l<z}{\underset{l \notin \Z\cup \A \cup \B \cup \T;}{\prod}} \U_{(\e_{l}-\e_{z},-1)}(k_{l\bar{z}})\underset{t \in \T^{<z}}{\prod} \U_{(\e_{t}-\e_{z},-1)}(k_{t\bar{z}})\underset{b \in \B^{<z}}{\prod} \U_{(\e_{b}-\e_{z},-1)}(k_{b\bar{z}}) \hbox{ for } z \in \Z
\end{align}
\end{cl}

\begin{proof}[Proof of Claim \ref{firstcontentionfirstclaim}]
\noindent
The points of $\pi'(\C_{\gamma_{w(\mathscr{K})}*\nu})$ are of the form 

\begin{align}
\label{dwr1}
\mathbb{P}_{d_{1}}\cdots \mathbb{P}_{d_{r+k}}\mathbb{P}_{\bar{f}_{k+s}}\cdots \mathbb{P}_{\bar{f}_{1}} \T^{\geq 2k +r+s}_{\gamma_{w(\mathscr{K})}* \nu}
\end{align}

\noindent
where
\small{
\begin{align}
\label{dwr2}
&\mathbb{P}_{d} = \U_{(\e_{d}, 0)}(g_{d}) \underset{d<l\leq n}{\prod} \U_{(\e_{d}-\e_{l}, 0)}(g_{d\bar l}) \underset{l \notin (\Z\cup \A)^{<d}}{\prod} \U_{(\e_{d}+\e_{l}, 0)}(g_{dl})\underset{l\in (\Z\cup \A)^{<d}}{\prod} \U_{(\e_{d}+\e_{l}, 1)}(g^{1}_{dl})\\
\label{dwr3}
&\mathbb{P}_{\bar b} =
\s_{\bar b}\overbrace{ \underset{l<b}{\underset{l \notin \Z\cup \A \cup \B \cup \T;}{\prod}}\U_{(\e_{l}-\e_{b},0)}(g_{l\bar{b}}) \underset{t \in \T^{<b}}{\prod} \U_{(\e_{t}-\e_{b}, 0)}(g_{t\bar{b}})\underset{a\in \A^{<b}}{\prod} \U_{(\e_{a}-\e_{b},1)}(g_{a\bar{b}})}^{= \mathbb{P}^{iv}_{\bar b}}\\
\label{dwr4}
& \s_{\bar b} =\underset{b' \in \B^{<b}}{\prod} \U_{(\e_{b'}-\e_{b},0)}(g_{b'\bar{b}})\underset{z \in \Z^{<b}}{\prod} \U_{(\e_{z}-\e_{b},1)}(g^{1}_{z\bar{b}}) \in \U_0\\
\label{dwr5}
&\mathbb{P}_{\bar {z}} = \J_{\bar z}\underbrace{\underset{l<z}{\underset{l \notin \Z\cup \A \cup \B \cup \T;}{\prod}} \U_{(\e_{l}-\e_{z},-1)}(g_{l\bar{z}})\underset{t \in \T^{<z}}{\prod} \U_{(\e_{t}-\e_{z},-1)}(g_{t\bar{z}})\underset{b \in \B^{<z}}{\prod} \U_{(\e_{b}-\e_{z},-1)}(g_{b\bar{z}})}_{= \mathbb{P}^{iv}_{\bar z}} \\
\label{dwr6}
&\J_{\bar z} = \underset{a \in\A^{<z}}{\prod} \U_{(\e_{a}-\e_{z},0)}(g_{a\bar{z}})\underset{z'\in \Z^{<z}}{\prod} \U_{(\e_{z'}-\e_{z},0)}(g_{z'\bar{z}}) \in \U_{0}.
\end{align}
}
\normalsize

\noindent
for $d \in \A \cup \Z$, $z \in \Z,$ and $b \in \B$. All the terms in $\J_{\bar z}$ commute with $\mathbb{P}^{iv}_{z'}$ for $z' \in \Z^{>z}$ and with $\mathbb{P}^{iv}_{\bar b}$ for $b \in \B^{>z}$. All the terms in $\s_{\bar b}$ commute with $\mathbb{P}^{iv}_{\bar b'}$ for $b' \in \B^{>b}$. For $z' > b$ it commutes with all terms of $\mathbb{P}^{iv}_{\bar z'}$ except for the term $\U_{(\e_{b}-\e_{z'}, -1)}(g_{b \bar z'})$. However, commuting $\s_{\bar b}$ with this term (using Chevalley's commutator formula \ref{chevalley}) produces terms $\U_{(\e_{z}-\e_{z'}, 0)}(*)$ and $\U_{(\e_{b'}-\e_{z'}, -1)}(*)$. Out of these terms, $\U_{(\e_{z}-\e_{z'}, 0)}(*)$ commutes with $\mathbb{P}^{iv}_{z'}$ for $z' \in \Z^{>z}$ and with $\mathbb{P}^{iv}_{\bar b}$ for $b \in \B^{>z}$, and $\U_{(\e_{b'}-\e_{z'}, -1)}(*)$ is a term of the form of those appearing in $\mathbb{P}^{iv}_{\bar z}$. Therefore (and since the the terms that appear in $\mathbb{P}^{iv}_{\bar b}$ and $\mathbb{P}^{iv}_{\bar z}$ are the same as $\mathbb{P}^{''}_{\bar b}$ and $\mathbb{P}^{''}_{\bar z}$ respectively) concludes the proof of Claim \ref{firstcontentionfirstclaim}.

\end{proof}

\begin{cl}
\label{firstcontentionsecondclaim}
There is a dense subset of  $\mathbb{P}'''_{\bar{f}_{k+s}}\cdots \mathbb{P}'''_{\bar{f}_{1}}\T^{\geq 2k +r+s}_{\gamma_{w(\mathscr{K})}* \nu}$ that is contained in the subset

\begin{align*}
&\mathbb{P}_{\T, \B}\mathbb{P}_{\mathscr{K},\bar{f}_{s}}\cdots \mathbb{P}_{\mathscr{K}, \bar{f}_{s}}\T^{\geq 2k +r+s}_{\gamma_{w(\mathscr{K})}* \nu} \subset \overline{\pi(\C_{\gamma_{\mathscr{K}}*\nu})},
\end{align*}
\noindent where

\begin{align*}
&\mathbb{P}_{\T, \B} = \underset{ t \in \T l < t}{\underset{l\notin \Z\cup \A\cup \B\cup \T,}{\prod}}\U_{(\e_{l}-\e_{t},0)}(v_{l\bar t}) \underset{b \in \B, l < b}{\underset{l\notin \Z\cup \A\cup \B\cup \T,}{\prod}} \U_{(\e_{l}-\e_{b},0)}(v_{l\bar b}) \in \U_{\V_{0}}\\
&\mathbb{P}_{\mathscr{K}, \bar b} =  \underset{ t \in \T^{<b}}{\underset{b \in \B;}{\prod}} \U_{(\e_{t}-\e_{b}, 0)}(v_{t\bar{b}})\underset{a \in \A^{<b}}{\prod} \U_{(\e_{a}-\e_{b},1)}(v_{a\bar{b}}) \in\U_{\V_1}\\
&\mathbb{P}_{\mathscr{K}, \bar z} =\underset{t\in \T^{<z}}{\prod} \U_{(\e_{t}-\e_{z},-1)}(v_{t\bar{z}}) \underset{b \in \B^{<z}}{\prod} \U_{(\e_{b}-\e_{z},-1)}(v_{b\bar{z}}) \in \U_{\V_1}, 
\end{align*}

\noindent
for $v_{ij} \in \mathbb{C}, b \in \B$ and $z \in \Z$. (It is indeed a subset by Corollary \ref{goodtrick}.)
\end{cl}

\noindent
Note that $\T^{\geq 2k +r+s}_{\gamma_{w(\mathscr{K})}* \nu} = \T^{\geq 2}_{\gamma_{\mathscr{K}}* \nu}$ and that $$u =  \underset{ t \in \T l < t}{\underset{l\notin \Z\cup \A\cup \B\cup \T,}{\prod}}\U_{(\e_{l}-\e_{t},0)}(v_{l\bar t}) \in \U_{\mu_{\gamma_{\mathscr{K}}}}.$$
\noindent We have the following equalities
\begin{align*}
&\mathbb{P}_{\T, \B}\mathbb{P}_{\mathscr{K}, \bar{f}_{s}}\cdots \mathbb{P}_{\mathscr{K},\bar{f}_{s}}\T^{\geq 2k +r+s}_{\gamma_{w(\mathscr{K})}* \nu} =\\
&\mathbb{P}''_{\bar{f}_{s}}\cdots \mathbb{P}''_{\bar{f}_{s}}u\T^{\geq 2}_{\gamma_{\mathscr{K}}* \nu} = \\
&\mathbb{P}''_{\bar{f}_{s}}\cdots \mathbb{P}''_{\bar{f}_{s}}\T^{\geq 2}_{\gamma_{\mathscr{K}}* \nu}
\end{align*}
\noindent
where, for $z \in \Z$ and $b \in \B$:
\begin{align*}
&\mathbb{P}''_{\bar b}= \underset{l<b}{\underset{l \notin \Z\cup \A \cup \B \cup \T;}{\prod}}\U_{(\e_{l}-\e_{b},0)}(\xi_{l\bar{b}}) \underset{t \in \T^{<b}}{\prod} \U_{(\e_{t}-\e_{b}, 0)}(\xi_{t\bar{b}})\underset{a\in \A^{<b}}{\prod} \U_{(\e_{a}-\e_{b},1)}(\xi_{a\bar{b}}) \hbox{ c.f (\ref{dwr-1})}\\
&\mathbb{P}''_{\bar z} = \underset{l<z}{\underset{l \notin \Z\cup \A \cup \B \cup \T;}{\prod}} \U_{(\e_{l}-\e_{z},-1)}(\xi_{l\bar{z}})\underset{t \in \T^{<z}}{\prod} \U_{(\e_{t}-\e_{z},-1)}(\xi_{t\bar{z}})\underset{b \in \B^{<z}}{\prod} \U_{(\e_{b}-\e_{z},-1)}(\xi_{b\bar{z}}) \hbox{ c.f (\ref{dwr0})}\\
&\xi_{l\bar{b}} = v_{l\bar {b}} + \underset{t \in \T}{\underset{l<t<b,}{\sum}} c^{1,1}_{s\bar t, t \bar{b}}(-v_{l\bar t})v_{t \bar{b}}\\
& \xi_{l \bar z} = \rho_{l \bar z} + \underset{z' \in \Z}{\sum} c^{1,1}_{l\bar{z}', z'\bar{z}}(-\rho_{l\bar{z}'})v_{z'\bar z} + \underset{b \in \B}{\underset{l<b<z,}{\sum}}c^{1,1}_{l\bar b , b \bar z}(-\xi_{l\bar b})v_{b \bar z} \\
& \rho_{l\bar{z}} = \underset{ t \in \T}{\sum_{l<t<z,}}c^{1,1}_{l\bar{t},t\bar{z}}(-v_{l\bar{t}})v_{t\bar{z}} \hbox{ (for }z \in \Z)\\
& \xi_{t \bar z} = v_{t \bar z} \\
& \xi_{b \bar z} = v_{b \bar z}\\
& \xi_{t \bar b} = v_{t \bar b}.
\end{align*}
\noindent
To prove Claim \ref{firstcontentionsecondclaim} we must set open conditions on the parameters $k_{ij}$ such that the system of equations defined by $v_{ij} = \xi_{ij}$ has a solution in the variables $v_{ij}$. Setting $v_{t \bar z}: = k_{t \bar z}$ and $v_{b \bar z}: = k_{b \bar z}$ this is reduced to setting conditions on the $k_{ij}$ so that the following system can be solved:
\begin{align}
\label{fceq1}
&k_{l\bar{b}} = v_{l\bar {b}} + \underset{l<t<b, t \in \T}{\sum}c^{1,1}_{l\bar t, t \bar{b}}(-v_{l\bar t})k_{t \bar{b}}\\
\label{fceq2}
& k_{l \bar z} = \rho_{l \bar z} - \underset{l<b<z, b \in \B}{\sum}c^{1,1}_{l\bar b , b \bar z}(v_{l\bar {b}} + \underset{t \in \T}{\underset{l<t<b,}{\sum}}c^{1,1}_{l\bar t, t \bar{b}}(-v_{l\bar t})k_{t \bar{b}})k_{b \bar z} \\
\label{fceq3}
& \rho_{l\bar{z}} = \sum_{l<t<z, t \in \T}c^{1,1}_{l\bar{t},t\bar{z}}(-v_{l\bar{t}})k_{t\bar{z}}.
\end{align}
\noindent
Lines (\ref{fceq1}) and (\ref{fceq2}) above define a linear system of as many equations as variables: the variables are $\{v_{l\bar b}\}_{l \notin \A \cup \B \cup \T; b \in \B^{>l}}\cup 
\{v_{l\bar t}\}_{l \notin \A\cup \B \cup \Z\cup \T; t  \in \T^{>l}}$, there is one equation for each $l\bar b, l \notin \A \cup \B \cup \T; b \in \B^{>l}$, for each $l\bar z, l \notin \A \cup \B \cup \T; z \in \Z^{>l}$, and note that by definition of an LS block the sets $\{ l\bar z, l \notin \A \cup \B \cup \T; z \in \Z^{>l} \}$ and $\{ l\bar t, s \notin \A \cup \B \cup \T; b \in \B^{>l} \}$ have the same cardinality ($t_{i}$ is the maximal element of the set $\{l \notin \A \cup \B \cup \T, s<t_{i+1}, s< z_{i}\}$). Therefore the system has a solution as long as the matrix of coefficients has non-zero determinant, which imposes open conditions on the $k_{ij}'s$. Hence Claim \ref{firstcontentionsecondclaim} is proven. Now, to finish the proof of Lemma \ref{firstcontention}, note that if the $k_{ij}'s$ satisfy the open conditions established by Claim \ref{firstcontentionsecondclaim}, then 
\begin{align*}
\mathbb{P}'''_{\bar{f}_{k+s}}\cdots \mathbb{P}'''_{\bar{f}_{1}} \T^{\geq 2k +r+s}_{\gamma_{w(\K)}* \nu} \subseteq \pi (\C_{\gamma_{\mathscr{K}}*\nu}),
\end{align*}

\noindent
and therefore Proposition \ref{stable} implies that

\begin{align*}
\U_{0}\mathbb{P}'''_{\bar{f}_{k+s}}\cdots \mathbb{P}'''_{\bar{f}_{1}} \T^{\geq 2k +r+s}_{\gamma_{w(\K)}* \nu} \subseteq \pi (\C_{\gamma_{\mathscr{K}}*\nu}),
\end{align*}

\noindent
which implies Lemma \ref{firstcontention}. Now we show the second contention towards Proposition \ref{densewordreading}.

\begin{lem}
\label{secondcontention}
Let $\nu$ be a combinatorial gallery and $\mathscr{K}$ be an LS block. Then 
\begin{align*}
 \overline{\pi (\C_{\gamma_{\mathscr{K}}*\nu})}\subseteq \overline{\pi'(\C_{\gamma_{w(\mathscr{K})}*\nu})} 
\end{align*}
\end{lem}

\noindent
Recall that 
\begin{align*}
\pi (\C_{\gamma_{\mathscr{K}}*\nu}) = \mathbb{U}^{\gamma_{\mathscr{K}}*\nu}_{0}\mathbb{U}^{\gamma_{\mathscr{K}}*\nu}_{1}\T^{\geq 2}_{\gamma_{\mathscr{K}}*\nu}.
 \end{align*}
 Notice that $ \mathbb{U}^{\gamma_{\mathscr{K}}*\nu}_{0} \subset \U_{0}$ and that all generators of $ \mathbb{U}^{\gamma_{\mathscr{K}}*\nu}_{1}$ also belong to $\U_{0}$ except for  those of the form $\U_{(\e_{t}-\e_{z},-1)}(v_{t\bar z})$ or $\U_{(\e_{t}+\e_{t'},-1)}(v_{tt'})$ for $t, t' \in \T, z \in \Z^{>t}$, and $v_{t \bar z}, v_{tt'} \in \mathbb{C}$. Hence, since, again, $\T^{\geq 2}_{\gamma_{\mathscr{K}}*\nu} = \T^{\geq 2k+r+s}_{\gamma_{w(\mathscr{K})}*\nu}$ all elements of $\pi(\C_{\gamma_{\mathscr{K}}*\nu})$ belong to 
\begin{align}
\label{dwr9}
\U_{0} \underset{z \in \Z^{>t}}{\underset{ t \in \T}{\prod}}\U_{(\e_{t}-\e_{z}, -1)}(v_{t\bar z})\underset{ t, t' \in \T}{\prod} \U_{(\e_{t}+\e_{t'}, -1)}(v_{tt'}) \T^{\geq 2k+r+s}_{\gamma_{w(\mathscr{K})}*\nu}.
\end{align}

\noindent
Now consider
\begin{align*}
\underset{t' \in \T, z \in \Z}{\prod}\U_{(\e_{z}+\e_{t'}, 0)}(k_{zt'})\underset{t \in \T, z \in \Z^{>t}}{\prod}\U_{(\e_{t}-\e_{z}, -1)}(k_{t\bar z}) \T^{\geq 2k+r+s}_{\gamma_{w(\mathscr{K})}*\nu}
\end{align*}
which is a subset of $\pi'(\C_{\gamma_{w(\mathscr{K})}*\nu})$ (by Proposition \ref{stable}) because 
\begin{align*}
&\underset{t \in \T, z \in \Z}{\prod}\U_{(\e_{z}+\e_{t}, 0)}(k_{zt}) \in \U_{0} \hbox{ and  }\\
&\underset{z \in \Z^{>t}}{\underset{ t \in \T}{\prod}} \U_{(\e_{t}-\e_{z}, -1)}(k_{t\bar z}) \T^{\geq 2k+r+s}_{\gamma_{w(\mathscr{K})}*\nu} \subset \pi'(\C_{\gamma_{w(\mathscr{K})}*\nu}).
\end{align*}
We have
\begin{align}
\label{dwr10}
&\underset{t' \in \T, z \in \Z}{\prod}\U_{(\e_{z}+\e_{t'}, 0)}(k_{zt'})\underset{t \in \T, z \in \Z^{>t}}{\prod}\U_{(\e_{t}-\e_{z}, -1)}(k_{t\bar z}) \T^{\geq 2k+r+s}_{\gamma_{w(\mathscr{K})}*\nu} = \\
\label{dwr11}
&\underset{t \neq t'}{\underset{ t, t' \in \T}{\prod}} \U_{(\e_{t}+\e_{t'}, -1)}(\xi_{tt'})\underset{t \in \T, z \in \Z^{>t}}{\prod}\U_{(\e_{t}-\e_{z}, -1)}(k_{t\bar z})\underset{t' \in \T, z \in \Z}{\prod}\U_{(\e_{z}+\e_{t'}, 0)}(k_{zt'}) \T^{\geq 2k+r+s}_{\gamma_{w(\mathscr{K})}*\nu}\\
\label{dwr12}
&\underset{t \neq t'}{\underset{ t, t' \in \T}{\prod}} \U_{(\e_{t}+\e_{t'}, -1)}(\xi_{tt'})\underset{t \in \T, z \in \Z^{>t}}{\prod}\U_{(\e_{t}-\e_{z}, -1)}(k_{t\bar z}) \T^{\geq 2k+r+s}_{\gamma_{w(\mathscr{K})}*\nu}
\end{align}
where
\begin{align}
\label{equations}
\xi_{tt'} & = \underset{z \in \Z^{>t'}}{\sum} c^{1,1}_{zt, t' \bar z}(-k_{zt})k_{t'\bar{z}} + \underset{z \in \Z^{>t}}{\sum} c^{1,1}_{zt', t \bar z}(-k_{zt'})k_{t\bar{z}}.
\end{align}
The equality between (\ref{dwr10}) and  (\ref{dwr11}) is due to Chevalley's commutator formula (\ref{chevalley}) and the equality between (\ref{dwr11}) and (\ref{dwr12}) is obtained by using Proposition \ref{stable} and $\U_{(\e_{z}+\e_{t'}, 0)}(k_{zt'}) \in \U_{\mu_{\gamma_{\mathscr{K}}}}$. 
Now fix an element in (\ref{dwr9}). Setting $k_{t\bar z} = v_{t\bar z}$ defines the linear equations

\begin{align*}
v_{tt'} & = \underset{z \in \Z^{>t'}}{\sum} c^{1,1}_{zt, t' \bar z}(-k_{zt})v_{t'\bar{z}} + \underset{z \in \Z^{>t}}{\sum} c^{1,1}_{zt', t \bar z}(-k_{zt'})v_{t\bar{z}}
\end{align*}
in the variables $k_{zt},$ for $z \in \Z$ and $t \in \T$. There are more variables than equations: for each equation indexed by a non ordered pair $(t_{i}, t_{j})$ there are the variables $v_{zt_{i}}$ and $v_{z' t_{j}}$ for $z > t'$ and $z' > t$ (which always exist by definition of an LS block);  hence the system has solutions as long as the matrix of coefficients has non-zero determinants. This imposes an open condition on the parameters $v_{t\bar z}$. Hence for such $v_{t \bar z}, v_{tt'}, k_{t\bar z} = v_{t\bar z}$, and solutions $k_{ij},$ for the latter equations we have
\begin{align*}
& \underset{z \in \Z^{>t}}{\underset{ t \in \T}{\prod}}\U_{(\e_{t}-\e_{z}, -1)}(v_{t\bar z})\underset{ t, t' \in \T}{\prod} \U_{(\e_{t}+\e_{t'}, -1)}(v_{tt'}) \T^{\geq 2k+r+s}_{\gamma_{w(\mathscr{K})}*\nu}= \\
&\underset{z \in \Z}{\underset{t' \in \T,}{\prod}}\U_{(\e_{z}+\e_{t'}, 0)}(k_{zt'})\underset{ z \in \Z^{>t}}{\underset{t \in \T,}{\prod}}\U_{(\e_{t}-\e_{z}, -1)}(k_{t\bar z}) \T^{\geq 2k+r+s}_{\gamma_{w(\mathscr{K})}*\nu} \subset \pi'(\C_{\gamma_{w(\mathscr{K})}*\nu});
\end{align*}
Proposition \ref{stable} then implies 
\begin{align*}
\U_{0} \underset{z \in \Z^{>t}}{\underset{ t \in \T}{\prod}}\U_{(\e_{t}-\e_{z}, -1)}(v_{t\bar z})\underset{ t, t' \in \T}{\prod} \U_{(\e_{t}+\e_{t'}, -1)}(v_{tt'}) \T^{\geq 2}_{\gamma_{\mathscr{K}}*\nu} \subset \pi'(\C_{\gamma_{w(\mathscr{K})}*\nu});
\end{align*}
this completes the proof of Lemma \ref{secondcontention} and hence of Proposition \ref{densewordreading}.

Now let $\mathscr{K}$ be a zero lump. This means there exists $k>1$ such that the right (respectively left) column of $\mathscr{K}$ has as entries the integers $1 < \cdots < k$ (respectively $\bar k < \cdots < \bar 1$); its word is therefore $w(\mathscr{K}) = 1 \cdots k \bar k \cdots \bar 1$. This means, in particular, that the truncated images $\T^{\geq 2k}_{\gamma_{w(\mathscr{K})}*\nu} = \T^{\geq 2}_{\gamma_{\mathscr{K}}*\nu}$ are stabilised by $\U_{0}$, by Proposition \ref{stable}. We have

\begin{align*}
\pi'(\C_{\gamma_{w(\mathscr{K})}*\nu}) = \mathbb{U}^{\gamma_{w(\mathscr{K})}*\nu}_{0}\cdots \mathbb{U}^{\gamma_{w(\mathscr{K})}*\nu}_{2k -1} \T^{\geq 2k}_{\gamma_{w(\mathscr{K})}*\nu}
\end{align*}

\noindent
by Theorem \ref{celldescription}. Clearly all the subgroups $\mathbb{U}^{\gamma_{w(\mathscr{K})}*\nu}_{l} \subset \U_{0}$ for $1\leq l \leq k$. For $0\leq j \leq k-1$, the generators of $\mathbb{U}^{\gamma_{w(\mathscr{K})}*\nu}_{k+j}$ are all of the form $\U_{(\e_{s}-\e_{k-j}, n_{k-j})}$ for $l< k-j$. In particular the gallery $\gamma_{1\cdots k \bar k \cdots \overline{k-j-1}}$ has crossed the hyperplanes $\mm_{(\e_{s}-\e_{k-j}, m)}$ once positively at $m = 0$ and once negatively at $m=1$, which means that $n_{k-j} = 0$, $\U_{(\e_{s}-\e_{k-j}, n_{k-j})}(a) = \U_{(\e_{s}-\e_{k-j},0)}(a) \in \U_{0}$, for all $a \in \mathbb{C}$. Hence 

\begin{align*}
\pi'(\C_{\gamma_{w(\mathscr{K})}*\nu}) &= \mathbb{U}^{\gamma_{w(\mathscr{K})}*\nu}_{0}\cdots \mathbb{U}^{\gamma_{w(\mathscr{K})}*\nu}_{2k -1} \T^{\geq 2k}_{\gamma_{w(\mathscr{K})}*\nu} \\
&= \T^{\geq 2k}_{\gamma_{w(\mathscr{K})}*\nu}\\
&= \T^{\geq 2}_{\gamma_{\mathscr{K}}*\nu}.
\end{align*}

\noindent
In 

\begin{align*}
\pi(\C_{\gamma_{\mathscr{K}}* \nu}) = \mathbb{U}^{\gamma_{\mathscr{K}}* \nu}_{0}\mathbb{U}^{\gamma_{\mathscr{K}}* \nu}_{1}  \T^{\geq 2}_{\gamma_{\mathscr{K}}*\nu} 
\end{align*}

\noindent
we have $\mathbb{U}^{\gamma_{\mathscr{K}}* \nu}_{1}  = \{\op{Id}\}$ and $\mathbb{U}^{\gamma_{\mathscr{K}}* \nu}_{0} \subset \U_{0}$, therefore

\begin{align*}
\pi(\C_{\gamma_{\mathscr{K}}*\nu}) = \T^{\geq 2}_{\gamma_{\mathscr{K}}*\nu} = \T^{\geq 2k}_{\gamma_{w(\mathscr{K})}*\nu}
\end{align*}

\noindent
 since $\mu_{\gamma_{\mathscr{K}}} = \mu_{\gamma_{w(\mathscr{K})}}$.


\subsection{Proof of Proposition \ref{denseplacticrelations}}
\label{proofofdenseplacticrelations}
\begin{proof}[Proof of Proposition \ref{denseplacticrelations}]

Let $\nu$ be a combinatorial gallery. 

\subsection*{Relation R1}
For $z \neq \overline{x}$:
$$
\begin{aligned}
a) \hbox{ }& y\hbox{ }x\hbox{ }z \equiv y\hbox{ }z\hbox{ }x &&\text{    for    }& x\leq y < z \\
b) \hbox{ }& x\hbox{ }z\hbox{ }y \equiv z\hbox{ }x\hbox{ }y &&\text{  for   }& x<y\leq z
\end{aligned}
$$

\begin{lem}
\label{denser1}
Let $w_{1} = y\hbox{ }x\hbox{ }z$ and $w_{2} = y\hbox{ }z\hbox{ }x$, $w_{3} = x\hbox{ }z\hbox{ }y$, and $w_{4} = z\hbox{ }x\hbox{ }y$ for $z \neq \bar x$. Then

\begin{align*}
a) \overline{\pi(\C_{\gamma_{w_{1}}*\nu})} = \overline{\pi(\C_{\gamma_{w_{2}}*\nu})}\\
b) \overline{\pi(\C_{\gamma_{w_{3}}*\nu})} = \overline{\pi(\C_{\gamma_{w_{4}}* \nu})}
\end{align*}

\end{lem}

\begin{proof}
For the proof we recall the notation $\varepsilon_{\bar a} = -\e_{a}$ and $\bar \bar i  = i$ for any $i \in \{1, \cdots, n\}$. Also note that $\T^{\geq 3}_{\gamma_{w_{i}}*\nu}$ all coincide for $i \in \{1,2,3,4\}$; we will denote them by $\T^{w}$. We divide the proof of Lemma \ref{denser1} in three cases.\\

\noindent
\textbf{Case 1}: $x < y < z$\\
\noindent
\begin{cl} If $z \neq \bar y$ and $y \neq \bar x:$
\label{expressionsforcase1}
\begin{itemize}
\item[i.] $\pi(\C_{\gamma_{w_{1}}*\nu}) = \U_{0}\U_{(\e_{x}-\e_{y}, -1)}(v_{x\bar y})\T^{w}$
\item[ii.] $\pi(\C_{\gamma_{w_{2}}*\nu}) = \U_{0}\U_{(\e_{x}-\e_{y}, -1)}(v_{x\bar y})\U_{(\e_{x}-\e_{z}, -1)}(v_{x \bar z})\T^{w}$
\item[iii.] $\pi(\C_{\gamma_{w_{3}}*\nu}) = \U_{0}\U_{(\e_{y}-\e_{z}, -1)}(v_{y\bar z})\T^{w}$
\item[iv.] $\pi(\C_{\gamma_{w_{4}}*\nu}) =  \U_{0}\U_{(\e_{x}-\e_{z}, -1)}(v_{x\bar z})\U_{(\e_{y}-\e_{z}, -1)}(v_{y \bar z})\T^{w}$.
\end{itemize}
\end{cl}

\begin{proof}[Proof of Claim \ref{expressionsforcase1}]
 We first remark that, regardless whether $x, y,$ and $z$ are barred or unbarred, the roots $\e_{x} - \e_{z}, \e_{y} - \e_{z},$ and $\e_{x} - \e_{y}$ are always positive. Now we recall the notation from Theorem \ref{celldescription}:
\begin{align*}
\pi(\C_{\gamma_{w_{1}}*\nu}) = \mathbb{U}^{\gamma_{w_{i}}* \nu}_{0}\mathbb{U}^{\gamma_{w_{i}}* \nu}_{1}\mathbb{U}^{\gamma_{w_{i}}* \nu}_{2}\T^{w}
\end{align*}
Assume that $z \neq \bar y$ and $y \neq \bar x$.

i. We have $\U_{(\e_{x} - \e_{y}, -1)}(v_{x \bar y}) \in  \mathbb{U}^{\gamma_{w_{1}}* \nu}_{1}$ for any $v_{x \bar y} \in \mathbb{C}$, hence 
$$\U_{0}\U_{(\e_{x}-\e_{y}, -1)}(v_{l \bar y})\T^{w} \subseteq \pi(\C_{\gamma_{w_{1}}*\nu}).$$
Out of all generators of $\mathbb{U}^{\gamma_{w_1}* \nu}_{i}$ for $i \in \{0,1,2\}$, the only one that does not belong to $\U_{0}$ is of the form $\U_{(\e_{x}-\e_{y}, -1)}(v_{x \bar y}) \in \mathbb{U}^{\gamma_{w_{1}}* \nu}_{1}$, and the ones from $\mathbb{U}^{\gamma_{w_{1}}* \nu}_{2}$ that do not commute with it are those of the form $\U_{(\e_{y}+\e_{z}, 1)}(a)$, but in that case Chevalley's commutator formula produces a term $\U_{(\e_{x}+\e_{z}, 0)} (c^{1,1}_{x \bar y, yz}(-v_{x\bar y})a) \in \U_{0}$. This implies the other inclusion, together with Proposition \ref{order}, which allows us to write down the generators of each $\mathbb{U}^{\gamma_{w_{1}}* \nu}_{i}$ in any order. \\ 
\noindent ii. The only generators of $\mathbb{U}^{\gamma_{w_{2}}* \nu}_{i}$ for $i \in \{0,1,2\}$ that do not belong to $\U_{0}$ are those of the form $\U_{(\e_{x} - \e_{y}, -1)}(v_{x\bar y}) \in \mathbb{U}^{\gamma_{w_{2}}* \nu}_{2}$ and $\U_{(\e_{x}- \e_{z}, -1)}(v_{x \bar z}) \in \mathbb{U}^{\gamma_{w_{2}}* \nu}_{2}$. The equality follows by Proposition \ref{order}, Theorem \ref{celldescription}, and Proposition \ref{stable}. \\
\noindent iii. All the generators of $\mathbb{U}^{\gamma_{w_{3}}*\nu}_{0}$ and $\mathbb{U}^{\gamma_{w_{3}}*\nu}_{1}$ belong to $\U_{0}$, and the only generators of $\mathbb{U}^{\gamma_{w_{3}}*\nu}_{2}$ that do not are $\U_{(\e_{y}-\e_{z}, -1)}$. Thus Claim \ref{expressionsforcase1} follows by Proposition \ref{stable} and Theorem \ref{celldescription}.\\
\noindent iv. As in the previous cases, we have
\begin{align*}
\pi(\C_{\gamma_{w_{4}*\nu}}) = \mathbb{U}^{\gamma_{w_{4}}*\nu}_{0} \mathbb{U}^{\gamma_{w_{4}}*\nu}_{1} \mathbb{U}^{\gamma_{w_{4}}*\nu}_{2}\T^{w},
\end{align*}
and $\mathbb{U}^{\gamma_{w_{4}}*\nu}_{0} \subset \U_{0}$. All generators of $\mathbb{U}^{\gamma_{w_{4}}*\nu}_{1}$ and respectively $\mathbb{U}^{\gamma_{w_{4}}*\nu}_{2}$ belong to $\U_{0}$ except for $\U_{(\e_{x}-\e_{z}, -1)}(a) \in \mathbb{U}^{\gamma_{w_{4}}*\nu}_{1}$ and $\U_{(\e_{y}-\e_{z}, -1)}(b) \in \mathbb{U}^{\gamma_{w_{4}}*\nu}_{2}$, respectively, for $\{a,b\} \subset \mathbb{C}$. To prove this part of Claim \ref{expressionsforcase1} we observe that $\U_{(\e_{x}-\e_{z}, -1)}(a)$ commutes with all generators of $\mathbb{U}^{\gamma_{w_{4}}*\nu}_{2}$ except for $\U_{(\e_{z}+\e_{y},1)}(d)$, with $d \in \mathbb{C}$. However, commuting the latter two terms produces elements $\U_{(\e_{x}+\e_{y}, 0)}(c^{1,1}_{x\bar z, zy}(-a)d) \in \U_{0}$. Therefore 
\begin{align*}
\pi(\C_{\gamma_{w_{4}}*\nu}) \subseteq  \U_{0}\U_{(\e_{x}-\e_{z}, -1)}(v_{x\bar z})\U_{(\e_{y}-\e_{z}, -1)}(v_{y \bar z})\T^{w},
\end{align*}
and the other inclusion is clear by Proposition \ref{stable} and the above discussion. This finishes the proof of Claim \ref{expressionsforcase1}.

\end{proof}

Now we make use of Claim \ref{expressionsforcase1} to prove Lemma \ref{denser1} in this case, assuming $z \neq \bar y$ and $y \neq \bar x$. For both a) and b) Claim \ref{expressionsforcase1} immediately implies 
\begin{align*}
\pi(\C_{\gamma_{w_{1}}*\nu}) &\subseteq \pi (\C_{\gamma_{w_{2}}*\nu}) \hbox{ and }\\
\pi (\C_{\gamma_{w_{3}}*\nu}) &\subseteq \pi(\C_{\gamma_{w_{4}}*\nu}).
\end{align*}

\noindent
Next we will show 

\begin{align*}
\ov{\pi(\C_{\gamma_{w_{2}}*\nu})} \subseteq \ov{\pi(\C_{\gamma_{w_{1}}*\nu})}.
\end{align*}

\noindent
For this, let $v_{y\bar z } \in \mathbb{C}$ and $v_{x\bar y} \in \mathbb{C}$ with $v_{x\bar y} \neq 0$. Then since $\U_{(\e_{y}-\e_{z}, 0)}(v_{y\bar z}) \in \U_{\mu_{w}} \cap \U_{0}$ for any $v_{y\bar z} \in \mathbb{C}$, Lemma \ref{denser1} ,  Chevalley's commutator formula, and Proposition \ref{stable} imply 

\begin{align*}
\pi (\C_{\gamma_{w_{1}}* \nu}) \supset \U_{(\e_{y}-\e_{z}, 0)}(v_{y\bar z})\U_{(\e_{x}-\e_{y}, -1)}(v_{x\bar y})\T^{w} = \\
\U_{(\e_{x}-\e_{z}, -1)}(c^{1,1}_{y\bar z, v\bar y}(-v_{y\bar z})v_{x\bar y})\U_{(\e_{x}-\e_{y}, -1)}(v_{x\bar y}) \U_{(\e_{y}-\e_{z}, 0)}(v_{y\bar z})\T^{w} = \\
\U_{(\e_{x}-\e_{z}, -1)}(c^{1,1}_{y\bar z, v\bar y}(-v_{y\bar z})v_{x\bar y})\U_{(\e_{x}-\e_{y}, -1)}(v_{x\bar y})\T^{w}
\end{align*}

\noindent
Therefore $$\U_{(\e_{x}-\e_{y}, -1)}(v_{x\bar y})\U_{(\e_{x}-\e_{z}, -1)}(v_{x \bar z})\T^{w} \subset \pi (\C_{\gamma_{w_{1}}*\nu})$$

\noindent
as long as $v_{x \bar y} \neq 0$, since in that case $c^{1,1}_{y\bar z, v\bar y}(-v_{y\bar z})v_{x\bar y} = v_{x\bar z}$ has a solution in $v_{y\bar z}$. Hence Proposition \ref{stable} implies 

\begin{align*}
\U_{0} \U_{(\e_{x}-\e_{y}, -1)}(v_{x\bar y})\U_{(\e_{x}-\e_{z}, -1)}(v_{x \bar z})\T^{w} \subset \pi (\C_{\gamma_{w_{1}}*\nu}).
\end{align*}

\noindent
Claim \ref{expressionsforcase1} (i. and ii.) then implies that a dense subset of $\pi(\C_{\gamma_{w_{2}}*\nu})$ is contained in $\pi (\C_{\gamma_{w_{1}}*\nu})$, which implies Lemma \ref{denser1} , a) in this case. To finish the proof of 
Lemma \ref{denser1} b), let $v_{x \bar y} \in \mathbb{C}$ and $v_{y \bar z} \in \mathbb{C}$ with $v_{y \bar z} \neq 0$. Then, just as for a)

\begin{align}
& \pi (\C_{\gamma_{w_{3}}*\nu}) \supset \U_{(\e_{x}-\e_{y}, 0)}(v_{x\bar y})\U_{(\e_{y}-\e_{z}, -1)}(v_{y\bar z})\T^{w} = \\
&\U_{(\e_{x}-\e_{z}, -1)}(c^{1,1}_{x\bar y, y\bar z}(-v_{x\bar y})v_{y\bar z})\U_{(\e_{y}-\e_{z}, -1)}(v_{y\bar z}) \U_{(\e_{x}-\e_{y}, 0)}(v_{y\bar z})\T^{w} = \\
\label{r1.1w43}
&\U_{(\e_{x}-\e_{z}, -1)}(c^{1,1}_{x\bar y, y\bar z}(-v_{x\bar y})v_{y\bar z})\U_{(\e_{y}-\e_{z}, -1)}(v_{y\bar z})\T^{w}.
\end{align}

\noindent
Therefore the elements of the set 

$$\U_{(\e_{x}-\e_{z}, -1)}(v_{x\bar z})\U_{(\e_{y}-\e_{z}, -1)}(v_{y \bar z})\T^{w}$$

\noindent
such that $v_{y \bar z} \neq 0$ are contained in (\ref{r1.1w43}). By Claim \ref{expressionsforcase1} (iii. and iv)  and Proposition \ref{stable} there is a dense subset of 

\begin{align*}
&\pi(\C_{\gamma_{w_{4}}*\nu}) =  \U_{0}\U_{(\e_{x}-\e_{z}, -1)}(v_{x\bar z})\U_{(\e_{y}-\e_{z}, -1)}(v_{y \bar z})\T^{w}
\end{align*}

\noindent
that is contained in $\pi (\C_{\gamma_{w_{3}}*\nu})$. \\

The cases $z= \bar y$ and $y = \bar x$ are missing so far. (Note that $ z \neq \bar x$ is not allowed. Also note that and that if $y = \bar x$ then $x$ must be unbarred and if $z= \bar y$ then $y$ must be unbarred.) \\

 Case 1.1 $z= \bar y$\\

a. We first show that 
\begin{align}
\label{hardcaseeasycontention}
\overline{\pi(\C_{\gamma_{w_{1}}*\nu})} \subseteq \overline{\pi(\C_{\gamma_{w_{2}}*\nu})}. 
\end{align} 

All of the generators of $\mathbb{U}^{\gamma_{w_{1}}*\nu}_{1}$ belong to $\U_{0}$ except for $\U_{(\e_{x}-\e_{y},-1)}(v_{x\bar y}),$ for $v_{x\bar y} \in \mathbb{C}$. The generators of 
$\mathbb{U}^{\gamma_{w_{1}}*\nu}_{1}$ are $\U_{(\e_{l}-\e_{y}, -1)}(v_{l \bar y})$ for $l \neq x$ and $v_{l \bar y} \in \mathbb{C}$, and $\U_{(\e_{x}-\e_{y}, 0)}(v_{x\bar y})$ for $v_{x\bar y}\in \mathbb{C}$. This last term commutes with $\U_{(\e_{x}-\e_{y},-1)}(v_{x\bar y})$. Therefore, by parallel arguments to those given in the proof of Claim \ref{expressionsforcase1}, 

\begin{align*}
\pi(\C_{\gamma_{w_{1}}*\nu}) = \U_{0} \U_{(\e_{x}-\e_{y},-1)}(v_{x\bar y}) \underset{l \neq x}{\underset{l< y}{\prod}}\U_{(\e_{l}-\e_{y}, -1)}(v_{l \bar y})\T^{w}
\end{align*}

All terms in the product $\U_{(\e_{x}-\e_{y},-1)}(v_{x\bar y}) \underset{l \neq x}{\underset{l< y}{\prod}}\U_{(\e_{l}-\e_{y}, -1)}(v_{l \bar y})$ are at the same time generators of $\mathbb{U}_{1}^{\gamma_{w_2}}$ as well, therefore, by Proposition \ref{stable}, 

\begin{align*}
\pi(\C_{\gamma_{w_{1}}*\nu}) \subseteq \pi(\C_{\gamma_{w_{2}}*\nu}),
\end{align*}

\noindent
as wanted. Next we would like to show 
\begin{align}
\label{hardcasehardcontention}
\overline{\pi(\C_{\gamma_{w_{2}}*\nu})} \subseteq \overline{\pi(\C_{\gamma_{w_{1}}*\nu})}. 
\end{align} 
To do so we will make use of Proposition \ref{densewordreading}. Let 
$$\mathscr{K}_{1} = \Skew(0:\mbox{\tiny{x}}, \mbox{\tiny{x}}, \mbox{\tiny{y}}|0: \ov{\mbox{\tiny{y}}}, \ov{\mbox{\tiny{y}}})$$ and $$\mathscr{K}_{2} = \Skew(0: \mbox{\tiny{x}}, \mbox{\tiny{y-1}}, \mbox{\tiny{y}}|1: \ov{\mbox{\tiny{y}}}, \ov{\mbox{\tiny{y-1}}}).$$ Then we have $w_{1} = y \hbox{ }x \hbox{ }\bar y = w(\mathscr{K}_{1})$ and $w_{2} =y \hbox{ } \bar y \hbox{ } x = w(\mathscr{K}_{2})$. By Proposition \ref{densewordreading} it then suffices to show 
\begin{align*}
\overline{\pi''(\C_{\gamma_{\mathscr{K}_{2}}})} \subseteq \overline{\pi'(\C_{\gamma_{\mathscr{K}_{1}}})}. 
\end{align*}

\noindent
First assume $y-1 \neq x$.  Note that in this case $\mathbb{U}^{\gamma_{\mathscr{K}_{2}}*\nu}_{1}$ is generated by terms $\U_{(\e_{y-1}-\e_{y}, -1)}(a)$ with $a \in \mathbb{C}$, and all generators of $\mathbb{U}^{\gamma_{\mathscr{K}_{2}}*\nu}_{0}$ and $\mathbb{U}^{\gamma_{\mathscr{K}_{2}}*\nu}_{2}$ belong to $\U_{0}$. Out of these, the only ones in $\mathbb{U}^{\gamma_{\mathscr{K}_{2}}*\nu}_{2}$ that do not commute with  with $\U_{(\e_{y-1}-\e_{y}, -1)}(a)$ are $\U_{(\e_{x}+\e_{y}, 0)}(b)$ and $\U_{(\e_{x}-\e_{y-1},0)}(d)$.  Then for every element in $\pi(\C_{\gamma_{\mathscr{K}_{2}}*\nu})$ there is a $u \in \U_{0}$ such that it belongs to
\begin{align*}
&u \U_{(\e_{y-1}-\e_{y}, -1)}(a) \overbrace{\U_{(\e_{x}+\e_{y}, 0)}(b) \U_{(\e_{x}-\e_{y-1},0)}(d)}^{=: u'} \T^{w} = \\
&u u' \U_{(\e_{y-1}+\e_{x}, -1)}(c^{1,1}_{y-1 \bar y, xy}(-a)b)\U_{(\e_{x}-\e_{y}, -1)}(c^{1,1}_{y-1 \bar y, x \ov{y-1}}(-a)d)  \U_{(\e_{y-1}-\e_{y}, -1)}(a) \T^{w}. 
\end{align*}

\noindent
Fix such $u, a, b,$ and $d$ such that $abd \neq 0$. Such elements form a dense subset of $\pi''(\C_{\gamma_{\mathscr{K}_{2}}*\nu})$. We will show 
\begin{align*}
 \U_{(\e_{y-1}+\e_{x}, -1)}(c^{1,1}_{y-1 \bar y, xy}(-a)b)\U_{(\e_{x}-\e_{y}, -1)}(c^{1,1}_{y-1 \bar y, x \ov{y-1}}(-a)d) \U_{(\e_{y-1}-\e_{y}, -1)}(a) \T^{w} \\ \subset \pi'(\C_{\gamma_{\mathscr{K}_{1}}*\nu})
\end{align*}
If this is true, then (\ref{hardcasehardcontention}) is then implied by Proposition \ref{stable} applied to \newline $u  \U_{(\e_{x}+\e_{y}, 0)}(b) \U_{(\e_{x}-\e_{y-1},0)}(d) \in \U_{0}$. \\

\noindent
First note that for all $\{a_{x\bar{y}},a_{y-1 \bar{y}}, a_{y y-1}\} \subset \mathbb{C}$, $\U_{(\e_{x}-\e_{y}, -1)}(a_{x\bar y})$ and \newline $\U_{(\e_{y-1}-\e_{y}, -1)}(a_{y-1 y})$ belong to $\mathbb{U}^{\gamma_{\mathscr{K}_{1}}*\nu}_{1}$, and  $ v: =\U_{(\e_{y}+\e_{y-1}, 0)}(a_{y y-1}) \in \U_{\e_{x}} \cap \U_{0}$ stabilises the truncated image $\T^{w}$ as well as the whole image $\pi'(\C_{\gamma_{\mathscr{K}_{1}}* \nu})$. Therefore all elements of 

\begin{align*}
v^{-1}\U_{(\e_{x} - \e_{y}, -1)}(a_{x\bar y}) \U_{(\e_{y-1}-\e_{y}, -1)}(a_{y-1 \bar{y}}) v \T^{w} = \\
\U_{(\e_{x}+\e_{y-1}, -1)}(c^{1,1}_{x\bar y, y y-1}(-a_{x\bar y})a_{y y-1})\U_{(\e_{x} - \e_{y}, -1)}(a_{x\bar y}) \U_{(\e_{y-1}-\e_{y}, -1)}(a_{y-1 \bar{y}}) \T^{w}
\end{align*}

belong to $\pi'(\C_{\gamma_{\mathscr{K}_{1}}*\nu})$ and since $abd \neq 0$ we may find $a_{x\bar{y}}, a_{y-1 \bar{y}},$ and $a_{y y-1}$ such that

\begin{align*}
 a_{x\bar y} &= c^{1,1}_{y-1 \bar y, x \ov{y-1}}(-a)d,\\
 c^{1,1}_{x\bar y, y y-1}(-a_{x\bar y})a_{y y-1} &= c^{1,1}_{y-1 \bar y, xy}(-a)b, \hbox{ and }\\
 a_{y-1 \bar{y}} &= a.
\end{align*}
\noindent
This concludes the proof if $ y \neq x-1$. Now assume that $y = x-1$. In this case all generators of $\mathbb{U}_{2}^{\gamma_{\K_{2}}*\nu}$ commute with $\U_{(\e_{y-1}-\e_{y}, -1)}(a_{y-1\bar{y}})$, and therefore all elements in $\pi''(\C_{\gamma_{\mathscr{K}_{2}}*\nu})$ belong to 
\begin{align*}
u \U_{(\e_{y-1}-\e_{y}, -1)}(a) \T^{w}
\end{align*}
for some $u \in \U_{0}$ and $a \in \mathbb{C}$ - but $\U_{(\e_{y-1}-\e_{y}, -1)}(a) \in \mathbb{U}_{1}^{\gamma_{\mathscr{K}_{1}}*\nu}$, which implies (\ref{hardcasehardcontention}) by applying Proposition \ref{stable} to $u \in \U_{0}$. \\

b. We now have 
\begin{align*}
w_{3} = x \hbox{ } \bar y \hbox{ } y = w(\mathscr{K}_{3}) \hbox{ and } w_{4} = \bar y \hbox{ } x \hbox{ } y = w(\mathscr{K}_{4}),
\end{align*}
where $$\mathscr{K}_{3} = \Skew(0: \mbox{\tiny{y}}, \mbox{\tiny{x}}, \mbox{\tiny{x}}|1: \ov{\mbox{\tiny{y}}}, \ov{\mbox{\tiny{y}}})$$ and $$\mathscr{K}_{4} = \Skew(0:\mbox{\tiny{x}},\mbox{\tiny{x}}, \ov{\mbox{\tiny{y}}}|0:\mbox{\tiny{y}}, \mbox{\tiny{y}}).$$ We want to show 

\begin{align*}
\overline{\pi'''(\C_{\gamma_{\mathscr{K}_{3}}*\nu})} = \overline{\pi''''(\C_{\gamma_{\mathscr{K}_{4}}*\nu})}.
\end{align*}
First $\mathbb{U}^{\gamma_{\mathscr{K}_{3}}*\nu}_{0}$ and $\mathbb{U}^{\gamma_{\mathscr{K}_{3}}*\nu}_{1}$ are both contained in $\U_{0}$. The generators of $\mathbb{U}^{\gamma_{\mathscr{K}_{3}}*\nu}_{2}$ that do not belong to $\U_{0}$ are $\U_{(\e_{y}, -1)}(\alpha_{y}),\U_{(\e_{y} + \e_{l}, -1)}(\beta_{yl}),$ and $\U_{(\e_{y}-\e_{s}, -1)}(\gamma_{y\bar s})$ for $\{\alpha_{y},\beta_{yl},\gamma_{y \bar s}\} \subset \mathbb{C}$ and $l \leq n, l \neq x, y< s \leq n$. All of these are also generators of $\mathbb{U}^{\gamma_{\mathscr{K}_{4}}*\nu}_{1}$, hence by Proposition \ref{stable} and Theorem \ref{celldescription} we have
\begin{align*}
\pi'''(\C_{\gamma_{\mathscr{K}_{3}}*\nu}) \subset \pi''''(\C_{\gamma_{\mathscr{K}_{4}}*\nu}). 
\end{align*}
The discussion above also implies that
\begin{align}
\label{zisybarbk3}
\pi'''(\C_{\gamma_{\mathscr{K}_{3}}*\nu}) = \U_{0}\U_{(\e_{y}, -1)}(\alpha_{y}) \underset{l \neq x}{\underset{l \leq n}{\prod}}\U_{(\e_{y} + \e_{l}, -1)}(\beta_{yl}) \underset{y <s \leq n}{\prod}\U_{(\e_{y}-\e_{s}, -1)}(\gamma_{y\bar s})\T^{w}
\end{align}

There is one more generator of $\mathbb{U}^{\gamma_{\mathscr{K}_{4}}*\nu}_{1}$, not mentioned above, which is $\U_{(\e_{x}+\e_{y}, -1)}(d_{xy})$. Since all generators of $\mathbb{U}^{\gamma_{\mathscr{K}_{4}}*\nu}_{2} (\hbox{ which are } \U_{(\e_{x}+\e_{y}, 0)}(d') \in \U_{0} \hbox{ for } d' \in \mathbb{C})$ commute with those of 
$\mathbb{U}^{\gamma_{\mathscr{K}_{3}}*\nu}_{1}$, we have by Proposition \ref{stable}:
\begin{align*}
&\pi''''(\C_{\gamma_{\mathscr{K}_{4}}*\nu}) = \\
&\U_{0}\U_{(\e_{x}+\e_{y}, -1)}(d_{xy})\U_{(\e_{y}, -1)}(a_{y})\underset{l \neq x}{\underset{l \leq n}{\prod}}\U_{(\e_{y} + \e_{l}, -1)}(b_{yl})\underset{s > y}{\underset{s \leq n}{\prod}}\U_{(\e_{y}-\e_{s}, -1)}(c_{y \bar s})\T^{w}
\end{align*}
We now would like to show
\begin{align*}
\overline{\pi''''(\C_{\gamma_{\mathscr{K}_{4}}*\nu})} \subset \overline{\pi'''(\C_{\gamma_{\mathscr{K}_{3}}*\nu})}.
\end{align*}
To do this we will  see that for complex numbers $a_{y}, b_{yl}, c_{y\bar s}$, and $d_{xy}$,  with $a_{y} \neq 0$,
 
\begin{align}
\label{zisybarbk3.1}
&\U_{(\e_{x}+\e_{y}, -1)}(d_{xy})\U_{(\e_{y}, -1)}(a_{y})\underset{l \neq x}{\underset{l \leq n}{\prod}}\U_{(\e_{y} + \e_{l}, -1)}(b_{yl})\underset{s > y}{\underset{s \leq n}{\prod}}\U_{(\e_{y}-\e_{s}, -1)}(c_{y \bar s})\T^{w}\\
&\subset  \pi'''(\C_{\gamma_{\mathscr{K}_{3}}*\nu}).
\end{align}

\noindent
By (\ref{zisybarbk3}) we conclude that for any complex numbers $\alpha_{y}, \beta_{yl}, \gamma_{y\bar s}$, and $\delta$ the following set is contained in $\pi'''(\C_{\gamma_{\mathscr{K}_{3}}*\nu})$
\begin{align*}
&v^{-1}\U_{(\e_{x}-\e_{y},1)}(\delta)\U_{(\e_{y}, -1)}(\alpha_{y}) \underset{l \neq x}{\underset{l \leq n}{\prod}}\U_{(\e_{y} + \e_{l}, -1)}(\beta_{yl}) \underset{s > y}{\underset{s \leq n}{\prod}}\U_{(\e_{y}-\e_{s}, -1)}(\gamma_{y\bar s})\T^{w} =\\
&v^{-1}v \U_{(\e_{x}+\e_{y}, -1)}(\rho_{xy})\U_{(\e_{y}, -1)}(\alpha_{y}) \underset{l \neq x}{\underset{l \leq n}{\prod}}\U_{(\e_{y} + \e_{l}, -1)}(\beta_{yl}) \underset{s > y}{\underset{s \leq n}{\prod}}\U_{(\e_{y}-\e_{s}, -1)}(\gamma_{y\bar s}) \T^{w}
\end{align*}

where 
\begin{align*}
&v = \\
& \U_{(\e_{x}, 0)}(c^{1,1}_{x\bar y, y}(-\delta)\alpha_{y}) \underset{l \neq x}{\underset{l \leq n}{\prod}} \U_{(\e_{x}+\e_{l}, 0)}(c^{1,1}_{x\bar y, y l}(-\delta)\beta_{yl})  \underset{s > y}{\underset{s \leq n}{\prod}}\U_{(\e_{x}-\e_{s}, 0)}(c^{1,1}_{x\bar y, y \bar s}(-\delta)\gamma_{y \bar s})\\
&\rho_{xy} = c^{1,2}_{x \bar y, y}(-\delta)\alpha^{2}_{y}, 
\end{align*}

\noindent
and where the latter equality is obtained by applying Chevalley's commutator formula and Proposition \ref{stable} applied to $\U_{(\e_{x}-\e_{y},1)}(\delta)$, which stabilises the truncated image $\T^{w}$. We will have shown our claim in (\ref{zisybarbk3.1}) if we find complex numbers $\alpha_{y}, \beta_{yl}, \gamma_{y\bar s}$, and $\delta$ such that
 
\begin{align*}
c^{1,2}_{x \bar y, y}(-\delta)\alpha^{2}_{y} &= d_{xy}\\
\alpha_{y} &= a_{y}\\
\beta_{yl} & = b_{yl}, 
\end{align*}
\noindent
which we may obtain since $a_{y} \neq 0$. This concludes the proof in case $z = \bar y$. \\

 Case 1.2 $y= \bar x$. This means that $x$ is necessarily unbarred and therefore $z = \bar b$ for some $b < x$. \\
 
 a. As before, we will use Proposition \ref{densewordreading}.  We have 
\begin{align*}
w_{1} = \bar x \hbox{ } x \hbox{ }\bar b = w(\mathscr{K}_{1}) \hbox{ and }\\
w_{2} = \bar x \hbox{ } \bar b \hbox{ } x = w(\mathscr{K}_{2}),
\end{align*}

where $$\mathscr{K}_{1} = \Skew(0: \mbox{\tiny{x}}, \mbox{\tiny{x}}, \ov{\mbox{\tiny{x}}}| 0: \ov{\mbox{\tiny{b}}}, \ov{\mbox{\tiny{b}}})$$ and $$\mathscr{K}_{2} = \Skew(0: \mbox{\tiny{x}}, \ov{\mbox{\tiny{x}}}, \ov{\mbox{\tiny{x}}}| 1: \ov{\mbox{\tiny{b}}}, \ov{\mbox{\tiny{b}}}).$$ First we show 
\begin{align}
\label{case1.2afirstcontention}
\overline{\pi'(\C_{\gamma_{\mathscr{K}_{1}}* \nu})} \subseteq \overline{\pi''(\C_{\gamma_{\mathscr{K}_{2}}*\nu})}.
\end{align}
To do this, we claim that
\begin{align}
\label{case1.2k1}
\pi'(\C_{\gamma_{\mathscr{K}_{1}}*\nu}) = \U_{0} \U_{(\e_{x}, -1)}(a_{x})\underset{\e_{x}+\e_{s} \in \Phi^{+}}{\underset{ s \in \mathcal{C}_{n} \neq b}{\prod}}\U_{(\e_{x}+\e_{s}, -1)}(a_{xs})\T^{w}.
\end{align}
Indeed, $\U_{(\e_{x}, -1)}(a_{x})$ and $\U_{(\e_{x}+\e_{s}, -1)}(a_{xs})$ for $s \in \mathcal{C}_{n}$ and $s\neq b$ are the generators of $\mathbb{U}^{\gamma_{\mathscr{K}_{1}}*\nu}_{1}$ that do not belong to $\U_{0}$, and $\mathbb{U}^{\gamma_{\mathscr{K}_{1}}*\nu}_{2}$ is the identity, because $\e_{x} - \e_{b}$ is not a positive root. Therefore (\ref{case1.2k1}) follows by Proposition \ref{stable}. The aforementioned terms are also generators (but not all!) of 
$\mathbb{U}^{\gamma_{\mathscr{K}_{2}}*\nu}_{2}$, therefore (\ref{case1.2afirstcontention}) follows. Now we show 
\begin{align}
\label{case1.2asecondcontention}
\overline{\pi''(\C_{\gamma_{\mathscr{K}_{2}}*\nu})} \subseteq \overline{\pi'(\C_{\gamma_{\mathscr{K}_{1}}*\nu})}.
\end{align}
To do this, let us first analyse the image 
\begin{align*}
\pi''(\C_{\gamma_{\mathscr{K}_{2}}*\nu}) = \mathbb{U}^{\gamma_{\mathscr{K}_{2}}*\nu}_{0}\mathbb{U}^{\gamma_{\mathscr{K}_{2}}*\nu}_{1}\mathbb{U}^{\gamma_{\mathscr{K}_{2}}*\nu}_{2}\T^{w}.
\end{align*}
In this case $\mathbb{U}^{\gamma_{\mathscr{K}_{2}}*\nu}_{0} \subset \U_{0}$ and $\mathbb{U}^{\gamma_{\mathscr{K}_{2}}*\nu}_{1}$ is the identity, because $-(\e_{x}+\e_{b})$ is not a positive root. The generators of $\mathbb{U}^{\gamma_{\mathscr{K}_{2}}*\nu}_{2}$ are $\U_{(\e_{x}, -1)}(\alpha_{x}), \U_{(\e_{x}+\e_{s}, -1)}(\alpha_{xs})$ and $\U_{(\e_{x}+\e_{b}, -2)}(\alpha_{xb})$ for $s \in \mathcal{C}_{n}$ such that $s \neq b$ and complex numbers $\alpha_{x}, \alpha_{xs}$, and $\alpha_{xb}$. Therefore 
\begin{align}
\label{case1.2k2}
\pi(\C_{\gamma_{\mathscr{K}_{2}}*\nu}) = \U_{0} \U_{(\e_{x}, -1)}(\alpha_{x})\underset{\e_{x}+\e_{s} \in \Phi^{+}}{\underset{s \neq b}{\prod}}\U_{(\e_{x}+\e_{s}, -1)}(\alpha_{xs})\U_{(\e_{x}+\e_{b}, -2)}(\alpha_{xb})\T^{w}. 
\end{align}
 Let us fix complex numbers $\alpha_{x}, \alpha_{xs},$ and $\alpha_{xb}$, such that $\alpha_{x} \neq 0$. We will show that (cf. (\ref{case1.2k1}))
 \begin{align}
 \label{case1.2k2usualtrick}
 \U_{(\e_{x}, -1)}(\alpha_{x})\underset{\e_{x}+\e_{s} \in \Phi^{+}}{\underset{s \neq b}{\prod}}\U_{(\e_{x}+\e_{s}, -1)}(\alpha_{xs})\U_{(\e_{x}+\e_{b}, -2)}(\alpha_{xb})\T^{w} \subset \pi'(\C_{\gamma_{\mathscr{K}_{1}}*\nu})
 \end{align}
To do this we will use Corollary \ref{goodtrick}, which says, in particular, that, if we write
\begin{align*}
\gamma_{\mathscr{K}_{1}} = (\V_{0}, \E_{0}, \V_{1}, \E_{1}, \V_{2}, \E_{2}, \V_{3}),
\end{align*}
then 
\begin{align*}
\pi'(\C_{\gamma_{\mathscr{K}_{1}}}) \supset \U_{\V_{0}}\U_{\V_{1}}\U_{\V_{2}}\T^{w}. 
\end{align*}
Therefore, since $u: =\U_{(\e_{b}-\e_{x}, 0)}(a) \in \U_{\V_{2}} \cap \U_{0}$ for all $a \in \mathbb{C}$, and since $\U_{(\e_{x}, -1)}(a_{x})$ and $\U_{(\e_{x}+\e_{s}, -1)}(a_{xs}),$ for $s \in \mathcal{C}_{n}$ and $s\neq b$ are the generators of $\mathbb{U}^{\gamma_{\mathscr{K}_{1}}*\nu}_{1} \subset \U_{\V_{1}}$ for any complex numbers $a_{xs}$ and $a_{x}$ we have (using, again, Proposition \ref{stable} applied to $u \in \U_{0}$ and $v \in \U_{\V_{3}}$ ($\V_{3}$ stabilises the truncated image $\T^{w}$; see below for a definition of $v$)): 
\begin{align*}
&\pi'(\C_{\gamma_{\mathscr{K}_{1}}*\nu}) \supset \\
& u^{-1}\U_{(\e_{x}, -1)}(a_{x})\underset{\e_{x}+\e_{s} \in \Phi^{+}}{\underset{s \neq b}{\prod}}\U_{(\e_{x}+\e_{s}, -1)}(a_{xs}) u \T^{w} = \\
&u^{-1}u \U_{(\e_{x} + \e_{b}, -2)}(c^{2,1}_{x, b \bar x}(a^{2}_{x})b) \U_{(\e_{x}, -1)}(a_{x})\underset{\e_{x}+\e_{s} \in \Phi^{+}}{\underset{s \neq b}{\prod}}\U_{(\e_{x}+\e_{s}, -1)}(a_{xs}) v  \T^{w} = \\ 
&\U_{(\e_{x} + \e_{b}, -2)}(c^{2,1}_{x, b \bar x}(a^{2}_{x})b) \U_{(\e_{x}, -1)}(a_{x})\underset{\e_{x}+\e_{s} \in \Phi^{+}}{\underset{s \neq b}{\prod}}\U_{(\e_{x}+\e_{s}, -1)}(a_{xs})\T^{w}.
\end{align*}
where 
\begin{align*}
v = \U_{(\e_{b}, -1)}(c^{1,1}_{x, b \bar x}(-a_{x})b) \underset{\e_{x}+\e_{s} \in \Phi^{+}}{\underset{s \neq b}{\prod}}\U_{(\e_{b}+\e_{s}, -1)}(c^{1,1}_{x, bs}(-a_{xs})b) \in \U_{\V_{3}}.
\end{align*}
In order to show (\ref{case1.2k2usualtrick}) it suffices to find complex numbers $a_{x}, a_{xs},$ and $b$ such that 
\begin{align*}
c^{2,1}_{x, b \bar x}(a^{2}_{x})b = \alpha_{xb}\\
a_{x} = \alpha_{x}\\
a_{xs} = \alpha_{xs},
\end{align*}
and we may do this, since $\alpha_{x} \neq 0$. \\

b.  We will again use Proposition \ref{densewordreading}. We have 
\begin{align*}
w_{3} = x \hbox{ }\bar b\hbox{ }\bar x = w(\mathscr{K}_{3}) \hbox{ and }
w_{4} = \bar b \hbox{ } x \hbox{ } \bar x = w(\mathscr{K}_{4}),
\end{align*}

where $$\mathscr{K}_{3} = \Skew(0: \ov{\mbox{\tiny{x}}}, \mbox{\tiny{x}}, \mbox{\tiny{x}}| 1: \ov{\mbox{\tiny{b}}}, \ov{\mbox{\tiny{b}}})$$ and $$\mathscr{K}_{4} = \Skew(0: \mbox{\tiny{x-1}}, \mbox{\tiny{x}}, \ov{\mbox{\tiny{b}}}|0: \ov{\mbox{\tiny{x}}}, \ov{\mbox{\tiny{x-1}}}).$$ By Proposition \ref{densewordreading} it is enough to show 

\begin{align}
\label{case1.2b}
\overline{\pi'''(\C_{\gamma_{\mathscr{K}_{3}}*\nu})} = \overline{\pi''''(\C_{\gamma_{\mathscr{K}_{4}}*\nu})}.
\end{align}

\noindent
We analyse both images $\pi'''(\C_{\gamma_{\mathscr{K}_{3}}*\nu})$ and $\pi''''(\C_{\gamma_{\mathscr{K}_{4}}*\nu})$ separately and then show (\ref{case1.2b}). First , since $\mathbb{U}^{\gamma_{\mathscr{K}_{3}}*\nu}_{0} \subset \U_{0}$ and $\mathbb{U}^{\gamma_{\mathscr{K}_{3}}*\nu}_{1}$ is the identity (this is because $\e_{x}-\e_{b}$ is not a positive root), we have
\begin{align}
\label{case1.2k3}
\pi'''(\C_{\gamma_{\mathscr{K}_{3}}*\nu}) = \U_{0}\underset{l \neq b}{\underset{l < x}{\prod}} \U_{(\e_{l}-\e_{x}, -1)}(a_{l \bar x}) \U_{(\e_{b}-\e_{x}, -2)}(a_{b \bar x})\T^{w}.
\end{align}

\noindent
Now, $\mathbb{U}^{\gamma_{\mathscr{K}_{4}}*\nu}_{2}$ is generated by elements $\U_{(\e_{x-1}-\e_{x}, -1)}(\alpha_{x-1 x}),$ for $\alpha_{x-1 x} \in \mathbb{C},$ and $\mathbb{U}^{\gamma_{\mathscr{K}_{4}}*\nu}_{1}$ is generated by $\U_{(\e_{b}-\e_{x-1}, -1)}(\alpha_{b \ov{x-1}})$ for $\alpha_{b \ov{x-1}} \in \mathbb{C}$, by $\U_{(\e_{l}-\e_{x-1}, 0)}(\alpha_{l \bar{x-1}})$ for $l < x-1$ and $\alpha_{l \bar{x-1}} \in \mathbb{C}$ (this last element stabilises the truncated image $\T^{w}$) , and by other elements of $\U_{0}$. Therefore

\begin{align}
\label{case1.2k4}
&\pi''''(\C_{\gamma_{\mathscr{K}_{4}}* \nu})  \\
\label{case1.2k4.2}
& = \U_{0} \underset{l \neq b}{\underset{l < x}{\prod}}\U_{(\e_{l}-\e_{x-1}, 0)}(\alpha_{l \ov{x-1}}) \U_{(\e_{b}-\e_{x-1}, -1)}(\alpha_{b \ov{x-1}})\U_{(\e_{x-1}-\e_{x}, -1)}(\alpha_{x-1 \bar{x}}) \T^{w}  \\
\label{case1.2k4.3}
&=\U_{0}\underset{l \neq b, l \neq x-1}{\underset{l < x}{\prod}} \U_{(\e_{l}-\e_{x}, -1)}(\xi_{l \bar x}) \U_{(\e_{x-1}-\e_{x}, -1)}(\alpha_{x-1 \bar x}) \U_{(\e_{b}-\e_{x}, -2)}(\xi_{b \bar x})\T^{w}, \hbox{ where }
\end{align} 

\begin{align*}
\xi_{b \bar x} = c^{1,1}_{b\ov{x-1}, x-1 \bar x}(-\alpha_{b\ov{x-1}\alpha_{x-1 \bar x}})\\
\xi_{l \bar x} = c^{1,1}_{l\ov{x-1}, x-1 \bar x}(-\alpha_{l\ov{x-1}\alpha_{x-1 \bar x}})
\end{align*}

\noindent
and where the equality between (\ref{case1.2k4.2}) and (\ref{case1.2k4.3}) arises by using (\ref{chevalley}) and Proposition \ref{stable} applied to $\U_{(\e_{l}-\e_{x-1}, 0)}(\alpha_{l \ov{x-1}}) \U_{(\e_{b}-\e_{x-1}, -1)}(\alpha_{b \ov{x-1}}) \in \U_{\mu_{\gamma_{\mathscr{K}_{4}}}}$. The sets displayed in (\ref{case1.2k3}) and  (\ref{case1.2k4.3}) are equal as long as all the parameters are non-zero.  \\

\noindent
\textbf{Case 2}: $x = y < z, z \neq \bar x$\\
In this case we have $w_{1}= y \hbox{ }y \hbox{ } z$ and $w_{2} = y \hbox{ } z \hbox{ }y$. We want to look at

\begin{align*}
\pi(\C_{\gamma_{w_{1}}*\nu}) &= \mathbb{U}^{\gamma_{w_{1}}*\nu}_{0}\mathbb{U}^{\gamma_{w_{1}}*\nu}_{1}\mathbb{U}^{\gamma_{w_{1}}*\nu}_{2} \T^{w} \\
\pi(\C_{\gamma_{w_{2}}*\nu}) &= \mathbb{U}^{\gamma_{w_{2}}*\nu}_{0}\mathbb{U}^{\gamma_{w_{2}}*\nu}_{1}\mathbb{U}^{\gamma_{w_{2}}*\nu}_{2} \T^{w}
\end{align*}
In this case all generators of $\mathbb{U}^{\gamma_{w_{1}}*\nu}_{i}$ and of $\mathbb{U}^{\gamma_{w_{2}}*\nu}_{i}$ belong to $\U_{0}$ for $i \in \{1, 2 , 3\}$. Therefore Proposition \ref{stable} implies in this case that

\begin{align*}
\pi(\C_{\gamma_{w_{1}}*\nu}) = \U_{0}\T^{w} = \pi(\C_{\gamma_{w_{2}}*\nu}),
\end{align*}

which concludes the proof. \\

\noindent
\textbf{Case 3}: $x < y =z, z \neq \bar x$\\
For this case it will be convenient to use Proposition \ref{densewordreading}. Let $$\mathscr{K}_{1} = \Skew(0: \mbox{\tiny{y}}, \mbox{\tiny{x}}|1:\mbox{\tiny{y}})$$ and $$\mathscr{K}_{2} = \Skew(0: \mbox{\tiny{x}}, \mbox{\tiny{y}}|0: \mbox{\tiny{y}}).$$ It is then enough to show (by Proposition \ref{densewordreading}) that
\begin{align*}
\overline{\pi'(\C_{\gamma_{\mathscr{K}_{1}}*\nu})} = \overline{\pi''(\C_{\gamma_{\mathscr{K}_{2}}*\nu})},
\end{align*}
since 
\begin{align*}
w_{1} = x \hbox{ } y \hbox{ } y = w(\mathscr{K}_{1}) \hbox{ and }\\
w_{2} = y \hbox{ }x \hbox{ }y = w(\mathscr{K}_{2}). 
\end{align*}

\noindent
However, this case is now the same as the previous one: all generators of $\mathbb{U}^{\gamma_{\mathscr{K}_{1}}*\nu}_{i}$ and $\mathbb{U}^{\gamma_{\mathscr{K}_{2}}*\nu}_{i}$ belong to $\U_{0}$, therefore, as before, 
\begin{align*}
\pi'(\C_{\gamma_{\mathscr{K}_{1}}*\nu}) = \U_{0}\T^{w} = \pi''(\C_{\gamma_{\mathscr{K}_{2}}*\nu}).
\end{align*}
With this case we conclude the proof of Lemma \ref{denser1}. 
\end{proof}

\subsection*{Relation R2}
For $1< x \leq n$ and $x\leq y \leq \bar x$:
\begin{itemize}
\item[a.] $y\hbox{ } \ov{x-1}\hbox{ }x-1 \equiv y\hbox{ }x\hbox{ }\bar{x}$ and 
\item[b.] $\ov{x-1}\hbox{ }x-1\hbox{ }y \equiv x\hbox{ }\bar{x}\hbox{ }y$. 
\end{itemize}


\begin{lem}
\label{denser2}
Let 

\begin{align*}
w_{1} &= y\hbox{ } \ov{x-1}\hbox{ }x-1\\
w_{2} &=  y\hbox{ }x\hbox{ }\bar{x} \\
w_{3} &= \ov{x-1}\hbox{ }x-1\hbox{ }y \\
w_{4} &= x\hbox{ }\bar{x}\hbox{ }y
\end{align*}
\noindent
for $z \neq \bar x$. Then 

\begin{align*}
a) \overline{\pi(\C_{\gamma_{w_{1}}*\nu})} &= \overline{\pi(\C_{\gamma_{w_{2}}*\nu})}\\
b) \overline{\pi(\C_{\gamma_{w_{3}}*\nu})} &= \overline{\pi(\C_{\gamma_{w_{4}}*\nu})}
\end{align*}
\end{lem}

\begin{proof}
As usual, the proof is divided in some cases: we first consider the case where $y \notin \{x, \bar x\}$ and then we analyse $y = x$ and $ y = \bar x$ separately. \\

\noindent
\textbf{Case 1} $ y \notin \{x, \bar x\}$

a) We will use Proposition \ref{densewordreading}. Note that 
\begin{align*}
w_{1} = y\hbox{ } \ov{x-1} \hbox{ } x-1 = w\Bigg(\Skew(0:\mbox{\tiny{x-1}}, \mbox{\tiny{y}}, \mbox{\tiny{y}}| 1: \ov{\mbox{\tiny{x-1}}},\ov{\mbox{\tiny{x-1}}})\Bigg)
\end{align*}
and 
\begin{align*}
w_{2} = y \hbox{ } x \hbox{ } \bar{x} = w\Bigg( \Skew(0:\mbox{\tiny{x-1}}, \mbox{\tiny{x}},\mbox{\tiny{y}}|0: \mbox{\tiny{$\bar{x}$}}, \ov{\hbox{\tiny{x-1}}}) \Bigg).
\end{align*}
Hence by Proposition \ref{densewordreading}, to show Lemma \ref{denser2} a) it is enough to show that 
\begin{align*}
\overline{\pi'(\C_{\gamma_{\mathscr{K}_{1}}*\nu})} = \overline{\pi''(\C_{\gamma_{\mathscr{K}_{2}}*\nu})}, 
\end{align*}

where
\begin{align*}
 \mathscr{K}_{2} =  \Skew(0:\mbox{\tiny{x-1}}, \mbox{\tiny{x}},\mbox{\tiny{y}}|0: \ov{\mbox{\tiny{x}}}, \ov{\hbox{\tiny{x-1}}}) \hbox{ and }
\mathscr{K}_{1} = \Skew(0:\mbox{\tiny{x-1}}, \mbox{\tiny{y}}, \mbox{\tiny{y}}| 1: \ov{\mbox{\tiny{x-1}}},\ov{\mbox{\tiny{x-1}}}). 
\end{align*}

First we check 

\begin{align*}
 \overline{\pi''(\C_{\gamma_{\mathscr{K}_{2}}*\nu})} \subseteq \overline{\pi'(\C_{\gamma_{\mathscr{K}_{1}}*\nu})}.
\end{align*}

\noindent Clearly $\mathbb{U}^{\gamma_{\mathscr{K}_{2}}*\nu}_{0} \subset \U_{0}$; the only generators of $\mathbb{U}^{\gamma_{\mathscr{K}_{2}}*\nu}_{1}$ that do not belong to $\U_{0}$ are those of the form $\U_{(\e_{x}-\e_{y}, -1)}(a), a \in \mathbb{C}$, and those in $\mathbb{U}^{\gamma_{\mathscr{K}_{2}}*\nu}_{2}$ are $\U_{(\e_{x-1}-\e_{x}, -1)}(b),$ for $b \in \mathbb{C}$. This means that every element in $ \overline{\pi''(\C_{\gamma_{\mathscr{K}_{2}}*\nu})}$ belongs to 
\begin{align*}
 u \U_{(\e_{x} - \e_{y}, -1)}(a)\U_{(\e_{x-1}-\e_{x}, -1)}(b)\T^{w}
\end{align*}
for some $u \in \U_{0}$. Both $\U_{(\e_{x} - \e_{y}, -1)}(a)$ and $\U_{(\e_{x-1}-\e_{x}, -1)}(b)$ belong to $\U_{\e_{y}-\e_{x-1}}$, and this implies the contention by Proposition \ref{stable} and Corollary \ref{goodtrick}. Now we want to show 
\begin{align*}
 \overline{\pi'(\C_{\gamma_{\mathscr{K}_{1}}*\nu})} \subseteq \overline{\pi''(\C_{\gamma_{\mathscr{K}_{2}}*\nu})}.
\end{align*}
\noindent
By Theorem \ref{celldescription}, all elements of $\pi'(\C_{\gamma_{\mathscr{K}_{1}}*\nu})$ belong to 
\small{
\begin{align}
\label{r2ak2}
u \U_{(\e_{x-1}-\e_{y}, -2)}(v_{x-1 \bar{y}})\U_{(\e_{x-1}, -1)}(v_{x-1})\underset{l \neq y}{\underset{l \geq x}{\prod}}\U_{(\e_{x-1}-\e_{l}, -1)}(v_{x-1 \bar{l}}) \underset{s \neq y}{\prod}\U_{(\e_{x-1}+\e_{s}, -1)}(v_{x-1 s}) \T^{w}
\end{align}
}
\normalsize
for $u \in \U_{0}$ and $v_{x-1 j} \in \mathbb{C}$. This is because both $\mathbb{U}^{\gamma_{\mathscr{K}_{1}}*\nu}_{0}$ and $\mathbb{U}^{\gamma_{\mathscr{K}_{1}}*\nu}_{1}$ are contained in $\U_{0}$. Fix such an element such that $v_{x-1 \bar x} \neq 0$. We know that $\U_{(\e_{x-1}-\e_{x}, -1)}(v_{x-1 \bar x}) \in \mathbb{U}^{\gamma_{\mathscr{K}_{2}}*\nu}_{2}$ and that for any $a_{x\bar y} \in \mathbb{C}, \U_{(\e_{x}-\e_{y},-1)}(a_{x \bar y}) \in \U_{\e_{y}}$; this means that these elements stabilise both the truncated images $\T^{\geq 3}_{\gamma_{\mathscr{K}_{2}}*\nu}$ and $\T^{\geq 1}_{\gamma_{\mathscr{K}_{2}}*\nu}$. Hence the elements in 
\begin{align}
& \U_{(\e_{x-1}-\e_{x}, -1)}(v_{x-1 \bar x})\U_{(\e_{x}-\e_{y}, -1)}(v_{x\bar y}) \T^{w} = \\
\label{r2.1}
&\U_{(\e_{x}-\e_{y}, -1)}(v_{x\bar y})\U_{(\e_{x-1}-\e_{y}, -2)}(c^{1,1}_{x-1 \bar x, x\bar y}(-v_{x-1 \bar x}) a_{x \bar y}) \U_{(\e_{x-1}-\e_{x}, -1)}(v_{x-1 \bar x}) \T^{w}
\end{align}
all belong to $\pi''(\C_{\gamma_{\mathscr{K}_{2}}*\nu})$; more precisely to $\mathbb{U}^{\gamma_{\mathscr{K}_{1}}*\nu}_{2} \T^{w} \subset \T^{\geq 1}_{\gamma_{\mathscr{K}_{1}}*\nu}$, hence by Proposition \ref{stable}, we may multiply by $\U_{(\e_{x}-\e_{y}, -1)}(-v_{x\bar y})$ on the left of line (\ref{r2.1}) and the product still belongs to 
$\pi''(\C_{\gamma_{\mathscr{K}_{2}}*\nu})$, hence 
\begin{align*}
\U_{(\e_{x-1}-\e_{y}, -2)}(c^{1,1}_{x-1 \bar x, x\bar y}(-v_{x-1 \bar x}) a_{x \bar y}) \U_{(\e_{x-1}-\e_{x}, -1)}(v_{x-1 \bar x}) \T^{w} \subset \pi''(\C_{\gamma_{\mathscr{K}_{2}}*\nu}).
\end{align*}
Now consider the product
\begin{align*}
u = \U_{(\e_{y}+\e_{x}, 1)}(a_{yx})\U_{(\e_{x}, 0)}(a_{x}) \underset{l \neq y}{\underset{l > x}{\prod}} \U_{(\e_{x}-\e_{l}, 0)}(a_{x\bar l}) \underset{s \neq y} {\prod} \U_{(\e_{x}+\e_{s}, 0)}(a_{xs}) \in \U_{\e_{y}}\cap \U_{0}.
\end{align*}
Proposition \ref{stable} then implies that 
\begin{align*}
&\pi(\C_{\gamma_{\mathscr{K}_{2}}*\nu}) \supset \\
&u^{-1}\U_{(\e_{x-1}-\e_{y}, -2)}(c^{1,1}_{x-1 \bar x, x\bar y}(-v_{x-1 \bar x}) a_{x \bar y}) \U_{(\e_{x-1}-\e_{x}, -1)}(v_{x-1 \bar x})u \T^{w}  = \\
&\U_{(\e_{x-1}+\e_{x}, -1)}(\rho_{x-1 x})\U_{(\e_{x-1}, -1)}(\rho_{x-1}) \U_{(\e_{x-1}-\e_{y}, -2)}(\rho_{x-1 y}) \\
&\underset{l \neq y}{\underset{l >x}{\prod}}\U_{(\e_{x-1}-\e_{l}, -1)}(\rho_{x-1 l}) \underset{s \neq y }{\prod}\U_{(\e_{x-1}+\e_{s}, -1)}(\rho_{x-1 s})\U_{(\e_{x-1}-\e_{x}, -1)}(v_{x-1 \bar{x}})\T^{w} \\
&\rho_{x-1 x} = c^{1,2}_{x-1\bar x, x}(-v_{x-1 \bar x})a_{x}^{2} - c^{1,1}_{x-1 y, yx} c^{1,1}_{x-1 \bar x, x\bar y}(v_{x-1 \bar x}) a_{x \bar y}a_{y x}\\
&\rho_{x-1 j} = c^{1,1}_{x-1 \bar{x}, xj}(-v_{x-1 \bar x}) a_{x j} j \neq y, j \in \{\bar{l}: l> x\} \cup \{s: \e_{x-1}+\e_{s} \in \Phi^{+}\}\\
&\rho_{x-1} = c^{1,1}_{x-1 \bar{x}, x}(-v_{x-1 \bar x}) a_{x}
\end{align*}
The system of equations defined by $v_{x-1} = \rho_{x-1}, v_{x-1 j} = \rho_{x-1 j}$ has indeed solutions (the variables are $a_{x}, a_{yx}, a_{x\bar l},$ and $a_{xs}$) since $v_{x-1, x} \neq 0$! This means that for such solutions  (cf. (\ref{r2ak2}))
\begin{align*}
&\U_{(\e_{x-1}-\e_{y}, -2)}(v_{x-1 \bar{y}})\U_{(\e_{x-1}, -1)}(v_{x-1})\underset{l \neq y}{\underset{l \geq x}{\prod}}\U_{(\e_{x-1}-\e_{l}, -1)}(v_{x-1 \bar{l}}) \underset{s \neq y}{\prod}\U_{(\e_{x-1}+\e_{s}, -1)}(v_{x-1 s}) \T^{w} = \\
&\U_{(\e_{x-1}+\e_{x}, -1)}(\rho_{x-1 x})\U_{(\e_{x-1}, -1)}(\rho_{x-1}) \U_{(\e_{x-1}-\e_{y}, -2)}(\rho_{x-1 y}) \\
&\underset{l \neq y}{\underset{l >x}{\prod}}\U_{(\e_{x-1}-\e_{l}, -1)}(\rho_{x-1 l}) \underset{s \neq y }{\prod}\U_{(\e_{x-1}+\e_{s}, -1)}(\rho_{x-1 s})\U_{(\e_{x-1}-\e_{x}, -1)}(v_{x-1 \bar{x}})\T^{w} \subset \pi(\C_{\gamma_{\mathscr{K}_{2}}*\nu})
\end{align*}
\noindent
and so by Proposition \ref{stable} we get that all elements in (\ref{r2ak2}) belong to $\pi''(\C_{\gamma_{\mathscr{K}_{2}}*\nu})$. All such elements of $\pi'(\C_{\gamma_{\mathscr{K}_{1}}*\nu})$ form a dense open subset. This finishes the proof in this case. \\
\noindent b) Let 
$$
\mathscr{K}_{3} = \Skew(0:\mbox{\tiny{x-1}}, \mbox{\tiny{x-1}}, \ov{\mbox{\tiny{x-1}}}|0: \mbox{\tiny{y}}, \mbox{\tiny{y}})$$    and  $$\mathscr{K}_{4} = \Skew(0:\mbox{\tiny{y}},\mbox{\tiny{x-1}}, \mbox{\tiny{x}}| 1: \ov{\mbox{\tiny{x}}}, \ov{\mbox{\tiny{x-1}}}).$$

Then $w_{3} = \ov{x-1}\hbox{ }x-1\hbox{ }y = w(\mathscr{K}_{3})$ and $w_{4} = x\hbox{ }\bar{x}\hbox{ }y = w(\mathscr{K}_{4})$. As in a), by Proposition \ref{densewordreading}, it is enough to show

\begin{align*}
\overline{\pi'''(\C_{\gamma_{\mathscr{K}_{3}}*\nu})} = \overline{\pi''''(\C_{\gamma_{\mathscr{K}_{4}}*\nu})}. 
\end{align*}
To show 
\begin{align*}
\overline{\pi''''(\C_{\gamma_{\mathscr{K}_{4}}*\nu})} \subset \overline{\pi'''(\C_{\gamma_{\mathscr{K}_{3}}*\nu})},  
\end{align*}
note first that the only generator of $\mathbb{U}^{\gamma_{\mathscr{K}_{4}}*\nu}_{i}$ that does not belong to $\U_{0}$ is $$\U_{(\e_{x-1}-\e_{x}, -1)}(a) \in \mathbb{U}^{\gamma_{\mathscr{K}_{4}}*\nu}_{1}, \hbox{ for } a \in \mathbb{C}.$$ Of $\mathbb{U}^{\gamma_{\mathscr{K}_{4}}*\nu}_{2}$, the only generators that do not commute with $\U_{(\e_{x-1}-\e_{x}, -1)}(a)$ are $\U_{(\e_{y}+\e_{x}, 0)}(b),$ with $b \in \mathbb{C}$. Then Chevalley's commutator formula (\ref{chevalley}) implies that all elements of $\pi''''(\C_{\gamma_{\mathscr{K}_{4}}*\nu})$ belong to the set
\begin{align}
\label{r2imagek4}
\U_{0}\U_{(\e_{x-1}+\e_{y}, -1)}(c^{1,1}_{x-1 \bar x, xy}(-a)b) \U_{(\e_{x-1}-\e_{x}, -1)}(a) \T^{w}.
\end{align}
Since both $\U_{(\e_{x-1}+\e_{y}, -1)}(c^{1,1}_{x-1 \bar x, xy}(-a)b)$ and $\U_{(\e_{x-1}-\e_{x}, -1)}(a)$ belong to $\mathbb{U}^{\gamma_{\mathscr{K}_{3}}*\nu}_{1}$, the desired contention follows by Proposition \ref{stable}. Now we show 
\begin{align}
\label{r2bgeneralcaselastcontention}
\overline{\pi'''(\C_{\gamma_{\mathscr{K}_{3}}})} \subset \overline{\pi''''(\C_{\gamma_{\mathscr{K}_{4}}})}.
\end{align}
The proof is similar to that of a), but there are some subtle differences. First we look at the image $\pi'''(\C_{\gamma_{\mathscr{K}_{3}}*\nu})$. Out of all the generators of $\mathbb{U}^{\gamma_{\mathscr{K}_{3}}*\nu}_{i}$, the only ones that do not belong to $\U_{0}$ belong to $\mathbb{U}^{\gamma_{\mathscr{K}_{3}}*\nu}_{1}$: $\U_{(\e_{x-1}, -1)}(v_{x}), \U_{(\e_{x-1}-\e_{s}, -1)}(v_{x-1 s}),$ and $\U_{(\e_{x-1}+\e_{l}, -1)}(v_{x-1 l})$ for $l \neq x-1, s>x, s \neq y$, and complex numbers $v_{x-1}, v_{x-1 s}$, and $v_{x-1 l}$. The group $\mathbb{U}^{\gamma_{\mathscr{K}_{3}}*\nu}_{2}$ has as generators (only) the terms $\U_{(\e_{x-1}+\e_{y}, 0)}(a)$, and these commute with all the latter terms. Therefore all elements of $\pi'''(\C_{\gamma_{\mathscr{K}_{3}}*\nu})$ belong to 
\begin{align}
\label{r2generalcaseexpressionfork3}
u \U_{(\e_{x-1}, -1)}(v_{x}) \underset{s \neq y}{\underset{s> x-1}{\prod}}\U_{(\e_{x-1}-\e_{s}, -1)}(v_{x-1 s}) \underset{l \neq x-1}{\prod}\U_{(\e_{x-1}+\e_{l}, -1)}(v_{x-1 l}) \T^{w}
\end{align}
for some  $u \in \U_{0}$. Fix such a $u$, and assume $v_{x-1 \bar x} \neq 0$ and $v_{x-1y} \neq 0$. Such elements as (\ref{r2generalcaseexpressionfork3}) form a dense open subset of  $\pi'''(\C_{\gamma_{\mathscr{K}_{3}}*\nu})$. Now, for all complex numbers $a, a_{xy}$, and $a_{x\bar y}$ we have $\U_{(\e_{x-1}-\e_{x}, -1)}(a) \in \mathbb{U}^{\gamma_{\mathscr{K}_{4}}* \nu}_{1}$, $\U_{(\e_{x} + \e_{y}, 0)}(a_{xy}) \in \mathbb{U}^{\gamma_{\mathscr{K}_{4}}* \nu}_{1}$, and $\U_{(\e_{x}-\e_{y}, 0)}(a_{x\bar{y}}) \in \U_{0}$, which stabilises the truncated image $\T^{\geq 2}_{\gamma_{\mathscr{K}_{4}}* \nu}$. Therefore, setting $c = \U_{(\e_{x} + \e_{y}, 0)}(a_{xy}) \U_{(\e_{x}-\e_{y}, 0)}(a_{x\bar{y}}) \in \U_{0}$, all elements in 
\small{
\begin{align*}
&c^{-1}\U_{(\e_{x-1}-\e_{x}, -1)}(a) c \T^{w} = \\
&\U_{(\e_{x-1}+\e_{x}, -1)}(\varrho_{x-1 x})\U_{(\e_{x-1}+\e_{y}, -1)}(\varrho_{x-1 y}) \U_{(\e_{x-1}-\e_{y}, -1)}(c^{1,1}_{x-1 x, x \bar y}(-a)a_{x\bar y})\U_{(\e_{x-1}-\e_{x}, -1)}(a) \T^{w} = \\
&\U_{(\e_{x-1}+\e_{x}, -1)}(\varrho_{x-1 x})\U_{(\e_{x-1}+\e_{y}, -1)}(\varrho_{x-1 y}) \U_{(\e_{x-1}-\e_{x}, -1)}(a) \U_{(\e_{x-1}-\e_{y}, -1)}(c^{1,1}_{x-1 x, x \bar y}(-a)a_{x\bar y}) \T^{w} =\\
&\U_{(\e_{x-1}+\e_{x}, -1)}(\varrho_{x-1 x})\U_{(\e_{x-1}+\e_{y}, -1)}(\varrho_{x-1 y}) \U_{(\e_{x-1}-\e_{x}, -1)}(a)\T^{w}
\end{align*}
}
\normalsize
belong to $\pi''''(\C_{\gamma_{\mathscr{K}_{4}}*\nu})$, where 
\begin{align*}
\varrho_{x-1 x} &= c^{1,1}_{x-1 y, x\bar y}c^{1,1}_{x-1 \bar{x}, xy}aa_{xy}a_{x\bar y}\\
\varrho_{x-1 y} &= c^{1,1}_{x-1 \bar{x}, xy}(-a)a_{xy}, 
\end{align*}

\noindent and where the last equality holds because $\U_{(\e_{x-1}-\e_{y}, -1)}(c^{1,1}_{x-1 x, x \bar y}(-a)a_{x\bar y}) \in \U_{\e_{y}}$, and all elements of the latter stabilise the truncated image $\T^{w}$ by Proposition \ref{stable}. Now let 
$$c' = \U_{(\e_{x}, 0)}(a_{x}) \underset{s \neq y}{\underset{s > x}{\prod}}\U_{(\e_{x}-\e_{s}, 0)}(a_{x\bar s}) \underset{l \neq y}{\underset{l \neq x-1}{\prod}}\U_{(\e_{x}+\e_{l}, 0)}(a_{xl}) \in \U_{\e_{y}}\cap \U_{0}$$ \noindent for $a_{x}, a_{x\bar s},$ and $a_{xl}$ complex numbers; by Proposition \ref{stable} this element stabilises the truncated image $\T^{w}$ and the image $\pi''''(\C_{\gamma_{\mathscr{K}_{4}}*\nu})$. Therefore
\begin{align}
&\pi''''(\C_{\gamma_{\mathscr{K}_{4}}}) \supset \\
&c'^{-1} \U_{(\e_{x-1}+\e_{x}, -1)}(\varrho_{x-1 x})\U_{(\e_{x-1}+\e_{y}, -1)}(\varrho_{x-1 y}) \U_{(\e_{x-1}-\e_{x}, -1)}(a) c' \T^{w} = \\
\label{r2bgral}
&\U_{(\e_{x-1}, -1)}(\varrho_{x}) \underset{ s \neq x}{\underset{s \neq y}{\underset{s> x-1}{\prod}}}\U_{(\e_{x-1}-\e_{s}, -1)}(\varrho_{x-1 s})\U_{(\e_{x-1} - \e_{x}, -1)}(a)\U_{(\e_{x-1}+\e_{x}, -1)}(\varrho'_{x-1 x}) \\
\label{r2bgrall}
&\underset{l \notin  \{x-1, x\}}{\prod}\U_{(\e_{x-1}+\e_{l}, -1)}(\varrho_{x-1 l})\T^{w}
\end{align}
where
\begin{align*}
\varrho_{x-1} & = c^{1,1}_{x-1 x, x}(-a)a_{x}\\
\varrho'_{x-1 x} &= \varrho_{x-1 x} + c^{1,2}_{x-1 x, x}(-a)a^{2}_{x}\\
\varrho_{x-1 l} & = c^{1,1}_{x-1 \bar{x}, xl}(-a)a_{xl}\\
\varrho_{x-1 \bar s} & = c^{1,1}_{x-1 \bar{x}, x\bar s}(-a)a_{x\bar s}. 
\end{align*}

\noindent
We want to show that $\U_{(\e_{x-1}, -1)}(v_{x-1}) \underset{s \neq y}{\underset{s> x-1}{\prod}}\U_{(\e_{x-1}-\e_{s}, -1)}(v_{x-1 s}) \underset{l \neq x-1}{\prod}\U_{(\e_{x-1}+\e_{l}, -1)}(v_{x-1 l}) \T^{w}$ is equal to the product in the  last lines (\ref{r2bgral}) and (\ref{r2bgrall}) above (cf. (\ref{r2generalcaseexpressionfork3})), for some $a_{x}, a_{xl}, a_{x\bar s}$. This determines a system of equations 

\begin{align*}
v_{x-1 \bar x} &= a \\
v_{x-1 x} &= c^{1,1}_{x-1 y, x\bar y}c^{1,1}_{x-1 \bar{x}, xy}aa_{xy}a_{x\bar y} +  c^{1,2}_{x-1 x, x}(-a)a^{2}_{x}\\
v_{x-1} &= c^{1,1}_{x-1 x, x}(-a)a_{x}\\
v_{x-1 \bar s} &= c^{1,1}_{x-1 \bar{x}, x\bar s}(-a)a_{x\bar s}\\
v_{x-1 l} &= c^{1,1}_{x-1 \bar{x}, xl}(-a)a_{xl}\\
v_{x-1 y} &= c^{1,1}_{x-1 \bar{x}, xy}(-a)a_{xy}.
\end{align*}
\noindent
which can always be solved since $v_{x-1 y} \neq 0$ and $v_{x \ov{x-1}} \neq 0$. This completes the proof of b) in this case!\\

\noindent
\textbf{Case 2} $y = x$\\

\noindent a) As in Case 1, we will make use of Proposition \ref{densewordreading}. Let $$\mathscr{K}_{1} = \Skew(0:\mbox{\tiny{x-1}}, \mbox{\tiny{x}}, \mbox{\tiny{x}}| 1: \ov{\mbox{\tiny{x-1}}}, \ov{\mbox{\tiny{x-1}}})$$ and 
$$\mathscr{K}_{2} = \Skew(0: \mbox{\tiny{x-1}}, \mbox{\tiny{x}}, \mbox{\tiny{x}}| 0: \ov{\mbox{\tiny{x}}}, \ov{\mbox{\tiny{x-1}}}).$$ Then
\begin{align*}
w_{1} &= x\hbox{ }\ov{x-1} \hbox{ } x-1 = w(\mathscr{K}_{1}) \hbox{ and }\\
w_{2} &= x\hbox{ }x\hbox{ }\ov{x} = w(\mathscr{K}_{2}).
\end{align*}
By Proposition \ref{densewordreading} it is enough to show 
\begin{align*}
\overline{\pi'(\C_{\gamma_{\mathscr{K}_{1}}*\nu})} = \overline{\pi''(\C_{\gamma_{\mathscr{K}_{2}}*\nu})}. 
\end{align*}
First we show 
\begin{align}
\label{r2ayequalsxeasy}
\overline{\pi''(\C_{\gamma_{\mathscr{K}_{2}}*\nu})} \subseteq \overline{\pi'(\C_{\gamma_{\mathscr{K}_{1}}*\nu})}.
\end{align}
Since $\mathbb{U}^{\gamma_{\mathscr{K}_{2}}*\nu}_{2}$ is generated by elements of the form $\U_{(\e_{x-1}-\e_{x}, -2)}(a), a \in \mathbb{C}$ and the generators of $\mathbb{U}^{\gamma_{\mathscr{K}_{2}}*\nu}_{i}$ belong to $\U_{0}$ for $i \in \{1, 2\}$, all elements of $\pi''(\C_{\gamma_{\mathscr{K}_{2}}*\nu})$ are of the form 
\begin{align*}
u \U_{(\e_{x-1}-\e_{x}, -2)}(a) \T^{w}
\end{align*}
\noindent
for some $u \in \U_{0}$. Since $\U_{(\e_{x-1}-\e_{x}, -2)}(a) \in \mathbb{U}^{\gamma_{\mathscr{K}_{1}}*\nu}_{2}$, (\ref{r2ayequalsxeasy}) follows by applying Proposition \ref{stable} to $u$. To finish the proof in this case it remains to show 
\begin{align}
\label{r2ayequalsxhard}
\overline{\pi'(\C_{\gamma_{\mathscr{K}_{1}}*\nu})} \subseteq \overline{\pi''(\C_{\gamma_{\mathscr{K}_{2}}*\nu})}.
\end{align}
The generators of $\mathbb{U}^{\gamma_{\mathscr{K}_{1}}*\nu}_{i}$ belong to $\U_{0}$ for $i \in \{0,1\}$, and the generators of $\mathbb{U}^{\gamma_{\mathscr{K}_{2}}*\nu}_{2}$ that do not are $\U_{(\e_{x-1}, -1)}(v_{x}), \U_{(\e_{x-1}-\e_{l}, -1)}(v_{x-1 \bar l}), \U_{(\e_{x-1} + \e_{s}, -1)}(v_{x-1 s}),$ and \newline $\U_{(\e_{x-1}-\e_{x}, -2)}(v_{x-1 \bar x})$, for $n\geq l> x, s \notin \{x, x-1\}$, and  complex numbers $v_{x}, v_{x-1 \bar l}, v_{x-1 s},$ and $v_{x-1 \bar x}$. Therefore all elements of $\pi'(\C_{\gamma_{\mathscr{K}_{1}}*\nu})$ belong to 
\begin{align*}
u \U_{(\e_{x-1}, -1)}(v_{x})\U_{(\e_{x-1}-\e_{l}, -1)}(v_{x-1 \bar l})\U_{(\e_{x-1} + \e_{s}, -1)}(v_{x-1 s})\U_{(\e_{x-1}-\e_{x}, -2)}(v_{x-1 \bar x}) \T^{w}.
\end{align*}
Fix such $u \in \U_{0}$ and $v_{x}, v_{x-1 \bar l}, v_{x-1 s}, v_{x-1 \bar x}$ complex numbers such that $v_{x-1 \bar x} \neq 0$. We know that for any $a \in \mathbb{C}, \U_{(\e_{x-1}-\e_{x}, -2)}(a) \in \mathbb{U}_{\gamma_{\mathscr{K}_{2}}*\nu}$; let $$q = \U_{(\e_{x}, 1)}(a_{x}) \underset{s> x}{\prod}\U_{(\e_{x}-\e_{s}, 1)}(a_{x\bar s}) \underset{l \neq x}{\prod}\U_{(\e_{x}+\e_{l},1)}(a_{xl}) \in \U_{(\e_{x})}\cap \U_{0}$$ for any complex numbers $a_{x}, a_{x\bar s}, a_{xl}$. Then by Proposition \ref{stable}, 
\begin{align}
\label{r2ayisxlastbit}
q^{-1}\U_{(\e_{x-1}-\e_{x}, -2)}(a) q \T^{w} \subset \pi''(\C_{\gamma_{\mathscr{K}_{2}}*\nu}). 
\end{align}
As in the previous cases, we want to find $a, a_{x}, a_{x\bar s}, a_{xl}$ such that 
\begin{align*}
t \U_{(\e_{x-1}, -1)}(v_{x})\U_{(\e_{x-1}-\e_{l}, -1)}(v_{x-1 \bar l})\U_{(\e_{x-1} + \e_{s}, -1)}(v_{x-1 s})\U_{(\e_{x-1}-\e_{x}, -2)}(v_{x-1 \bar x}) \T^{w}
\end{align*}
equals (\ref{r2ayisxlastbit}), for some $t \in \U_{0}$. But 
\begin{align*}
&q^{-1}\U_{(\e_{x-1}-\e_{x}, -2)}(a) q \T^{w} = \\
& t^{-1} \U_{(\e_{x-1}, -1)}(\varrho_{x})\U_{(\e_{x-1}-\e_{l}, -1)}(\varrho_{x-1 \bar l})\U_{(\e_{x-1} + \e_{s}, -1)}(\varrho_{x-1 s})\U_{(\e_{x-1}-\e_{x}, -2)}(a) \T^{w}
\end{align*}
where
\begin{align*}
t^{-1} &= \U_{(\e_{x}+\e_{x-1}, 0)}(c^{1,2}_{x-1 \bar x, x})(-a)a_{x}^{2} \in \U_{0}\\
\varrho_{x}&= c^{1,1}_{x-1 \bar x, x}(-a)a_{x} \\
\varrho_{x-1 \bar l}&= c^{1,1}_{x-1 \bar x, x\bar l}(-a)a_{x\bar l} \\
\varrho_{x-1 s}&= c^{1,1}_{x-1 \bar x, xs}(-a)a_{xs}.
\end{align*}

The system
\begin{align*}
v_{x-1 \bar x} &= a\\
v_{x-1 \bar l}&= \varrho_{x-1 \bar l}\\
v_{x-1 s} &= \varrho_{x-1 s}
\end{align*}
always has a solution since $v_{x-1 \bar x} \neq 0$. This concludes the proof. 
\item[b)] Let $$\mathscr{K}_{3}= \Skew(0: \mbox{\tiny{x-1}}, \mbox{\tiny{x-1}}, \ov{\mbox{\tiny{x-1}}}|0: \mbox{\tiny{x}}, \mbox{\tiny{x}})$$ and $$\mathscr{K}_{4}= \Skew(0:\mbox{\tiny{x}},\mbox{\tiny{x-1}}, \mbox{\tiny{x}}|1: \ov{\mbox{\tiny{x}}}, \ov{\mbox{\tiny{x-1}}}).$$ Then 
\begin{align*}
w_{3} = \ov{x-1}\hbox{ } x-1\hbox{ } x  = w(\mathscr{K}_{3}) \hbox{ and }
w_{4} =  x \hbox{ }\bar x \hbox{ } x = w(\mathscr{K}_{4}).
\end{align*}
By Proposition \ref{densewordreading} it is enough to show 
\begin{align*}
\overline{\pi'''(\C_{\gamma_{\mathscr{K}_{3}}*\nu})} = \overline{\pi''''(\C_{\gamma_{\mathscr{K}_{4}}*\nu})}. 
\end{align*}
To do this we will describe a common dense subset of $\pi'''(\C_{\gamma_{\mathscr{K}_{3}}*\nu})$ and $\pi''''(\C_{\gamma_{\mathscr{K}_{4}}*\nu})$. \\

Consider first
$
 \pi'''(\C_{\gamma_{\mathscr{K}_{3}}*\nu}) = \mathbb{U}^{\gamma_{\mathscr{K}_{3}}*\nu}_{0} \mathbb{U}^{\gamma_{\mathscr{K}_{3}}*\nu}_{1}\mathbb{U}^{\gamma_{\mathscr{K}_{3}}*\nu}_{2}\T^{w}. 
$
We have $\mathbb{U}^{\gamma_{\mathscr{K}_{3}}*\nu}_{0} \subset \U_{0}$ and also $\mathbb{U}^{\gamma_{\mathscr{K}_{3}}*\nu}_{2} \subset \U_{0}$, since it is generated by the terms $\U_{(\e_{x-1}+\e_{x}, 0)}(d), d \in \mathbb{C}$. These commute with all generators of $\mathbb{U}^{\gamma_{\mathscr{K}_{3}}*\nu}_{1}$, out of which $\U_{(\e_{x-1}, -1)}(v_{x-1}), \U_{(\e_{x-1}+\e_{s}, -1)}(v_{x-1 s}),$ and $\U_{(\e_{x-1}-\e_{l}, -1)}(v_{x-1 \bar l})$, (for $s\leq n, s \neq x-1, l>x$, and $v_{x-1}, v_{x-1 s}$ and $v_{x-1 \bar l}$ complex numbers) do not belong to $\U_{0}$. Therefore 
$ \pi'''(\C_{\gamma_{\mathscr{K}_{3}}*\nu})$ coincides with 
\begin{align}
\label{r2case2k3}
\U_{0} \U_{(\e_{x-1}, -1)}(v_{x-1})\underset{s \neq x-1}{\underset{s\leq n}{\prod}}\U_{(\e_{x-1}+\e_{s}, -1)}(v_{x-1 s})\underset{x < l \leq n}{\prod}\U_{(\e_{x-1}-\e_{l}, -1)}(v_{x-1 \bar l}) \T^{w}
\end{align}
for complex numbers $v_{x-1}, v_{x-1 s}$ and $v_{x-1 \bar l}$. Now we look at elements of
\begin{align*}
\pi''''(\C_{\gamma_{\mathscr{K}_{4}}*\nu}) = \mathbb{U}^{\gamma_{\mathscr{K}_{4}}*\nu}_{0} \mathbb{U}^{\gamma_{\mathscr{K}_{4}}*\nu}_{1}\mathbb{U}^{\gamma_{\mathscr{K}_{4}}*\nu}_{2}\T^{w}. 
\end{align*}
Both $\mathbb{U}^{\gamma_{\mathscr{K}_{4}}*\nu}_{0}$ and $\mathbb{U}^{\gamma_{\mathscr{K}_{4}}*\nu}_{2}$ are contained in $\U_{0}$, and $\mathbb{U}^{\gamma_{\mathscr{K}_{4}}*\nu}_{1}$ is generated by the elements $\U_{(\e_{x-1}-\e_{x}, -1)}(d)$, which belong to $\U_{\e_{x}}$ and therefore stabilise the truncated image $\T^{w}$ by Proposition \ref{stable}. Now, by Proposition \ref{order}, we may write any element $k$ of $\mathbb{U}^{\gamma_{\mathscr{K}_{4}}*\nu}_{2}$ as
\begin{align*}
k = \U_{(\e_{x}, 0)}(k_{x}) \underset{x< l \leq n}{\prod}\U_{(\e_{x}-\e_{l}, 0)}(k_{x \bar l}) \underset{s\leq n s\neq x}{\prod}\U_{(\e_{x}+\e_{s}, 0)}(k_{xs}) \in \U_{0}
\end{align*}
for some complex numbers $k_{x}, k_{x \bar l}$, and $k_{xs}$. Theorem \ref{celldescription} and Proposition \ref{stable} imply that 
\small{
\begin{align}
\label{r2case2k4}
&\pi''''(\C_{\gamma_{\mathscr{K}_{4}}*\nu})  = \U_{0}\U_{(\e_{x-1}-\e_{x}, -1)}(d)k \T^{w} \\
& = \U_{0}k \U_{(\e_{x-1}, -1 )}(\sigma_{x-1}) \U_{(\e_{x-1} + \e_{x}, -1)}(\sigma_{x-1 x})
\underset{x< l \leq n}{\prod}\U_{(\e_{x-1 }-\e_{l}, -1)}(\sigma_{x-1\bar l})\\
&\underset{s\leq n s\neq x}{\prod}\U_{(\e_{x-1}+\e_{s}, 0)}(\sigma_{x-1 s}) \U_{(\e_{x-1}-\e_{x}, -1)}(d)\T^{w}
\end{align}}
\normalsize
for $k \in \mathbb{U}^{\gamma_{\mathscr{K}_{4}}*\nu}_{2}$ and $d \in \mathbb{C}$, where 
\begin{align*}
\sigma_{x-1} = c^{1,1}_{x-1 \bar x, x}(-d)k_{x}\\
\sigma_{x-1 x} = c^{1,2}_{x-1 \bar x, x}(-d)k^{2}_{x}\\
\sigma_{x-1 \bar l} = c^{1,1}_{x-1 \bar x, x\bar l}(-d)k_{x \bar l}\\
\sigma_{x-1 s} = c^{1,1}_{x-1 \bar x, x s}(-d)k_{xs}
\end{align*}
 This set (\ref{r2case2k4}) is clearly contained in (\ref{r2case2k3}). Moreover, the system 
 \begin{align*}
v_{x-1} = \sigma_{x-1}\\
v_{x-1 x} = \sigma_{x-1 x}\\
v_{x-1 \bar l} = \sigma_{x-1 \bar l} \\
v_{x-1 s} = \sigma_{x-1 s}
 \end{align*}
has solutions for $d, k_{x}, k_{x \bar l},$ and $k_{xs}$ as long as $\{v_{x-1}, v_{x-1 x}, v_{x-1 \bar l}, v_{x-1 s}\} \subset \mathbb{C}^{\times}$. Proposition \ref{stable} then implies that a dense subset of $\pi'''(\C_{\gamma_{\mathscr{K}_{3}}*\nu})$ is contained in $\pi''''(\C_{\gamma_{\mathscr{K}_{4}}*\nu})$, which finishes the proof in this case.


\textbf{Case 3} $y = \bar{x}$\\

\noindent a) Let $$\mathscr{K}_{1} = \Skew(0: \mbox{\tiny{x-1}}, \ov{\mbox{\tiny{x}}}, \ov{\mbox{\tiny{x}}}|1:  \ov{\mbox{\tiny{x-1}}},  \ov{\mbox{\tiny{x-1}}})$$ \hbox{ and } $$\mathscr{K}_{2} = \Skew(0: \mbox{\tiny{x-1}}, \mbox{\tiny{x}}, \ov{\mbox{\tiny{x}}}|0: \ov{\mbox{\tiny{x}}}, \ov{\mbox{\tiny{x-1}}}).$$ Then 

\begin{align*}
w_{1} = \bar{x} \hbox{ } \ov{x-1} \hbox{ } x-1 = w(\mathscr{K}_{1}) \hbox{ and } w_{2} = \bar{x} \hbox{ } x \hbox{ } \bar{x} = w(\mathscr{K}_{2})
\end{align*}

By Proposition \ref{densewordreading} it is enough to show

\begin{align*}
\overline{\pi'(\C_{\gamma_{\mathscr{K}_{1}}*\nu})} = \overline{\pi''(\C_{\gamma_{\mathscr{K}_{2}}*\nu})}
\end{align*}

In this case we have $\mathbb{U}^{\gamma_{\mathscr{K}_{1}}*\nu}_{0} = 1 = \mathbb{U}^{\gamma_{\mathscr{K}_{1}}*\nu}_{0}$; Proposition \ref{order} and Theorem \ref{celldescription} then say 

\begin{align}
&\pi'(\C_{\gamma_{\mathscr{K}_{1}}*\nu}) =\\
\label{r2case3k1}
&\U_{(\e_{x-1}-\e_{x}, 0)}(v_{x-1 x}) \U_{(\e_{x-1},-1)}(v_{x-1})\U_{(\e_{x-1}+\e_{x}, -2)}(v_{x-1 x})\\
&\underset{x< l \leq n}{\prod}\U_{(\e_{x-1}-\e_{l}, -1)}(v_{x-1 l}) \underset{ s \neq x}{\underset{s \neq x-1}{\underset{s\leq n}{\prod}}}\U_{(\e_{x-1}+\e_{s})}(v_{x-1 s}) \T^{w}
\end{align}
for complex numbers $v_{x-1 x}, v_{x-1}, v_{x-1 x}, v_{x-1 l},$ and $v_{x-1 s}$. Fix such complex numbers. Now we look at $\pi''(\C_{\gamma_{\mathscr{K}_{2}}})$. We have that $\mathbb{U}^{\gamma_{\mathscr{K}_{2}}*\nu}_{0}$ and $\mathbb{U}^{\gamma_{\mathscr{K}_{2}}*\nu}_{2}$ are both contained in $\U_{0}$, and the latter is generated by elements $\U_{(\e_{x-1}-\e_{x}, 0)}(a), a \in \mathbb{C}$. Out of the generators of $\mathbb{U}^{\gamma_{\K_{2}}*\nu}_{1}$, the ones that do not belong to $\U_{0}$ are $\U_{(\e_{x}, -1)}(a_{x}), \U_{(\e_{x}+\e_{s}, -1)}(a_{xs}),$ and $\U_{(\e_{x}-\e_{l}, -1)}(a_{x\bar l})$. Therefore, if 
\begin{align*}
\A = \U_{(\e_{x}, -1)}(a_{x})\U_{(\e_{x}+\e_{s}, -1)}(a_{xs})\U_{(\e_{x}-\e_{l}, -1)}(a_{x\bar l}) \in \U_{\e_{\bar{x}}},
\end{align*}
we conclude that 
\begin{align}
&\pi''(\C_{\gamma_{\mathscr{K}_{2}}*\nu}) =\\
&\U_{0}\A\U_{(\e_{x-1}-\e_{x}, 0)}(a) \T^{w} = \\
&\U_{0}\U_{(\e_{x-1}-\e_{x}, 0)}(a) \U_{(\e_{x-1}, -1)}(\xi_{x-1}) \U_{(\e_{x-1}+\e_{x}, -2)}(\xi_{x-1 x})\\ &\underset{x< l \leq n}{\prod}\U_{(\e_{x-1}-\e_{l}, -1)}(\xi_{x-1 l}) \underset{s\neq x}{\underset{s \neq x-1}{\underset{s\leq n}{\prod}}}\U_{(\e_{x-1}+\e_{s})}(\xi_{x-1 s}) \A \T^{w} = \\
\label{r2case3k2}
&\U_{0} \U_{(\e_{x-1}, -1)}(\xi_{x-1}) \U_{(\e_{x-1}+\e_{x}, -2)}(\xi_{x-1 x}) \underset{x< l \leq n}{\prod}\U_{(\e_{x-1}-\e_{l}, -1)}(\xi_{x-1 l}) \underset{s\neq x}{\underset{s \neq x-1}{\underset{s\leq n}{\prod}}}\U_{(\e_{x-1}+\e_{s})}(\xi_{x-1 s})\T^{w}
\end{align}
where 
\begin{align*}
\xi_{x-1} = c^{1,1}_{x, x-1 \bar x}(-a_{x})a\\
\xi_{x-1 x}  = c^{2,1}_{x, x-1 \bar x}(a^{2}_{x})a\\
\xi_{x-1 \bar l} = c^{1,1}_{x \bar l, x-1 \bar x}(-a_{x\bar l})a\\
\xi_{x-1 s} = c^{1,1}_{x s, x-1 \bar x}(-a_{x s})a. 
\end{align*}
Therefore it follows directly that in fact 
\begin{align*}
\pi''(\C_{\gamma_{\mathscr{K}_{2}}*\nu})  \subseteq \pi'(\C_{\gamma_{\mathscr{K}_{1}}*\nu}). 
\end{align*}
Now, the system of equations
\begin{align*}
v_{x-1} &= \xi_{x-1}\\
v_{x-1 x} &= \xi_{x-1 x}\\
v_{x-1 \bar l} &= \xi_{x-1 \bar l}\\
v_{x-1 s} &= \xi_{x-1 s}
\end{align*}
 has solutions as long as $\{v_{x-1}, v_{x-1 x}, v_{x-1 \bar l}, v_{x-1 s}\} \subset \mathbb{C}^{\times}$. For such a set of solutions we conclude 
 \begin{align*}
&\U_{(\e_{x-1},-1)}(v_{x-1})\U_{(\e_{x-1}+\e_{x}, -2)}(v_{x-1 x})
\underset{x< l \leq n}{\prod}\U_{(\e_{x-1}-\e_{l}, -1)}(v_{x-1 l}) \underset{ s \neq x}{\underset{s \neq x-1}{\underset{s\leq n}{\prod}}}\U_{(\e_{x-1}+\e_{s})}(v_{x-1 s}) =\\
&\U_{(\e_{x-1}, -1)}(\xi_{x-1}) \U_{(\e_{x-1}+\e_{x}, -2)}(\xi_{x-1 x}) \underset{x< l \leq n}{\prod}\U_{(\e_{x-1}-\e_{l}, -1)}(\xi_{x-1 l}) \underset{s\neq x}{\underset{s \neq x-1}{\underset{s\leq n}{\prod}}}\U_{(\e_{x-1}+\e_{s})}(\xi_{x-1 s})
 \end{align*}
 and therefore we conclude by Proposition \ref{stable} (applied to $\U_{(\e_{x-1}-\e_{x}, 0)}(v_{x-1 x})$ in (\ref{r2case3k1})) that a dense subset of $\pi'(\C_{\gamma_{\mathscr{K}_{1}}*\nu})$ is contained in $\pi''(\C_{\gamma_{\mathscr{K}_{2}}*\nu})$ (cf. (\ref{r2case3k1}), (\ref{r2case3k2})).\\
 
\noindent b) Let $$\mathscr{K}_{3} = \Skew(0: \mbox{\tiny{x-1}}, \mbox{\tiny{x-1}}, \ov{\mbox{\tiny{x-1}}}|0: \ov{\mbox{\tiny{x}}}, \ov{\mbox{\tiny{x}}})$$ and $$\mathscr{K}_{4} = \Skew(0: \ov{\mbox{\tiny{x}}}, \mbox{\tiny{x-1}}, \mbox{\tiny{x}}|1: \ov{\mbox{\tiny{x}}}, \ov{\mbox{\tiny{x-1}}}).$$ Then
\begin{align*}
w_{3} = \ov{x-1}\hbox{ }x-1\hbox{ }\ov{x} = w(\mathscr{K}_{3}) \\
w_{4} = x \hbox{ }\ov{x}\hbox{ }\ov{x} =  w(\mathscr{K}_{4})
\end{align*} 
By Proposition \ref{densewordreading} it is enough to show

\begin{align*}
\overline{\pi'''(\C_{\gamma_{\mathscr{K}_{3}}})} = \overline{\pi''''(\C_{\gamma_{\mathscr{K}_{4}}})}.
\end{align*}

First we claim 
\begin{align*}
\pi''''(\C_{\gamma_{\mathscr{K}_{4}}*\nu}) \subseteq \pi'''(\C_{\gamma_{\mathscr{K}_{3}}*\nu}).
\end{align*}
This is easy. Note that the terms $\U_{(\e_{x-1}-\e_{x}, -1)}(b), b \in \mathbb{C}$ generate both $\mathbb{U}^{\gamma_{\mathscr{K}_{4}}*\nu}_{1}$ and are contained in $\mathbb{U}^{\gamma_{\mathscr{K}_{3}}*\nu}_{1}$. Also, the terms $\U_{(\e_{l}-\e_{x}, 0)}$, which generate $\mathbb{U}^{\gamma_{\mathscr{K}_{4}}*\nu}_{2}$, commute with $\U_{(\e_{x-1}-\e_{x}, -1)}(b)$. Therefore 
\begin{align*}
\pi''''(\C_{\gamma_{\mathscr{K}_{4}}}) = \U_{0}\U_{(\e_{x-1}-\e_{x}, -1)}(b) \T^{w} \subseteq \pi'''(\C_{\gamma_{\mathscr{K}_{3}}}), 
\end{align*}
where the last contention follows by Proposition \ref{stable}. Now we will show 
\begin{align*}
\overline{\pi'''(\C_{\gamma_{\mathscr{K}_{3}}*\nu})} \subseteq \overline{\pi''''(\C_{\gamma_{\mathscr{K}_{4}}*\nu})}.
\end{align*}
We claim that 
\begin{align}
&\pi'''(\C_{\gamma_{\mathscr{K}_{3}}*\nu}) = \\
\label{r2final}
&\U_{0} \U_{(\e_{x-1}, -1)}(v_{x-1}) \U_{(\e_{x-1}-\e_{x}, -1)}(v_{x-1 \bar x}) \underset{ s \neq x\e_{s}+\e_{x-1} \in \Phi^{+}}{\prod}\U_{(\e_{x-1}+\e_{s}, -1)}(v_{x-1 s}) \T^{w}
\end{align}

for complex numbers $v_{x-1}, v_{x-1 \bar x},$ and $v_{x-1 s}$. Let us fix such complex numbers. Let 
\begin{align*}
D = \U_{(\e_{x}, 0)}(a_{x}) \underset{ s \neq x \e_{s}+\e_{x-1} \in \Phi^{+}}{\prod}\U_{(\e_{x}+\e_{s}, -1)}(a_{x-1 s}) \in \U_{0}.  
\end{align*}

Then by the usual arguments (note that $\U_{0}$ stabilises both the image $\pi''''(\C_{\gamma_{\mathscr{K}_{4}}})$ and the truncated image $\T^{\geq 2}_{\gamma_{\mathscr{K}_{4}}*\nu}$). 

\begin{align*}
\D^{-1} \U_{(\e_{x-1}-\e_{x}, -1)}(b) D \T^{w} \subset \pi''''(\C_{\gamma_{\mathscr{K}_{4}}})
\end{align*}
 and 
\begin{align*}
\D^{-1} \U_{(\e_{x-1}-\e_{x}, -1)}(b) D \T^{w} = \\
 \U_{(\e_{x-1}, -1)}(\rho_{x-1}) \U_{(\e_{x-1}-\e_{x}, -1)}(b) \underset{ s \neq x\e_{s}+\e_{x-1} \in \Phi^{+}}{\prod}\U_{(\e_{x-1}+\e_{s}, -1)}(\rho_{x-1 s}) \U_{(\e_{x}+\e_{x-1}, -1)}(\rho_{x x-1})
\end{align*}
where 
\begin{align*}
\rho_{x-1} = c^{1,1}_{x-1 \bar x, x}(-b)a_{x}\\
\rho_{x-1 x}  = c^{2,1}_{x-1 \bar x, x}(-b)a^{2}_{x}\\
\rho_{x-1 s} = c^{1,1}_{x-1 \bar x, x s}(-b)a_{xs}. 
\end{align*}
As usual by requiring that $v_{x-1},v_{x-1 \bar x}, v_{x-1 x},$ and $\rho_{x-1 s}$ be non-zero we may find suitable complex numbers $b, a_{x}, a_{xs}$ such that 
\begin{align*}
\U_{(\e_{x-1}, -1)}(v_{x-1}) \U_{(\e_{x-1}-\e_{x}, -1)}(v_{x-1 \bar x}) \underset{ s \neq x\e_{s}+\e_{x-1} \in \Phi^{+}}{\prod}\U_{(\e_{x-1}+\e_{s}, -1)}(v_{x-1 s}) = \\
\D^{-1} \U_{(\e_{x-1}-\e_{x}, -1)}(b) \D \T^{w}.
\end{align*}
Therefore Proposition \ref{stable} (cf. (\ref{r2final})) implies that a dense open subset of $\pi'''(\C_{\gamma_{\mathscr{K}_{3}}*\nu})$ is contained in $\pi''''(\C_{\gamma_{\mathscr{K}_{4}}*\nu})$.

\end{proof}

\begin{cl}(R3)
\label{P3}
Let $w \in \mathcal{W}_{\mathcal{C}_{n}}$ be a word and $w_{1}$ that is not the word of an LS block, and such that it has the form $w_{1} = a_{1}\cdots a_{r}z\ov{z} \ov{b_{s}} \cdots \ov{b_{1}}$,  and let $w_{2}= a_{1}\cdots a_{r}\ov{b_{s}} \cdots \ov{b_{1}}$ with $a_{1}<\cdots a_{r} <z > b_{s}> \cdots > b_{1}$. Then $\ov{\pi(\C_{\gamma_{w_{1}w}})} = \ov{\pi'(\C_{\gamma_{w_{2}w}})}.$
\end{cl}
\begin{proof}[Proof of Claim \ref{P3}]
Let $\A = \{a_{1}, \cdots, a_{r}\}$. We have
\begin{align*}
\pi(\C_{\gamma_{w_{1}w}}) = \mathbb{P}_{a_{1}}\cdots  \mathbb{P}_{a_{r}}\mathbb{P}_{z}\mathbb{P}_{\ov{z}}\mathbb{P}_{\ov{b_{s}}}\cdots \mathbb{P}_{\ov{b_{1}}}\T^{\geq r+s+2}_{\gamma_{w_{1}w}}
\end{align*}
where
\begin{align*}
\mathbb{P}_{z} =& \U_{(\e_{z},0)}(v_{z})\underset{l >z}{\prod}\U_{(\e_{z}-\e_{l}, 0)}(v_{z\bar{l}})\underset{l\notin \A}{\prod}\U_{(\e_{z}+\e_{l},0)}(v_{zl}) \underset{a_{i}\in \A}{\prod}\U_{(\e_{z}+\e_{a_{i}},1)}(v_{za_{i}}), \\
\mathbb{P}_{\ov{z}}=& \underset{a_{i}\in \A}{\prod}\U_{(\e_{a_{i}}-\e_{z},0)}(v_{a_{i}\bar{z}})
\end{align*}
and note that $\mu_{\gamma_{w_{1}}} = \mu_{\gamma_{w_2}} = \underset{i \in \I_{r}}{\sum} \e_{a_{i}} - \underset{j \in \I_{s}}{\sum} \e_{b_{j}}$.
The terms that appear in $\mathbb{P}_{z}$ all stabilise $\mu_{\gamma_{w_{1}}} $ and commute with  $\mathbb{P}_{\ov{b_{j}}}$, while the terms in $\mathbb{P}_{\ov{z}}$ all appear in $\mathbb{P}_{a_{i}}$ and commute with $\mathbb{P}_{a_{l}}$ for $l > i$. This concludes the proof of the claim with the usual arguments. 
\end{proof}
\end{proof}
\end{proof}
\section{Non-examples for non-readable galleries}
Let $n=2$ and  $\lambda = \varepsilon_{1} + \varepsilon_{2}$, and $(\Sigma_{\gamma_{\lambda}}, \pi)$ the corresponding Bott-Samelson as in (\ref{resolution}). Let $\gamma$ be the gallery corresponding to the block
$$\Skew(0:\mbox{\tiny{1}},\mbox{\tiny{$\overline{2}$}}|0:\mbox{\tiny{2}},\mbox{\tiny{$\overline{1}$}}).$$

Then points in $\pi(\C_{\gamma})$ are of the form

\begin{align*}
\U_{(\varepsilon_{1}+\varepsilon_{2},-1)}(b)[t^{0}]
\end{align*}

\noindent
for $b\in \mathbb{C}$, hence form an affine set of dimension 1. We claim that the set $\Z =  \overline{\pi(\C_{\gamma})}$ cannot be an MV cycle in $\mathcal{Z}(\mu)$ for any dominant coweight $\mu$. First note that for any $u \in \U(\mathcal{K})$ a necessary condition for $ut^{0}$ to lie in the closure $\overline{\U(\mathcal{K})t^{\nu} \cap G(\mathcal{O})t^{\mu}}$ is that $0 \leq \nu$, since it would in particular imply that $ut^{0} \in \overline{\U(\mathcal{K})t^{\nu}}$. Also note that it is necesary for $\nu \leq \mu$ in order for the set $\mathcal{Z}(\mu)_{\nu}$ not to be empty. Any MV cycle in $\mathcal{Z}(\mu)_{\nu}$ has dimension $\<\rho, \mu + \nu \rr$, and the only possibility for the latter to be equal to 1 (since $\mu +\nu$ is a sum of positive coroots) is for either $\mu = 0$ and $\nu = \alpha^{\vee}_{i}$, or $\nu = 0$ and $\mu = \alpha^{\vee}_{i}$, for some $i \in \I$, and both options are impossible: the first contradicts $\nu \leq \mu$, and the second contradicts the dominance of $\mu$. Note that $\gamma$ is not a Littelmann gallery. 

\section{Appendix}
\label{appendix}
Here we show that relation (R3) in Theorem \ref{lecouvey} is equivalent to relation $\R_{3}$ in \cite{lecouvey}, Definition 3.1. For a word $w \in \mathcal{W}_{\mathcal{C}_{n}}$ and $m \leq n$ define $\N(w, m) = |\{x \in w: x\leq m \hbox{ or } \overline{m} \leq x \}|$. Lecouvey's relation $\R_{3}$ is: ``Let w be a word that is not the word of an LS block and such that each strict subword is. Let z be the lowest unbarred letter such that the pair $(z, \overline{z})$ occurs in $w$ and $\N(w, z) = z+1$. Then $w \cong w'$, where $w'$ is the subword obtained by erasing the pair $(z, \overline{z})$ in w.'' The following Lemma is a translation between $\R_{3}$ and (R3). 
\begin{lem}
\label{translation}
Let w be a word that is not the word of an LS block and such that each strict subword is. Then $w = a_{1} \cdots a_{r}z\overline{z}\overline{b_{s}}\cdots \overline{b_{1}}$ for $a_{i}. b_{i}$ unbarred and $a_{1}< \cdots a_{r}, b_{1}< \cdots, b_{s}$.
\end{lem}

\begin{proof}
By Remark 2.2.2 in \cite{lecouvey}, w is the word of an LS block if and only if $\N(w, m)\leq m$ for all $m \leq n$. Let w be as in the statement of  Lemma \ref{translation}. Then there exists in $w$ a pair $(z,\overline{z})$ such that $\N(w, z)>z$. Let $z$ be minimal with this property. In particular $\N(w, z) = z+1$ since if $w''$ is the word obtained from $w$ by erasing $z$, then $z \geq \N(w'',z) = \N(w,z)-1$. We claim that $z$ is the largest unbarred letter to appear in $w$. If there was a larger letter $y$ then $\N(w''',z) = \N(w,z) = z+1$ where $w'''$ denotes the word obtained from $w$ by deleting $y$. This is impossible since by assumption $w'''$ is the word of an LS block. Likewise $\overline{z}$ is the smallest unbarred letter to appear in $w$. The $a_{i}'s$ and $b_{i}'s$ are then those from Definition \ref{lsblock} for the word obtained from $w$ by deleting $z, \overline{z}$ from it.
\end{proof}

\end{document}